\documentclass[10pt,reqno]{amsart}
\usepackage{amssymb,amsmath,bm,geometry,graphics,url,color}
\usepackage{amsfonts}
\usepackage{graphicx}
\usepackage{epstopdf}
\usepackage{pdfpages}
\usepackage{epsfig,multirow}
\usepackage{amscd}
\usepackage{enumerate}
\usepackage{mathrsfs}
\usepackage{xypic}
\usepackage{stmaryrd}
\usepackage{fancybox}
\usepackage{exscale,array}
\usepackage{enumitem}%
\usepackage{color,cases}
\usepackage{hyperref}

\def\pdt2{\partial_t^2}
\def\pdx2{\partial_x^2}

\def\RR{{\mathbb{R}}}

\newcommand{\bR}{{\mathbb R}}
\newcommand{\bB}{{\bf B}}
\newcommand{\bT}{{\mathbb T}}

\newcommand{\fe}{\mathrm{e}}

\def\e{{\mathrm e}}

\def\eps{\varepsilon}

\newcommand{\bv}{{\bf v}}

\newcommand{\bx}{\mathbf{x}}

\newcommand{\bE}{\mathbf{E}}

\def\forall{\qquad\hbox{ for all }\ }

\newtheorem{theo}{Theorem}[section]
\newtheorem{lem}[theo]{Lemma}

\newtheorem{rem}[theo]{Remark}
\newtheorem{defi}[theo]{Definition}

\numberwithin{equation}{section}

\newcommand{\norm}[1]{\left\Vert#1\right\Vert}
\newcommand{\normm}[1]{\left\Vert#1\right\Vert_{W_{\tau}^{1,\infty}}}
\newcommand{\normo}[1]{\left\Vert#1\right\Vert_{L_{\tau}^{\infty}}}

\newcommand{\normz}[1]{\left\Vert#1\right\Vert_{L_{t}^{\infty}(W_{\tau}^{1,\infty})}}

\newcommand{\abs}[1]{\left\vert#1\right\vert}
\title[Optimally accurate discretizations for CPD]{Semi-discretization and full-discretization with optimal accuracy for charged-particle
dynamics in a strong nonuniform magnetic field}

\author[B. Wang]{Bin Wang}\address{\hspace*{-12pt}B.~Wang: School of Mathematics and Statistics, Xi'an Jiaotong University, 710049 Xi'an, China}
\email{wangbinmaths@xjtu.edu.cn}\urladdr{http://gr.xjtu.edu.cn/web/wangbinmaths/home}

\author[Y. L. Jiang]{Yaolin Jiang}
\address{\hspace*{-12pt}Corresponding author.  Y. L.~Jiang: School of Mathematics and Statistics, Xi'an Jiaotong University, 710049 Xi'an, China}
\email{yljiang@mail.xjtu.edu.cn}
\urladdr{http://gr.xjtu.edu.cn/web/yljiang}
\begin{document}
\maketitle
\dedicatory{}

\begin{abstract}
 The aim of this paper is to formulate and analyze numerical discretizations of charged-particle
dynamics (CPD) in a strong nonuniform magnetic field. A strategy is firstly performed for the two dimensional CPD to construct the semi-discretization and full-discretization which have optimal accuracy. This accuracy is improved in the position and in the velocity when the strength of the magnetic field becomes stronger. This is a better feature than the usual so called ``uniformly accurate methods". To obtain this refined accuracy, some reformulations of the problem and two-scale exponential integrators are incorporated, and the optimal accuracy is derived from this new procedure. Then based on the strategy given for the two dimensional case, a new class of uniformly accurate methods with simple scheme is formulated for the three dimensional CPD in maximal ordering case. All the theoretical results of the accuracy are numerically illustrated by some numerical tests.
\end{abstract}

\section{Introduction}

The classical or relativistic description of the natural world is based on describing the interaction of
elements of matter via force fields. One typical  example is the system of plasmas which is
composed by charged particles   interacting via electric
and magnetic fields. This system is  of paramount importance and  applications, comprising in plasma physics, astrophysics and  magnetic fusion devices    \cite{Arnold97,Birdsall}.
Motivated by recent interest in numerical methods for the
 the plasmas \cite{VP1,Zhao,VP2, CPC,VP3, VP4,VP5,VP9,VP-filbet,VP6,VP8,lubich19,VP7,SonnendruckerBook,WZ},  this paper is devoted to charged-particle
dynamics (CPD) in a strong nonuniform magnetic field. For such system of a large number of charged particles, its behavior can be described by the Vlasov equation  \cite{VP1,Zhao,VP9}:
\begin{equation}\label{model:linear}\left\{
  \begin{split}
  &\partial_tf(t,\bx,\bv)+\bv\cdot\nabla_\bx f(t,\bx,\bv)+\left(\bE(t, \bx)+ \bv \times\frac{\bB(t,\bx)}{\eps} \right)\cdot\nabla_\bv f(t,\bx,\bv)
  =0,\\
  &\nabla_{\bx}\cdot \bE(t,\bx)= \int_{\bR^d}f(t,\bx,\bv)d \bv-n_ i,\quad  0<t\leq T,
  \\
  & f(0,\bx,\bv)=f_0(\bx,\bv),\quad \bx,\ \bv\in\bR^2\textmd{ or }\bR^3,
  \end{split}\right.
\end{equation}
where   $f(t,\bx,\bv)$  depends on the time $t$, the position $\bx$ and the
velocity $\bv$, and represents the distribution of  charged particles under the effects of  the strongly external magnetic field $\frac{\bB(t,\bx)}{\eps}$ and  the self-consistent electric-field function $\bE(t, \bx)$. Here $0<\eps \ll 1$ determines the strength of the  magnetic field,  $n_ i$ denotes the ion density of the background, and $f_0(\bx,\bv)$ is a given initial distribution.

For the numerical approximation of the Vlasov model \eqref{model:linear}, consider the Particle-In-Cell (PIC) approach (\cite{CPC,VP3,VP4,VP5,SonnendruckerBook}) :
\begin{equation}\label{dirac}
f_{p}(t,\bx,\bv)=\sum_{k=1}^{N_p}\omega_k\delta(\bx-\bx_k(t))\delta(\bv-\bv_k(t)),\quad 0<t\leq T,
\end{equation}
where $\delta$ is the Dirac delta function and $\omega_k$ is the weight.
Plugging (\ref{dirac}) into (\ref{model:linear}) gives the equation on the characteristics:
\begin{equation}\label{charact}\left\{
\begin{split}
   & \dot{\bx}_k(t)=\bv_k(t), \quad \dot{\bv}_k(t)= \bv_k(t)\times\frac{\bB(t,\bx_k(t))}{\eps}+\bE(t,\bx_k(t)),\quad 0<t\leq T, \\
   & \bx_k(0)=\bx_{k,0},\quad \ \bv_k(0)=\bv_{k,0}.
\end{split}\right.
\end{equation}
In practical computations, the Dirac delta function $\delta(\bx)$ is usually replaced by the particle shape functions \cite{SonnendruckerBook} and hence the  PIC approximation is  accomplished by a particle pusher for (\ref{charact}).

For simplicity of notations and
without loss of generality, from now on, we   are concerned with the numerical solution of  the following  CPD with the same scheme of \eqref{charact}:
\begin{equation}\label{charged-particle 3d}
    \dot{x}(t)=v(t), \quad\ \dot{v} (t)=v(t)\times \frac{B(x(t))}{\eps}+E(x(t)),\ \  0<t\leq T, \ \ \
     x(0)=x_0\in\RR^d,\ \ \ v(0)=v_0\in\RR^d,
\end{equation}
where $x(t):(0,T  ]\to \RR^d$ and $v(t):(0,T  ]\to \RR^d$ are respectively the unknown position and velocity of a charged particle with the dimension $d=2\ \textmd{or}\ 3$,   $E(x)\in \RR^d$ is a given nonuniform electric-field function,  $B(x)\in \RR^d$ is a given magnetic field and  $0<\eps \ll 1$ is a dimensionless parameter determining the strength of the  magnetic field. For the   two dimensional CPD (i.e., $d=2$),
the system
\eqref{charged-particle 3d}   can be formulated as
\begin{equation}\label{charged-particle sts}
  \dot{x}(t)=v(t),  \quad\ \dot{v} (t)=\frac{b(x(t))}{\eps}Jv(t)+E(x(t)),\ \  0<t\leq T, \ \ \
     x(0)=x_0\in\RR^2,\ \  \ v(0)=v_0\in\RR^2,
\end{equation}
with   $b(x):\RR^2\to \RR$ satisfying $\abs{b(x(t))}\geq C>0$ and $J=\left(
                                                                           \begin{array}{cc}
                                                                             0 & 1 \\
                                                                             -1 & 0 \\
                                                                           \end{array}
                                                                         \right)$. It is noted that the discretization presented in this paper can be applied to the system with $B(t,x(t))$ and $E(t,x(t))$ without any difficulty and  can be used to design optimally accurate methods for the Vlasov equation \eqref{model:linear} combined with the PIC  approach.

For the  CPD \eqref{charged-particle 3d} or (\ref{charged-particle sts}), it has a long research history in the physical literature  \cite{Arnold97,Benettin94,Cary2009,add1,Northrop63}.  Meanwhile, the modeling and simulation of CPD  is of practical interest in scientific computing.  After particle discretization of some kinetic models,     the system \eqref{charged-particle 3d} or (\ref{charged-particle sts})  is a core problem which needs to be computed via effective numerical algorithms \cite{VP1,Zhao,VP2, CPC,VP3, VP4,VP5,VP9,VP-filbet,VP6,VP8,VP7,SonnendruckerBook}.
Concerning the numerical algorithms for the  CPD \eqref{charged-particle 3d} or (\ref{charged-particle sts}),  two categories   have been in the center of research according to different regimes of magnetic field:   \emph{normal magnetic field} $\eps\approx 1$ and \emph{strong magnetic field} $0<\eps \ll 1$.

Earlier studies are devoted to the normal regime $\eps\approx 1$, comprising the well-known Boris method \cite{Boris1970} as well as some further researches on it   \cite{Hairer2017-1,lubich19,Qin2013},  volume-preserving algorithms \cite{He2015},  symmetric algorithms \cite{Hairer2017-2},  symplectic  algorithms \cite{He2017,Tao2016,Webb2014,Zhang2016},  variational integrators \cite{Hairer2018,PRL1},  splitting integrators \cite{Ostermann15} and energy-preserving algorithms  \cite{L. Brugnano2020,L. Brugnano2019,Li-ANM}.
 However, if those methods are applied to CPD with a strong magnetic field $0<\eps \ll 1$, this often adds a stringent restriction on
the time step used in the numerical algorithms.  The error constant of the methods mentioned above  is usually proportional to $1/\eps^p$ for some $p>0$, which is unacceptable for small $\eps$.

In order to handle this restriction, various novel methods with improved accuracy or  uniform  accuracy have been studied in recent years for CPD under a strong magnetic field  with $0<\eps \ll 1$.
An exponential energy-preserving integrator was formulated in \cite{Wang2020} for (\ref{charged-particle 3d}) in a strong uniform magnetic field
and uniform second order accuracy can be derived. In order to improve  the asymptotic behaviour of the  Boris method as $\eps\to0$,
two filtered Boris algorithms were developed and analysed in \cite{lubich19} under the maximal ordering scaling  \cite{scaling1,scaling2}, i.e. $B=B(\eps x)$. Some splitting methods with uniform error bounds have been proposed and studied in \cite{WZ}. Combined with the PIC discretization, some effective algorithms have been derived for the  CPD \eqref{charged-particle 3d} or (\ref{charged-particle sts}) such as  exponential integrators \cite{VP8}, asymptotic preserving schemes \cite{VP4,VP5,VP9,Chacon},  uniformly accurate schemes \cite{VP1,Zhao, CPC,VP3} and other efficient methods \cite{VP2,VP-filbet,VP6}.

Among those powerful numerical methods stated above for CPD in a strong magnetic field,  the best accuracy is $\mathcal{O}\big(\varepsilon h^r\big)$ in the $x$ and $\mathcal{O}\big(  h^r\big)$ in the $v$ with the time step size $h$ and the order $r=1,2$, which is achieved for solving the two dimensional CPD in   \cite{VP1}.
The main interest of this paper lies in a novel class of discretizations for solving the two dimensional CPD \eqref{charged-particle sts}, capable of  \textbf{having optimal accuracy
which is better than all the existing uniformly accurate algorithms}. More preciously, we  prove that \textbf{the  novel discretizations have the accuarcy
$\mathcal{O}\big(\varepsilon^2 h^{2r}\big)$ in the $x$ and $\mathcal{O}\big(\varepsilon h^{2r}\big)$ in the $v$, and thus as $\eps$ decreases, the method is more accurate in the approximation of both $x$ and $v$}.
 This optical  accuracy is   very competitive  in the  computation of  CPD in a strong nonuniform magnetic field.
 To get this refined accuracy,  some  reformulations of  the problem and   two-scale  exponential   integrators  are {incorporated
into} the formulation  of the discretization.  
 Meanwhile, based on the  strategy given  for the two dimensional case, we obtain a new kind of uniformly accurate algorithms with very simple scheme for  the three dimensional CPD  \eqref{charged-particle 3d} under the  maximal ordering scaling.

The outline of the paper is as follows. In section \ref{sec:2}, we propose the semi-discretization and rigorously prove its optimal accuracy for the two dimensional CPD. The full-discretization and  its optimal accuracy are performed in section \ref{sec:3}. In section \ref{sec:4}, some practical discretizations are constructed and the numerical tests are displayed to support  the theoretical results.  Section \ref{sec:5}
is devoted to  the application to the three dimensional CPD in maximal ordering case  and a class of uniformly accurate methods  is discussed.
Some  conclusions are drawn in section \ref{sec:con}.

 \section{Semi-discretization and optimal accuracy}\label{sec:2}
\subsection{Semi-discretization}\label{TRA}
For the two dimensional CPD \eqref{charged-particle sts},
let us first define some new variables $$\epsilon=\frac{\eps }{b_0} ,  \quad \
q(t)=x(t),\quad \  p(t)=\epsilon   v(t),$$ with $b_0=b(q(0))$. Then one gets from \eqref{charged-particle sts} that
\begin{equation}\label{H model problem}
  \dot{q}(t)=\frac{1}{\epsilon }p(t), \ \ \ \dot{p}(t)=   \frac{b(q(t))}{\epsilon b_0}Jp(t)+\epsilon  E(q(t)),\ \ \  \ q(0)=x_0,\ \     \ p(0)=\epsilon  v_0.
\end{equation}
Linearizing this system leads to
\begin{equation}\label{H3 model problem}  \dot{q}(t)=\frac{1}{\epsilon }p(t), \ \ \ \dot{p}(t)= \frac{1}{\epsilon }Jp(t)+F(q(t),p(t)),\ \ \ q(0)=x_0,\ \ \ p(0)=\epsilon  v_0,
\end{equation}
where $$F(q(t),p(t))=\frac{b(q(t))-b_0}{\epsilon b_0}Jp(t)+\epsilon  E(q(t)).$$
We notice that, by the fact that $p(t)=\epsilon   v(t)=\mathcal{O}(\epsilon)$,  the function $F(q(t),p(t))$ is bounded.

Then letting
 \begin{equation*}
 u(t)=q(t)-  t/\epsilon   \varphi_1( -t J/\epsilon )  p(t),\ \quad
 w(t)= \varphi_0( -t J /\epsilon )p(t),
\end{equation*}
with  $\varphi_0(z)=\fe^{z}$ and $\varphi_1(z)=(\e^z-1)/z$, \eqref{H3 model problem} can be reformulated as
  \begin{equation}\label{new-system}%
\left\{ \begin{aligned}
&\dot{u}(t)=-  t/\epsilon   \varphi_1( -t J /\epsilon ) F\big( u(t)+ t/\epsilon   \varphi_1( t J /\epsilon )w(t), \varphi_0( t J /\epsilon )w(t)\big),
\ \ u(0)=q(0),\\
&\dot{w}(t)= \varphi_0( -t J /\epsilon )F\big( u(t)+ t/\epsilon   \varphi_1( t J /\epsilon )w(t), \varphi_0( t J /\epsilon )w(t)\big),\ \   \ \ \ \ \ \
w(0)= p(0).
 \end{aligned}  \right.
\end{equation}
 Observing that  $J=\left(
                                                                           \begin{array}{cc}
                                                                             0 & 1 \\
                                                                             -1 & 0 \\
                                                                           \end{array}
                                                                         \right)$,  we deduce that
 $$s  \varphi_1(s J )=\left(
                           \begin{array}{cc}
                             \sin(s) & 1-\cos(s) \\
                             -1+\cos(s) & \sin(s) \\
                           \end{array}
                         \right)
 \quad \textmd{and}  \quad   \varphi_0(s J )=\left(
                           \begin{array}{cc}
                             \cos(s) & \sin(s) \\
                             -\sin(s) & \cos(s) \\
                           \end{array}
                         \right).
$$ This shows that $s  \varphi_1(s J )$ and $  \varphi_0(s J )$ are periodic in $s$ on  $[0,2\pi]$.
Therefore, the two-scale formulation (\cite{UAKG}) works for the transformed system \eqref{new-system} by considering $t/\eps$ as the fast time variable and $t$ as the slow one. With another variable $\tau$ to denote the fast time variable $t/\epsilon $,
the two-scale pattern of (\ref{new-system}) takes the form:
\begin{equation}\label{2scale}\left\{ \begin{split}
  &\partial_tX(t,\tau)+\frac{1}{\epsilon }\partial_\tau X(t,\tau)=-  \tau  \varphi_1( -\tau J ) F\big( X(t,\tau)+ \tau  \varphi_1( \tau J )V(t,\tau), \varphi_0( \tau J )V(t,\tau)\big),\\
    &\partial_tV(t,\tau)+\frac{1}{\epsilon }\partial_\tau V(t,\tau)=\varphi_0( -\tau J )F\big( X(t,\tau)+ \tau  \varphi_1( \tau J )V(t,\tau), \varphi_0( \tau J )V(t,\tau)\big),
    \end{split}\right.
\end{equation}
where $X(t,\tau)$ and $V(t,\tau)$ are periodic in $\tau$ on the
torus $\bT=[0,2\pi]$, and they satisfy  $X(t,\tau)=u(t),\ V(t,\tau)=w(t).$

The initial data for
(\ref{2scale}) is derived by using the  strategy from
\cite{VP1}, which is briefly introduced as follows.
With the notations $U(t,\tau)=[X(t,\tau);V(t,\tau)]$
and
\begin{equation}\label{ftau}f_\tau(U(t,\tau))=\left(
                      \begin{array}{c}
                      -  \tau  \varphi_1( -\tau J ) F\big( X(t,\tau)+ \tau  \varphi_1( \tau J )V(t,\tau), \varphi_0( \tau J )V(t,\tau)\big)\\
                      \varphi_0( -\tau J )F\big( X(t,\tau)+ \tau  \varphi_1( \tau J )V(t,\tau), \varphi_0( \tau J )V(t,\tau)\big)\\
                      \end{array}
                    \right),
\end{equation}
first compute $\underline{U}^{[k]}=\underline{U}^{[0]}-\epsilon B_0^{[k-1]}(\underline{U}^{[k-1]})$ with
 $\underline{U}^{[0]}=[x_0;\epsilon v_0]$. Then the $j$-th order initial data for (\ref{2scale})  is defined by \begin{equation}\label{inv}[X^0;V^0]:=U^{[j]}(\tau)=\underline{U}^{[0]}+\epsilon B_\tau^{[j]}\left(\underline{U}^{[j]}\right)
 -\epsilon B_0^{[j]}\left(\underline{U}^{[j]}\right),\end{equation}
where the result of {$B_\tau^{[k-1]}$   is derived by the recursion $B_\tau^{[0]}=0$ and
  \begin{equation*}
\begin{aligned}
 B_\tau^{[k+1]}(U)&=L^{-1}(I-\Pi)f_\tau\left(U+\epsilon B_\tau^{[k]}(U)\right)-
 \frac{L^{-1}}{\epsilon^{k-1}}\left[B_\tau^{[k]}\left(U+\epsilon ^k\Pi f_\tau\left(U+\epsilon B_\tau^{[k]}(U)\right)\right)-B_\tau^{[k]}(U)\right]. \end{aligned} \end{equation*}
In this paper,  $I$ is the identity operator, $\Pi$ denotes
 the averaging operator defined by $\Pi \vartheta:=\frac{1}{2\pi}\int_0^{2\pi}\vartheta(s)ds$ for some $\vartheta(\cdot)$ on $\bT$ and $L:=\partial_\tau $ is invertible with inverse defined by $(L^{-1} \vartheta)(\tau)=(I-\Pi)\int_0^{\tau}\vartheta(s)ds$.

We now present the novel semi-discrete scheme  of the  CPD \eqref{charged-particle sts}.
\begin{defi}\label{dIUA-PE}(\textbf{Semi-discrete scheme.})
For solving the CPD system  \eqref{charged-particle sts},  the semi-discrete scheme is defined as follows with a time step size $h$.

\textbf{Step 1.} The initial data  of \eqref{2scale} is derived from \eqref{inv} with $j=4$ and we denote it as $[X^{0};V^{0}]= U^{[4]}(\tau)$.

   \textbf{Step 2.}  For solving the two-scale system  \eqref{2scale},
  the  following $s$-stage two-scale exponential integrator is considered \begin{equation}\label{cs ei-ful2}
\begin{array}[c]{ll}%
&X^{ni}=\varphi_0(c_{i}h/\epsilon \partial_\tau)X^{n}-h\textstyle\sum\limits_{j=1}^{s}\bar{a}_{ij}(h/\epsilon \partial_\tau)\tau  \varphi_1( -\tau J )F\big( X^{nj} +\tau  \varphi_1( \tau J ) V^{nj},\varphi_0( \tau J )V^{nj}\big),\
i=1,2,\ldots,s,\\
&V^{ni}=\varphi_0(c_{i}h/\epsilon \partial_\tau)V^{n}+h\textstyle\sum\limits_{j=1}^{s}\bar{a}_{ij}(h/\epsilon \partial_\tau)\varphi_0(- \tau J )F\big( X^{nj} +\tau  \varphi_1( \tau J ) V^{nj},\varphi_0( \tau J )V^{nj}\big),\ \
i=1,2,\ldots,s,\\
&X^{n+1}=\varphi_0(h/\epsilon \partial_\tau)X^{n}- h\textstyle\sum\limits_{j=1}^{s}\bar{b}_{j}(h/\epsilon \partial_\tau)\tau  \varphi_1( -\tau J )F\big( X^{nj} +\tau  \varphi_1( \tau J ) V^{nj},\varphi_0( \tau J )V^{nj}\big),\\
&V^{n+1}=\varphi_0(h/\epsilon \partial_\tau)V^{n}+  h\textstyle\sum\limits_{j=1}^{s}\bar{b}_{j}(h/\epsilon \partial_\tau)\varphi_0(- \tau J )F\big( X^{nj} +\tau  \varphi_1( \tau J ) V^{nj},\varphi_0( \tau J )V^{nj}\big),
\end{array}\end{equation}
where $\bar{a}_{ij}(h/\epsilon \partial_\tau)$ and
$\bar{b}_{j}(h/\epsilon \partial_\tau)$ are uniformly bounded functions which will be determined in Section \ref{sec:4}.

\textbf{Step 3.}   The numerical solution  $x^{n+1}\approx x(t_{n+1})$ and $v^{n+1}\approx v(t_{n+1})$  of  \eqref{charged-particle sts} is defined by  \begin{equation*} \begin{aligned}
&x^{n+1}=  X^{n+1}+  \frac{t_{n+1}}{\epsilon }    \varphi_1(  t_{n+1} J/\epsilon )   V^{n+1},\qquad \ v^{n+1}=\frac{1}{\epsilon} \varphi_0(  t_{n+1} J/\epsilon )  V^{n+1},
 \end{aligned}
\end{equation*}
where $t_{n+1}=(n+1)h.$
\end{defi}

\subsection{Optimal accuracy}

In this present part, we derive the optimal accuracy of the semi-discrete scheme given in Definition \ref{dIUA-PE}.
For simplicity of notations,  we shall denote $C>0$   a generic constant independent of the time step $h$ or $\eps$ or $n$.
In this section, we   use the   norm $\norm{\cdot}$ of a vector
to denote the standard euclidian norm and that of a   scalar quantity  refers to the absolute value.
Meanwhile,  let $L_t^{\infty}:=L_t^{\infty}([0,T])$
and $L_{\tau}^{\infty}:=L_{\tau}^{\infty}([0,2\pi])$  denote the functional spaces in $t$
and $\tau$ variables, respectively.
 For a smooth periodic function $\vartheta(\tau)$ on $[0,2\pi]$, define its $W_{\tau}^{1,\infty}$-norm as (\cite{VP1})
 $\normm{\vartheta}:=\max\{\normo{\vartheta},\normo{\partial_{\tau}\vartheta}\},$
and
for a smooth vector field $\mathcal{V}(\tau)=\left(
                                     \begin{array}{c}
                                       \vartheta_1 \\
                                       \vartheta_2 \\
                                     \end{array}
                                   \right)
$ on $[0,2\pi]$,  introduce
 $\normm{\mathcal{V}}:=\max\{\normm{\vartheta_1},\normm{\vartheta_2}\}.$
 The main result is stated by the following theorem.
\begin{theo} \label{UA thm} (\textbf{Optimal accuracy})
\begin{table}[t!]
\centering
\begin{tabular}
[c]{|c|c|}\hline Stiff order conditions & Order\\\hline
$\psi_{1}(h/\epsilon \partial_\tau)=0$ & 1\\\hline
$\psi_{2}(h/\epsilon \partial_\tau)=0$ & 2\\
$\psi_{1,i}(h/\epsilon \partial_\tau)=0$ & 2\\\hline
$\psi_{3}(h/\epsilon \partial_\tau)=0$ & 3\\
$\sum\limits_{i=1}^{s}\bar{b}_{i}(h/\epsilon \partial_\tau)K\psi_{2,i}(h/\epsilon \partial_\tau)=0$ & 3\\\hline
$\psi_{4}(h/\epsilon \partial_\tau)=0$ & 4\\
$\sum\limits_{i=1}^{s}\bar{b}_{i}(h/\epsilon \partial_\tau)K\psi_{3,i}(h/\epsilon \partial_\tau)=0$ & 4\\
$\sum\limits_{i=1}^{s}\bar{b}_{i}(h/\epsilon \partial_\tau)K\sum\limits_{j=1}^{s}\bar{a}_{ij}(h/\epsilon \partial_\tau)A\psi_{2,i}(h/\epsilon \partial_\tau)=0$ & 4\\
$\sum\limits_{i=1}^{s}\bar{b}_{i}(h/\epsilon \partial_\tau)c_iK\psi_{2,i}(h/\epsilon \partial_\tau)=0$ & 4\\\hline
\end{tabular}
\caption{Stiff order conditions with any  bounded operators $K$ and $A$. }%
\label{tab1}%
\end{table}
  It is assumed that
 the nonlinear functions $E(x)$ and $b(x)$ are globally Lipschitz functions, i.e.,
 \begin{equation}\label{ASSU1}
\begin{aligned}
 \norm{E(x_1)-E(x_2)} \leq  C\norm{ x_1 -x_2},\qquad \  \norm{b(x_1)-b(x_2)} \leq  C\norm{x_1- x_2},  \forall x_1,\ x_2\in \RR^2.
\end{aligned}
\end{equation}
Let $\varphi_k(z)=\int_0^1
\theta^{k-1}\fe^{(1-\theta)z}/(k-1)!d\theta$ for $k=2,3,\ldots $ (\cite{Ostermann}) and
 \begin{equation*}
\begin{aligned} &  \psi_j(z)=\varphi_j(z)-\Sigma_{k=1}^s\bar{b}_k(z)\frac{c_k^{j-1}}{(j-1)!},\ \ \ \
\psi_{j,i}(z)=\varphi_j(c_iz)c_i^j-\Sigma_{k=1}^s\bar{a}_{ik}(z)\frac{c_k^{j-1}}{(j-1)!},\ \ \ i=1,2,\ldots,s.
 \end{aligned}
\end{equation*}
Choose $r=2$ or $4$. Assume that the conditions of order $r-1$  given in Table  \ref{tab1} are true and for those of order $r$, the first one has the form $\psi_{r}(0)=0$ and the others hold in a weaker form with $\bar{b}_i(0)$ instead of
$\bar{b}_i(h/\epsilon \partial_\tau)$   for $1 \leq i \leq s$. Under the conditions stated above,  the global error of the semi-discrete scheme given in Definition \ref{dIUA-PE} is \begin{equation*}
\begin{aligned}
\norm{x^n-x(t_n)} +\norm{\epsilon v^n- \epsilon v(t_n)}  \leq C  \epsilon ^2 h^r,\qquad 0\leq n\leq T/h,\\
\end{aligned}
\end{equation*}
where $C$ is independent of $n, h, \epsilon$.
\end{theo}

\begin{rem}
We note that this optimal accuracy is a better feature than the usual so called   ``uniformly accurate methods"  \cite{VP1,Zhao, CPC,VP3}.   As $\eps$ decreases,
 the accuracy is improved to be $\mathcal{O}\big(\varepsilon^2 h^r\big)$ in the $x$ and $\mathcal{O}\big(\varepsilon h^r\big)$ in the $v$, which is   very competitive  in the scientific computing of  CPD with strong nonuniform magnetic field.
\end{rem}
To derive this optimal accuracy, we first present four lemmas and then give its proof.
%

\begin{lem}\label{lem1}
Under the assumptions \eqref{ASSU1} and the condition that   $\norm{x_0}+\norm{v_0} \leq C$, the solution of  \eqref{H model problem} satisfies
 \begin{equation}\label{a-priori bound}
  \norm{q(t)}+\norm{p(t)/\epsilon } \leq C, \forall t\in(0,T].
\end{equation}
Moreover, we have
$$ \norm{b(q(t))-b(q(0))}\leq C \epsilon ,\forall t\in(0,T].$$
\end{lem}

\begin{proof}
This result can be shown in a similar way of \cite{VP1}.
With the new notation $\tilde{p}(t):=e^{-Jt/\epsilon} p (t)$, the system \eqref{H model problem} reads
\begin{equation}\label{H1 model problem} \dot{q}(t)=\frac{1}{\epsilon }\varphi_0(Jt/\epsilon)\tilde{p}(t), \ \ \
\dot{\tilde{p}}(t)= \varphi_0(-Jt/\epsilon)F(q(t),\varphi_0(Jt/\epsilon)\tilde{p}(t)),\ \ \
q(0)=x_0,\ \ \   \tilde{p}(0)=\epsilon  v_0.
\end{equation}
 We first take the inner product on both sides of \eqref{H1 model problem} with
$q(t)$ and $\tilde{p}(t)$ and then use Cauchy-Schwarz inequality to get
\begin{equation*}
\begin{aligned}
&\frac{d}{dt}\norm{q(t)}^2\leq \frac{2}{\epsilon }\norm{q(t)}\norm{\tilde{p}(t)},\qquad \frac{d}{dt}\norm{\tilde{p}(t)}^2\leq 2 \epsilon  \norm{E(q(t))} \norm{\tilde{p}(t)}.\end{aligned}\end{equation*}
These estimates can be simplified as
\begin{equation*}
\begin{aligned}
&\frac{d}{dt}\norm{q(t)} \leq \frac{2}{\epsilon } \norm{\tilde{p}(t)},\qquad
 \frac{d}{dt}\norm{\tilde{p}(t)} \leq 2 \epsilon  \norm{E(q(t))} \leq 2 \epsilon \norm{E(q(0))}  +2 \epsilon C\norm{ q(t)} +2 \epsilon C\norm{ q(0)}.\end{aligned}\end{equation*}
The bound \eqref{a-priori bound} is deduced from these two inequalities and Gronwall’s inequality.

To prove the second statement, we integrate   the first equation in \eqref{H1 model problem}
\begin{equation*}
\begin{aligned}
 q(t) &=q(0)+\int_{0}^{t}\frac{1}{\epsilon }\varphi_0(J\xi/\epsilon)\tilde{p}(\xi)d \xi =q(0)-J\big(\varphi_0(Jt/\epsilon)\tilde{p}(t)-\tilde{p}(0)\big)+J\int_{0}^{t} \varphi_0(J\xi/\epsilon)\dot{\tilde{p}}(\xi)d \xi.
 \end{aligned}\end{equation*}
By inserting the second equation of \eqref{H1 model problem},  it is obtained that
\begin{equation*}
\begin{aligned}
 q(t)
 &=q(0)-J\big(\varphi_0(Jt/\epsilon)\tilde{p}(t)-\tilde{p}(0)\big)+J\int_{0}^{t} \varphi_0(J\xi/\epsilon)\varphi_0(-J\xi/\epsilon) F(q(\xi),\varphi_0(J\xi/\epsilon)\tilde{p}(\xi)) d \xi\\
  &=q(0)-J\big(\varphi_0(Jt/\epsilon)\tilde{p}(t)-\tilde{p}(0)\big)+J\int_{0}^{t}  \Big(
  \frac{b(q(\xi))-b_0}{\epsilon b_0}J\varphi_0(J\xi/\epsilon)\tilde{p}(\xi)+\epsilon  E(q(\xi))  \Big)d \xi,
 \end{aligned}\end{equation*}
so that
 \begin{equation*}
\begin{aligned}
 \norm{q(t)-q(0)}
 & \leq\norm{\varphi_0(Jt/\epsilon)\tilde{p}(t)-\tilde{p}(0)}+ \int_{0}^{t}
  \norm{\frac{b(q(\xi))-b_0}{\epsilon b_0}J\varphi_0(J\xi/\epsilon)\tilde{p}(\xi)}d \xi+\epsilon  T\sup_{0\leq t\leq T}\norm{E(q(t)) }\\
   & \leq C  \epsilon+ C\epsilon\int_{0}^{t}
 \norm{ q(\xi)-q(0)}  d \xi+\epsilon T \sup_{0\leq t\leq T}\norm{E(q(t)) }.
 \end{aligned}\end{equation*}
According to Gronwall’s inequality, we  obtain  the estimate
$\norm{q(t)-q(0)}  \leq C\epsilon$ and this leads to  the second result of this lemma.
\end{proof}

\begin{lem}\label{LBS}  
Under the assumptions \eqref{ASSU1}, $E(x)\in C^r( \RR^2)$ and $b(x)\in C^r(\RR^2)$, the solution of  \eqref{2scale} with the initial value $U^{[r]}(\tau)$  and its derivatives w.r.t.  $t$ are bounded by
\begin{equation}\label{UBOUND}\begin{aligned} &\normz{X(t,\tau)}\leq C,\quad \quad\quad \ \  \normz{V(t,\tau)}\leq C \epsilon ,\\
&\normz{\partial_t^{k}X(t,\tau)}\leq C \epsilon,\quad\quad \normz{\partial_t^{k}V(t,\tau)}\leq C \epsilon, \end{aligned}\end{equation}
where $k=1,2,\ldots,r$ and $0\leq t\leq t_0 $ for some $t_0>0$.
\end{lem}

\begin{proof}
According to the Chapman-Enskog expansion, the solution of  \eqref{2scale} can be partitioned into two parts
\begin{equation*} \begin{split}
  &X(t,\tau)=\underline{X}(t)+\bar{\varphi}(t,\tau)\ \ \ \textmd{with}\ \ \underline{X}(t)=\Pi X(t,\tau),\ \ \Pi\bar{\varphi}(t,\tau)=0,\\
  &V(t,\tau)=\underline{V}(t)+\bar{\psi}(t,\tau)\ \ \ \ \textmd{with}\ \ \underline{V}(t)=\Pi V(t,\tau),\ \ \Pi\bar{\psi}(t,\tau)=0,
    \end{split}
\end{equation*}
where the averaging operator is defined by $\Pi X(t,\tau):=\frac{1}{2\pi}\int_0^{2\pi}X(t,\tau)d \tau$.
   These composers satisfy the differential equations
 \begin{equation*}\begin{split}
  \dot{\underline{X}}(t)&=\Pi\Big(-  \tau  \varphi_1( -\tau J ) F\big( X(t,\tau)+ \tau  \varphi_1( \tau J )V(t,\tau), \varphi_0( \tau J )V(t,\tau)\big)\Big),\\
    \dot{\underline{V}}(t)&=\Pi\Big(\varphi_0( -\tau J )F\big( X(t,\tau)+ \tau  \varphi_1( \tau J )V(t,\tau), \varphi_0( \tau J )V(t,\tau)\big)\Big),\\
      \partial_t \bar{\varphi}(t,\tau)+\frac{1}{\epsilon }\partial_\tau \bar{\varphi}(t,\tau)&=(I-\Pi)\Big(-  \tau  \varphi_1( -\tau J ) F\big( X(t,\tau)+ \tau  \varphi_1( \tau J )V(t,\tau), \varphi_0( \tau J )V(t,\tau)\big)\Big),\\
    \partial_t \bar{\psi}(t,\tau)+\frac{1}{\epsilon }\partial_\tau \bar{\psi}(t,\tau) &=(I-\Pi)\Big(\varphi_0( -\tau J )F\big( X(t,\tau)+ \tau  \varphi_1( \tau J )V(t,\tau), \varphi_0( \tau J )V(t,\tau)\big)\Big).
    \end{split}
\end{equation*}
For this system, we first derive the bounds of its solution, which is accomplished by the initial value  \eqref{inv}.  The first order initial value is
  \begin{equation*}\begin{split}
  U^{[1]}(\tau)=&\underline{U}^{[0]}+\epsilon B_\tau^{[1]}\left(\underline{U}^{[1]}\right)
 -\epsilon B_0^{[1]}\left(\underline{U}^{[1]}\right)+O(\epsilon ^2)\\
 =&\underline{U}^{[0]}+\epsilon B_\tau^{[1]}\left(\underline{U}^{[0]}-\epsilon B_0^{[0]}(\underline{U}^{[0]})\right)
 -\epsilon B_0^{[1]}\left(\underline{U}^{[0]}-\epsilon B_0^{[0]}(\underline{U}^{[0]})\right)+O(\epsilon ^2)\\
  =&\underline{U}^{[0]}+\epsilon B_\tau^{[1]}\left(\underline{U}^{[0]}\right)
 -\epsilon B_0^{[1]}\left(\underline{U}^{[0]}\right)+O(\epsilon ^2)\\
   =&\underline{U}^{[0]}+\epsilon L^{-1}(I-\Pi)f_\tau\left(\underline{U}^{[0]}\right)
 -\epsilon L^{-1}(I-\Pi)f_0\left(\underline{U}^{[0]}\right)+O(\epsilon ^2).  \end{split}
\end{equation*}
Noticing that $ \Pi$ and  $L^{-1}(I-\Pi)$ are bounded operators, one deduces that $U^{[1]}(\tau)$ is uniformly bounded w.r.t. $\epsilon$. This procedure can be continued and the uniform bound (w.r.t. $\epsilon$) can be obtained for $U^{[k]}(\tau)$ with $k=2,3,\ldots,r.$
Using the same strategy as in the proof of  Lemma \ref{lem1}, we get that
for the
two-scale problem \eqref{2scale}, with the initial value $U^{[k]}(\tau)$  for any $k=0,1,\ldots,$
 $X(t,\tau)=O(1)$, $V(t,\tau)=O(\epsilon)$ and  $X(t,\tau)-X(0,\tau)=O(\epsilon)$.
Combining the fact that $b\big(X(t,\tau)+ \tau  \varphi_1( \tau J )V(t,\tau)\big)-b_0=O(\epsilon)$ leads to
\begin{equation*} \begin{split}
  &F\big( X(t,\tau)+ \tau  \varphi_1( \tau J )V(t,\tau), \varphi_0( \tau J )V(t,\tau)\big)\\
  =&\frac{b\big(X(t,\tau)+ \tau  \varphi_1( \tau J )V(t,\tau)\big)-b_0}{\epsilon b_0}J\varphi_0( \tau J )V(t,\tau)+\epsilon  E\big(X(t,\tau)+ \tau  \varphi_1( \tau J )V(t,\tau)\big)=O(\epsilon ).
    \end{split}
\end{equation*}
Then for the first derivatives  $\partial_t X(t,\tau)$ and $\partial_t V(t,\tau)$, they satisfy the equation
 \begin{equation}\label{fde}\begin{split}\partial_t (\partial_t X(t,\tau))+\frac{1}{\epsilon }
 \partial_\tau (\partial_t X(t,\tau))
 =& -  \tau  \varphi_1( -\tau J ) \Big(F_1
 \big(\partial_t X(t,\tau)+ \tau  \varphi_1( \tau J )\partial_t V(t,\tau)\big)+F_2
\varphi_0( \tau J )\partial_t V(t,\tau) \Big),\\
  \partial_t (\partial_t V(t,\tau))+\frac{1}{\epsilon }
 \partial_\tau (\partial_t V(t,\tau))=&\varphi_0( -\tau J )  \Big(F_1
 \big(\partial_t X(t,\tau)+ \tau  \varphi_1( \tau J )\partial_t V(t,\tau)\big)+F_2
\varphi_0( \tau J )\partial_t V(t,\tau) \Big),\\   \end{split}
\end{equation}
where $F_j=\partial_{U_j} F\big( U_1, U_2\big)\mid_{U_1=X(t,\tau)+ \tau  \varphi_1( \tau J )V(t,\tau),U_2=\varphi_0( \tau J )V(t,\tau)} $ for $j=1,2$. The initial value of \eqref{fde} is given by
 \begin{equation*}\begin{split}
[\partial_t X(t,\tau)|_{t=0};\partial_t V(t,\tau)|_{t=0}]&=-\frac{1}{\epsilon }\partial_\tau [X^0;V^0] +\left(
                                                                                                          \begin{array}{c}
                                                                                                            -  \tau  \varphi_1( -\tau J ) F\big( X^0+ \tau  \varphi_1( \tau J )V^0, \varphi_0( \tau J )V^0\big) \\
                                                                                                            \varphi_0( -\tau J )F\big( X^0+ \tau  \varphi_1( \tau J )V^0, \varphi_0( \tau J )V^0\big)\\
                                                                                                          \end{array}
                                                                                                        \right)\\
&=-\frac{1}{\epsilon }\partial_\tau\Big(\epsilon L^{-1}(I-\Pi)f_\tau\left(\underline{U}^{[0]}\right)
 -\epsilon L^{-1}(I-\Pi)f_0\left(\underline{U}^{[0]}\right)\Big)+O(\epsilon )\\
 &= (I-\Pi) \Big(f_0\left(\underline{U}^{[0]}\right)-f_\tau\left(\underline{U}^{[0]}\right)
 \Big)+O(\epsilon )=O(\epsilon ).
    \end{split}
\end{equation*}
 This initial value and the fact that $F_1=O(\epsilon ),  F_2=O(\epsilon )$ yield $$\norm{\partial_t X(t,\tau)}_{L_t^{\infty}(L_{\tau}^{\infty})}\leq C \epsilon,\quad \quad\ \  \norm{\partial_t V(t,\tau)}_{L_t^{\infty}(L_{\tau}^{\infty})}\leq C \epsilon.$$
 In an analogous way, we can derive the bounds of   the $j$th derivatives  $\partial^j_t X(t,\tau)$ and $\partial^j_t V(t,\tau)$ with $j=2,3,\ldots,r$.

 By differentiating the above system with respect
to   $\tau$ and in a similar way, we obtain the bounds   with $W_{\tau}^{1,\infty}$
estimates.
 The proof of this lemma is complete.
\end{proof}
\begin{lem} \label{stiff thm} (\textbf{Boundedness of the schemes})
If $(x_0,v_0)$  is uniformly bounded, there exists a sufficiently small $0 < h_0 \leq 1$ such that when $h \leq h_0$,  we have the following bounds for the integrator \eqref{cs ei-ful2} with $i=1,2,\ldots,s$
\begin{equation*}
\begin{aligned} & \normo{X^{ni}}\leq C ,\ \ \  \normo {V^{ni}/ \epsilon} \leq C,  \ \ \
  \normo {X^{n+1}} \leq C  ,\ \ \    \normo {V^{n+1}/ \epsilon} \leq C, \ \ \    0\leq n< T/h.\end{aligned} 
\end{equation*}
\end{lem}
\begin{proof}For all $\epsilon$ and $h$, it is true that $\normo{\varphi_0(h/\epsilon \partial_\tau)}=1$, and
$\normo{\varphi_j(h/\epsilon \partial_\tau)}$ for  $j=1,2,\dots,$ are   uniformly bounded. According to the fact that
the coefficients of exponential integrators are composed of  $\varphi$-functions, we have
$\normo{\bar{a}_{ij}(h/\epsilon \partial_\tau)} \leq
C $ and $\normo{\bar{b}_{j}(h/\epsilon \partial_\tau)} \leq C$ for $i,j=1,2,\dots,s,$ where the constant $C$ is
independent of $h, \epsilon $.

We first prove the boundedness for a single time step of explicit methods.
Assume that the numerical solution
  satisfies $\normo{X^{n}} \leq C$ and $ \normo{ {V^{n}}/ \epsilon}\leq C$, then we have the estimates  for $n+1$:
  \begin{equation*}
\begin{array}[c]{ll}%
\normo{X^{ni}}& \leq\normo{X^{n}}+hC\textstyle\sum\limits_{j=1}^{i-1} \normo{F\big( X^{nj} +\tau  \varphi_1( \tau J ) V^{nj},\varphi_0( \tau J )V^{nj}\big) } ,\quad\
i=1,2,\ldots,s,\\
\normo{V^{ni}}& \leq\normo{V^{n}}+hC\textstyle\sum\limits_{j=1}^{i-1} \normo{F\big( X^{nj} +\tau  \varphi_1( \tau J ) V^{nj},\varphi_0( \tau J )V^{nj}\big)},\quad\
i=1,2,\ldots,s,\\
\normo{X^{n+1}}& \leq\normo{X^{n}}+ hC\textstyle\sum\limits_{j=1}^{s}\normo{F\big( X^{nj} +\tau  \varphi_1( \tau J ) V^{nj},\varphi_0( \tau J )V^{nj}\big)},\\
\normo{V^{n+1}}& \leq\normo{V^{n}}+  hC\textstyle\sum\limits_{j=1}^{s}\normo{F\big( X^{nj} +\tau  \varphi_1( \tau J ) V^{nj},\varphi_0( \tau J )V^{nj}\big)}.
\end{array}\end{equation*}
Based on the bounds of $F$ and $X^{n},V^{n}/ \epsilon$, the boundedness of $X^{n+1},V^{n+1}/ \epsilon$ is immediately obtained.

Then for a single time step of implicit methods,  iterative solutions
are {needed.} In this paper  we consider the fixed point iterative
pattern:
\begin{equation*}
\begin{aligned}
\big( {X^{ni}}\big)^{[0]}&=\varphi_0(c_{i}h/\epsilon \partial_\tau){X^{n}}-h\textstyle\sum\limits_{j=1}^{s}\bar{a}_{ij}(h/\epsilon \partial_\tau)\tau  \varphi_1( -\tau J )F\big( X^{n} +\tau  \varphi_1( \tau J ) V^{n},\varphi_0( \tau J )V^{n}\big),\\
\big( {V^{ni}}\big)^{[0]}&=\varphi_0(c_{i}h/\epsilon \partial_\tau){V^{n}}+h\textstyle\sum\limits_{j=1}^{s}\bar{a}_{ij}(h/\epsilon \partial_\tau)\varphi_0(- \tau J )F\big( X^{n} +\tau  \varphi_1( \tau J ) V^{n},\varphi_0( \tau J )V^{n}\big),\\
\big( {X^{ni}}\big)^{[m+1]}&=\varphi_0(h/\epsilon \partial_\tau) {X^{n}}- h\textstyle\sum\limits_{j=1}^{s}\bar{b}_{j}(h/\epsilon \partial_\tau)\tau  \varphi_1( -\tau J )F\big(\big( {X^{nj}}\big)^{[m]} +\tau  \varphi_1( \tau J )  \big( {V^{nj}}\big)^{[m]},\varphi_0( \tau J ) \big( {V^{nj}}\big)^{[m]}\big), \\
\big( {V^{ni}}\big)^{[m+1]}&=\varphi_0(h/\epsilon \partial_\tau) {V^{n}}+  h\textstyle\sum\limits_{j=1}^{s}\bar{b}_{j}(h/\epsilon \partial_\tau)\varphi_0(- \tau J )F\big(\big( {X^{nj}}\big)^{[m]} +\tau  \varphi_1( \tau J ) \big( {V^{nj}}\big)^{[m]},\varphi_0( \tau J ) \big( {V^{nj}}\big)^{[m]}\big),\\
&\qquad\qquad \qquad \qquad \qquad \qquad \qquad \qquad \qquad \qquad \qquad \qquad \qquad \qquad \qquad  m=0,1,\ldots.
\end{aligned}
\end{equation*}
With the
boundedness of $X^{n},V^{n}/ \epsilon$, the coefficients and the nonlinear function $F$, it is easy to derive the boundedness of $\big( {X^{ni}}\big)^{[m+1]},\ \big( {V^{ni}}\big)^{[m+1]}/ \epsilon$ and then of $X^{n+1},\ V^{n+1}/ \epsilon$.

Finally, considering the above results and using the mathematical induction,   the boundedness of explicit or implicit numerical solutions  \eqref{cs ei-ful2} as shown in this lemma over a long time interval is arrived.
\end{proof}

\begin{lem} \label{stiff thm} (\textbf{Remainders})
Inserting the exact solution of \eqref{2scale}
into the numerical approximation \eqref{cs ei-ful2}, we get
\begin{equation}\label{remainders}
\begin{array}[c]{ll}%
X(t_n+c_ih,\tau)&=\varphi_0(c_{i}h/\epsilon \partial_\tau)X(t_n,\tau)-h\textstyle\sum\limits_{j=1}^{s}\bar{a}_{ij}(h/\epsilon \partial_\tau)\tau  \varphi_1( -\tau J )G(X(t_n+c_jh,\tau), V(t_n+c_jh,\tau))+{\Delta_X^{ni}},\\
V(t_n+c_ih,\tau)&=\varphi_0(c_{i}h/\epsilon \partial_\tau)V(t_n,\tau)+h\textstyle\sum\limits_{j=1}^{s}\bar{a}_{ij}(h/\epsilon \partial_\tau)\varphi_0(- \tau J )G(X(t_n+c_jh,\tau), V(t_n+c_jh,\tau))+{\Delta_V^{ni}},\\
X(t_n+h,\tau)&=\varphi_0(h/\epsilon \partial_\tau)X(t_n,\tau)- h\textstyle\sum\limits_{j=1}^{s}\bar{b}_{j}(h/\epsilon \partial_\tau)\tau  \varphi_1( -\tau J )G(X(t_n+c_jh,\tau), V(t_n+c_jh,\tau))+ {\delta_X^{n+1}},\\
V(t_n+h,\tau)&=\varphi_0(h/\epsilon \partial_\tau)V(t_n,\tau)+  h\textstyle\sum\limits_{j=1}^{s}\bar{b}_{j}(h/\epsilon \partial_\tau)\varphi_0(- \tau J )G(X(t_n+c_jh,\tau), V(t_n+c_jh,\tau))+ {\delta_V^{n+1}},
\end{array}
\end{equation}
 where $ {\Delta_X^{ni}}, \ {\Delta_V^{ni}},\ {\delta_X^{n+1}},\ {\delta_V^{n+1}}$ are the remainders and $G( {X}, {V}):=
F\big( X  +\tau  \varphi_1( \tau J )  V ,\varphi_0( \tau J ) V \big)$.
 Under the conditions  of Theorem \ref{UA thm} and the
local assumptions of ${X^{n}}= {X(t_n,\tau)},\  {V^{n}}= {V(t_n,\tau)}$,
  the remainders are bounded   for $i=1,2,\ldots,s$ and $0\leq n< T/h$
\begin{equation*}
\begin{aligned} & \normm{ \Delta_X^{ni}}\leq C \epsilon ^2 h^{r},\quad   \normm{ \Delta_V^{ni}}\leq C \epsilon ^2
h^{r},  \quad  \normm{ \delta_X^{n+1}}\leq C \epsilon ^2 h^{r+1},\quad  \normm{ \delta_V^{n+1}}\leq C \epsilon ^2
h^{r+1}. \end{aligned} 
\end{equation*}

\end{lem}
\begin{proof}
  Since the variable $\tau$ plays essentially no role in subsequent computations of the proof, we shall omit it for brevity.
%
From the Duhamel principle, it is clear that
\begin{equation*}
\begin{array}[c]{ll}%
X(t_n+c_ih)&=\varphi_0(c_{i}h/\epsilon \partial_\tau)X(t_n)-\int_0^{c_ih}\fe^{(\theta-c_ih)\frac{\partial_\tau}{\epsilon }}
\tau  \varphi_1( -\tau J )G(X(t_n+\theta), V(t_n+\theta))d\theta,\\
V(t_n+c_ih)&=\varphi_0(c_{i}h/\epsilon \partial_\tau)V(t_n)+\int_0^{c_ih}\fe^{(\theta-c_ih)\frac{\partial_\tau}{\epsilon }}
\varphi_0(- \tau J )G(X(t_n+\theta), V(t_n+\theta))d\theta,\\
X(t_n+h)&=\varphi_0(h/\epsilon \partial_\tau)X(t_n)-  \int_0^{h}\fe^{(\theta-h)\frac{\partial_\tau}{\epsilon }}
\tau  \varphi_1( -\tau J )G(X(t_n+\theta), V(t_n+\theta))d\theta,\\
V(t_n+h)&=\varphi_0(h/\epsilon \partial_\tau)V(t_n)+\int_0^{h}\fe^{(\theta-h)\frac{\partial_\tau}{\epsilon }}
\varphi_0(- \tau J )G(X(t_n+\theta), V(t_n+\theta))d\theta.
\end{array}
\end{equation*}
Subtracting this expression from \eqref{remainders} gives the equations of  remainders
 \begin{equation*}
\begin{aligned}
&[\Delta_X^{ni};\Delta_V^{ni}]=[X(t_n+c_ih);V(t_n+c_ih)]-\Big(\varphi_0(c_{i}h/\epsilon \partial_\tau) \otimes \textmd{diag}(1,1) \Big)
[X(t_n);V(t_n)] \\
&\qquad\qquad\qquad+h\textstyle\sum\limits_{j=1}^{s} \Big(\bar{a}_{ij}(h/\epsilon \partial_\tau) \otimes \textmd{diag}(1,1) \Big)G(t_n+c_jh),\\
&[\delta_X^{n+1}; \delta_V^{n+1}]=[X(t_n+h);V(t_n+h)]-\Big(\varphi_0(h/\epsilon \partial_\tau) \otimes \textmd{diag}(1,1) \Big)
[X(t_n);V(t_n)] \\
&\qquad\qquad\qquad+h\textstyle\sum\limits_{j=1}^{s} \Big(\bar{b}_{j}(h/\epsilon \partial_\tau) \otimes \textmd{diag}(1,1) \Big)G(t_n+c_jh),\\
\end{aligned}
\end{equation*}
where we use the notation
$ {G(t)}:=[-\tau  \varphi_1( -\tau J )
G( {X(t)}, {V(t)});\varphi_0(- \tau J ) G( {X(t)}, {V(t)})]$ and the Kronecker product $\otimes$.
Applying   Taylor expansions, one gets
 \begin{equation*}
\begin{aligned}\   [\delta_X^{n+1}; \delta_V^{n+1}] =& h \epsilon  \int_{0}^1 \Big(\varphi_0((1-\xi)c_ih/\epsilon \partial_\tau) \otimes \textmd{diag}(1,1) \Big)\sum\limits_{j=0}^{\infty}\frac{(\xi c_i h)^j}{j!}\frac{\textmd{d}^j}{\textmd{d} t^j}  {G(t_n) } {\rm d}\xi \\&- h \epsilon  \sum\limits_{k=1}^{s} \Big(\bar{b}_{k}(h/\epsilon \partial_\tau) \otimes \textmd{diag}(1,1) \Big)\sum\limits_{j=0}^{\infty}\frac{ c_k^jh^j}{j!}\frac{\textmd{d}^j}{\textmd{d} t^j} {G(t_n)}\\
=&    h \epsilon  \sum\limits_{j=0}^{\infty}h^j \Big(\psi_j(h/\epsilon \partial_\tau) \otimes \textmd{diag}(1,1) \Big)\frac{\textmd{d}^j}{\textmd{d} t^j} {G(t_n)}.
\end{aligned}
\end{equation*}
Following the analysis of  \cite{Ostermann06},
the bounds of  $ \delta_X^{n+1}$  and $ \delta_V^{n+1}$ given in this lemma are deduced from the stiff order conditions presented in Theorem \ref{UA thm} and the bound \eqref{UBOUND}.  In an analogous way, we proceed for the bound of
$ \Delta_X^{ni}$  and $ \Delta_V^{ni}$, and the
proof of Lemma  \ref{stiff thm} is complete.
%
%
\end{proof}

\textbf{Proof of Theorem \ref{UA thm}.}\begin{proof}
Based on the above preparations,
 we are ready to  derive the error bounds of Theorem \ref{UA thm}.
Define the error functions by
\[
\begin{aligned}
&e_X^n:=X(t_n,\tau)- X^n,\  e_V^n:=V(t_n,\tau)- V^n, \  E_X^{ni}:=X(t_n+c_ih,\tau)-  X^{ni},\  E_V^{ni}:=V(t_n+c_ih,\tau)-  V^{ni}.\end{aligned}
\]
Subtracting the scheme of the method \eqref{cs ei-ful2} from\eqref{remainders} gives the error recursions
 \begin{equation*}
\begin{array}[c]{ll}%
E_X^{ni}&=\varphi_0(c_{i}h/\epsilon \partial_\tau)e_X^n+h\textstyle\sum\limits_{j=1}^{s}\bar{a}_{ij}(h/\epsilon \partial_\tau) \delta G^{nj}+{\Delta_X^{ni}},\\
E_V^{ni}&=\varphi_0(c_{i}h/\epsilon \partial_\tau)e_V^n+h\textstyle\sum\limits_{j=1}^{s}\bar{a}_{ij}(h/\epsilon \partial_\tau) \delta G^{nj}+{\Delta_V^{ni}},\\
e_X^{n+1}&=\varphi_0(h/\epsilon \partial_\tau)e_X^n+ h\textstyle\sum\limits_{j=1}^{s}\bar{b}_{j}(h/\epsilon \partial_\tau) \delta G^{nj}+ {\delta_X^{n+1}},\\
e_V^{n+1}&=\varphi_0(h/\epsilon \partial_\tau)e_V^n+  h\textstyle\sum\limits_{j=1}^{s}\bar{b}_{j}(h/\epsilon \partial_\tau) \delta G^{nj}+ {\delta_V^{n+1}},
\end{array}
\end{equation*}
where $\delta G^{nj}=G(X(t_n+c_jh,\tau), V(t_n+c_jh,\tau))-G(X^{nj}, V^{nj})$. This contributes with
\begin{equation*}
\begin{aligned}
& \normo{ e_{X}^{n+1}} \leq  \normo{e_{X}^{n}}+h  C
\sum\limits_{j=1}^{s}\normo{\delta G^{nj}} + C\epsilon ^2 h^{r+1},\\
& \normo{ e_{V}^{n+1}} \leq  \normo{e_{V}^{n}}+h  C
\sum\limits_{j=1}^{s}\normo{\delta G^{nj}} + C\epsilon ^2 h^{r+1}.
\end{aligned}\end{equation*}
It stems from $F$  that
$ \normo{\delta G^{nj}} \leq C\Big(\epsilon \normo{E_{X}^{nj}
}+\normo{E_{V}^{nj}}\Big), $
and based on which, it is further  arrived
\begin{equation}\label{convergence 3}
\begin{aligned}
& \normo{e_{X}^{n+1}} \leq  \normo{e_{X}^{n}}+h C
\sum\limits_{j=1}^{s}\Big(\epsilon \normo{E_{X}^{nj}
}+\normo{E_{V}^{nj}}\Big) + C\epsilon ^2 h^{r+1},\\
& \normo{e_{V}^{n+1}} \leq  \normo{e_{V}^{n}}+h C
\sum\limits_{j=1}^{s}\Big(\epsilon \normo{E_{X}^{nj}
}+\normo{E_{V}^{nj}}\Big) + C\epsilon ^2 h^{r+1}.
\end{aligned}\end{equation}
Similar result can
be obtained for $E_{X}^{ni}$ and $E_{V}^{ni}$ in an analogous way as follows: \begin{equation*}
\begin{aligned}
& \normo{E_{X}^{ni}} \leq  \normo{e_{X}^{n}}+h C
\sum\limits_{j=1}^{s}\Big(\epsilon \normo{E_{X}^{nj}
}+\normo{E_{V}^{nj}}\Big) + C\epsilon ^2 h^{r},\\
& \normo{E_{V}^{ni}} \leq  \normo{e_{V}^{n}}+h  C
\sum\limits_{j=1}^{s}\Big(\epsilon \normo{E_{X}^{nj}
}+\normo{E_{V}^{nj}}\Big) + C\epsilon ^2 h^{r}.
\end{aligned}\end{equation*}
  Based on the above results, one gets \begin{equation*}
\begin{aligned}
& \sum\limits_{i=1}^{s}\Big(\normo{E_{X}^{ni}} +\normo{E_{V}^{ni}}\Big) \leq   s\normo{e_{X}^{n}}+s\normo{e_{V}^{n}}+2s h C
\sum\limits_{j=1}^{s}\Big(\epsilon \normo{E_{X}^{nj}
}+\normo{E_{V}^{nj}}\Big) + C\epsilon ^2 h^{r}.
\end{aligned}\end{equation*}
 If the
stepsize $h$ satisfies $h\epsilon  \leq \frac{1}{4sC}$,  it is straightforward to show that $$
\sum\limits_{j=1}^{s}\Big(\normo{E_{X}^{nj}
}+\normo{E_{V}^{nj}}\Big)  \leq 2s\normo{e_{X}^{n}}+2s\normo{e_{V}^{n}}+C\epsilon ^2 h^{r}. $$ Inserting this into
\eqref{convergence 3} and using Gronwall inequality eventually leads to
\begin{equation*}
\begin{aligned} & \normo{e_{X}^{n+1}}\leq C\epsilon ^2 h^{r},\qquad  \ \ \ \normo{e_{V}^{n+1}}\leq C\epsilon ^2 h^{r}. \end{aligned}
\end{equation*}
This bound and the transformations of Section \ref{TRA} immediately complete the   proof  of Theorem \ref{UA thm}.
\end{proof}

 \section{Full-discretization and optimal accuracy}\label{sec:3}
  In the present section, by using the Fourier
pseudospectral method in $\tau$, we first  present the full-discretization for solving the CPD   \eqref{charged-particle sts} and then derive its optimal accuracy.

\subsection{Full-discretization}


To use the Fourier  method in the variable $\tau$, we introduce $\tau_l=\frac{2\pi}{N_\tau}l$ with an even positive integer $N_\tau$ and $l\in \mathcal{M}: =\{-N_\tau/2,-N_\tau/2+1,\ldots,N_\tau/2\}$.
 Then the
Fourier spectral method is proposed by considering the trigonometric
polynomials
$$X^{\mathcal{T}}(t,\tau)=\big(X_j^{\mathcal{T}}(t,\tau)\big)_{j=1,2},\qquad \
V^{\mathcal{T}}(t,\tau)=\big(V_j^{\mathcal{T}}(t,\tau)\big)_{j=1,2},$$
with
\begin{equation*}
\begin{array}[c]{ll}
 X_j^{\mathcal{T}}(t,\tau)=\sum\limits_{k\in \mathcal{T}}
 \widehat{X_{k,j}}(t)\mathrm{e}^{\mathrm{i} k \tau },\quad \ \  V_j^{\mathcal{T}}(t,\tau)=\sum\limits_{k\in \mathcal{T}}
 \widehat{V_{k,j}}(t) \mathrm{e}^{\mathrm{i} k \tau },\ \  (t,\tau)\in[0,T]\times
[-\pi,\pi],
\end{array}
\end{equation*}
such that
\begin{equation*}
\left\{
\begin{aligned}
  &\partial_t X^{\mathcal{T}}(t,\tau)+\frac{1}{\eps}\partial_\tau X^{\mathcal{T}}(t,\tau)
  =-  \tau  \varphi_1( -\tau J )F\big( X^{\mathcal{T}}(t,\tau)+ \tau  \varphi_1( \tau J )V^{\mathcal{T}}(t,\tau), \varphi_0( \tau J )V^{\mathcal{T}}(t,\tau)\big),\\
    &\partial_t V^{\mathcal{T}}(t,\tau)+\frac{1}{\eps}\partial_\tau V^{\mathcal{T}}(t,\tau)
    =\varphi_0( -\tau J )F\big( X^{\mathcal{T}}(t,\tau)+ \tau  \varphi_1( \tau J )V^{\mathcal{T}}(t,\tau), \varphi_0( \tau J )V^{\mathcal{T}}(t,\tau)\big),
\end{aligned}\right.
\end{equation*}
where $ \widehat{X_{k,j}}(t)$ and  $ \widehat{V_{k,j}}(t)$  are referred to the discrete Fourier coefficients of $ X_j^{\mathcal{M}}$ and $ V_j^{\mathcal{M}}$, respectively.
Collecting all the coefficients
in $D:=2 (N_{\tau}+1)$ dimensional
coefficient vectors $
\widehat{\mathbf{X}}(t) = (\widehat{X_{k,j}}(t)),\
\widehat{\mathbf{V}}(t) = (\widehat{V_{k,j}} (t))$ implies a system of ordinary differential equations (ODEs)
\begin{equation}\label{2scale Fourier-f}\left\{ \begin{aligned}
&\frac{d}{dt}\widehat{\mathbf{X}}(t)=\mathrm{i}\Omega\widehat{\mathbf{X}}(t)
+\mathcal{F} \Big(\mathbf{S}^{-}F\big(\mathcal{F}^{-1}\widehat{\mathbf{X}}(t)+ \mathbf{S}^{+}\mathcal{F}^{-1}\widehat{\mathbf{V}}(t), \mathbf{C}^{+}\mathcal{F}^{-1}\widehat{\mathbf{V}}(t)\big) \Big),\\
&\frac{d}{dt}\widehat{\mathbf{V}}(t)=\mathrm{i}\Omega\widehat{\mathbf{V}}(t)
+ \mathcal{F}\Big(\mathbf{C}^{-}F\big(\mathcal{F}^{-1}\widehat{\mathbf{X}}(t)+ \mathbf{S}^{+}\mathcal{F}^{-1}\widehat{\mathbf{V}}(t), \mathbf{C}^{+}\mathcal{F}^{-1}\widehat{\mathbf{V}}(t)\big) \Big),
 \end{aligned}\right.
\end{equation}
where  $\mathcal{F}$
 is the discrete Fast Fourier Transform (FFT),
$\mathbf{S}^{\pm}=\textmd{diag}\big(\pm\tau_l  \varphi_1(  \pm \tau_l J )\big)_{l=0,1,\ldots,N_\tau},\ \
\mathbf{C}^{\pm}=\textmd{diag}\big(\varphi_0(\pm\tau_l  J )\big)_{l=0,1,\ldots,N_\tau}$,  and $\Omega=\textmd{diag}(\Omega_1,\Omega_2)$ with
$\Omega_1=\Omega_2:=\frac{1}{\eps}\textmd{diag}\big(\frac{N_{\tau}}{2},
  \frac{N_{\tau}}{2}-1,\ldots,-\frac{N_{\tau}}{2}\big).$

The full-discretization of \eqref{charged-particle sts}  is stated as follows.
\begin{defi}\label{dIUA-PE-F}(\textbf{Fully discrete scheme.})
The initial data  of \eqref{2scale} is derived from \eqref{inv} with $j=4$ and we denote it as $[X^{0};V^{0}]= U^{[4]}(\tau)$.
Choose a time step size $h$ and a positive even number $N_{\tau}$.  The full-discretization of \eqref{charged-particle sts} is defined as follows.
\begin{itemize}
  \item The first step is to 
  compute the initial value of \eqref{2scale Fourier-f} by
  $[\widehat{X^{0}};\widehat{V^{0}}]=[\mathcal{F}( X^{0});\mathcal{F}(V^{0})]$.

   \item  For solving the ODEs  \eqref{2scale Fourier-f} with the initial value $[\widehat{X^{0}};\widehat{V^{0}}]$,
we consider the same $s$-stage two-scale exponential integrator as semi-discretization, that is for $M=\mathrm{i}\Omega$  \begin{equation*}
\begin{array}[c]{ll}%
\widehat{X^{ni}}&=\varphi_0(c_{i}hM)\widehat{X^{n}}+h\textstyle\sum\limits_{j=1}^{s}\bar{a}_{ij}(hM)\mathcal{F} \Big(\mathbf{S}^{-}F\big(\mathcal{F}^{-1}\widehat{X^{nj}} + \mathbf{S}^{+}\mathcal{F}^{-1}\widehat{V^{nj}}, \mathbf{C}^{+}\mathcal{F}^{-1}\widehat{V^{nj}}\big) \Big),\
i=1,2,\ldots,s,\\
\widehat{V^{ni}}&=\varphi_0(c_{i}hM)\widehat{V^{n}}+  h\textstyle\sum\limits_{j=1}^{s}\bar{a}_{ij}(hM)\mathcal{F} \Big(\mathbf{C}^{-}F\big(\mathcal{F}^{-1}\widehat{X^{nj}} + \mathbf{S}^{+}\mathcal{F}^{-1}\widehat{V^{nj}}, \mathbf{C}^{+}\mathcal{F}^{-1}\widehat{V^{nj}}\big) \Big),\
i=1,2,\ldots,s,\\
\widehat{X^{n+1}}&=\varphi_0(hM)\widehat{X^{n}}+ h\textstyle\sum\limits_{j=1}^{s}\bar{b}_{j}(hM)\mathcal{F} \Big(\mathbf{S}^{-}F\big(\mathcal{F}^{-1}\widehat{X^{nj}} + \mathbf{S}^{+}\mathcal{F}^{-1}\widehat{V^{nj}}, \mathbf{C}^{+}\mathcal{F}^{-1}\widehat{V^{nj}}\big) \Big),\\
\widehat{V^{n+1}}&=\varphi_0(hM)\widehat{V^{n}}+  h\textstyle\sum\limits_{j=1}^{s}\bar{b}_{j}(hM)\mathcal{F} \Big(\mathbf{C}^{-}F\big(\mathcal{F}^{-1}\widehat{X^{nj}} + \mathbf{S}^{+}\mathcal{F}^{-1}\widehat{V^{nj}}, \mathbf{C}^{+}\mathcal{F}^{-1}\widehat{V^{nj}}\big) \Big).
\end{array}\end{equation*}

  \item  The full-discretization $x^{n+1}\approx x(t_{n+1})$ and $v^{n+1}\approx v(t_{n+1})$  of  \eqref{charged-particle sts}  is formulated as  \begin{equation*} \begin{aligned}
&x^{n+1}=  \mathbf{X}^{n+1}+  \frac{t_{n+1}}{\epsilon }    \varphi_1(  t_{n+1} J/\epsilon )   \mathbf{V}^{n+1},\ \ v^{n+1}=\frac{1}{\epsilon} \varphi_0(  t_{n+1} J/\epsilon )   \mathbf{V}^{n+1},
 \end{aligned}
\end{equation*}
{where} $\mathbf{X}^{n+1}$ and $\mathbf{V}^{n+1}$ are obtained by the Fourier
pseudospectral method
\begin{equation*}\begin{aligned}
\mathbf{X}^{n+1}=\sum\limits_{\ell\in \mathcal{M}}(\widehat{X^{n+1}})_{\ell}\fe^{\mathrm{i}\ell t_{n+1}/\eps},\qquad \ \mathbf{V}^{n+1}=\sum\limits_{\ell\in \mathcal{M}}(\widehat{V^{n+1}})_{\ell}\fe^{\mathrm{i}\ell  t_{n+1}/\eps}.
 \end{aligned}
\end{equation*}
\end{itemize}
\end{defi}
\subsection{Optimal accuracy}

%
%
%
%

\begin{theo} \label{UA thm2} (\textbf{Optimal accuracy})
For a smooth periodic function $\vartheta(\tau)$ on $\bT$,
denote the Soblev space $\mathcal{H}^{m}(\bT)=\{\vartheta(\tau)\in\mathcal{H}^{m}:\partial^l_{\tau}\vartheta(0)=\partial^l_{\tau}\vartheta(2\pi),l=0,1,\ldots,m \}$.
Assume that the exact solution $X(t,\tau)$ and $V(t,\tau)$ of the system  \eqref{2scale}
satisfy that $X(t,\tau), V(t,\tau)\in \mathbb{C}^r([0,T],\mathcal{H}^{m_0}(\bT))$ with $m_0\geq0.$
Under the conditions  of Theorem \ref{UA thm},   the global error of
the fully discrete scheme
is bounded as
\begin{equation*}
\begin{aligned}
 \normo{X(t_n,\tau)-\mathbf{X}^n}+\normo{\eps V(t_n,\tau)-\eps\mathbf{V}^n} \leq C \big(\eps^2 h^r+ (2\pi/N_{\tau})^{m_0}\big),\quad 0\leq n\leq T/h,\\
\end{aligned}
\end{equation*}
where $C$ is a generic constant independent of $n, h, N_{\tau}, \epsilon$.
\end{theo}
\begin{proof}To prove this result, we introduce an intermediate algorithm (IA) and by which,
the conclusions of this theorem are  converted to the estimations for IA.
To this end,  consider the following trigonometric
polynomials
$$X^{\mathcal{M}}(t,\tau)=\big(X_j^{\mathcal{M}}(t,\tau)\big)_{j=1,2},\quad \quad
V^{\mathcal{M}}(t,\tau)=\big(V_j^{\mathcal{M}}(t,\tau)\big)_{j=1,2},$$
with
\begin{equation*}
\begin{array}[c]{ll}
 X_j^{\mathcal{M}}(t,\tau)=\sum\limits_{k\in \mathcal{M}}
 \widetilde{X_{k,j}}(t)\mathrm{e}^{\mathrm{i} k \tau },\ \ \ \  V_j^{\mathcal{M}}(t,\tau)=\sum\limits_{k\in \mathcal{M}}
 \widetilde{V_{k,j}}(t)\mathrm{e}^{\mathrm{i} k \tau },\ \ \  (t,\tau)\in[0,T]\times
[0,2\pi],
\end{array}
\end{equation*}
such that
\begin{equation*}
\begin{array}[c]{ll}
  &\partial_tX^{\mathcal{M}}(t,\tau)+\frac{1}{\epsilon }\partial_\tau X^{\mathcal{M}}(t,\tau)=-  \tau  \varphi_1( -\tau J ) F\big( X^{\mathcal{M}}(t,\tau)+ \tau  \varphi_1( \tau J )V^{\mathcal{M}}(t,\tau), \varphi_0( \tau J )V^{\mathcal{M}}(t,\tau)\big),\\
    &\partial_tV^{\mathcal{M}}(t,\tau)+\frac{1}{\epsilon }\partial_\tau V^{\mathcal{M}}(t,\tau)=\varphi_0( -\tau J )F\big( X^{\mathcal{M}}(t,\tau)+ \tau  \varphi_1( \tau J )V^{\mathcal{M}}(t,\tau), \varphi_0( \tau J )V^{\mathcal{M}}(t,\tau)\big).
\end{array}
\end{equation*}
Here  $\widetilde{X_{k,j}}$ and $\widetilde{V_{k,j}}$
are the Fourier transform
coefficients of the periodic functions $ X_j^{\mathcal{M}}$  and $ V_j^{\mathcal{M}}$, respectively.
{According to the Fourier
functions' orthogonality and collecting all the $\widetilde{X_{k,j}},\
\widetilde{V_{k,j}}$ in   $(N_\tau+1)$-periodic coefficient vectors $
\widetilde{\mathbf{X}}(t)= (\widetilde{X_{k,j}}(t)),\
\widetilde{\mathbf{V}}(t) = (\widetilde{V_{k,j} }(t))$, one gets}
\begin{equation}\label{2scale Fourier-f-f}\begin{aligned}
&\frac{d}{dt}\widetilde{\mathbf{X}}(t)=\mathrm{i}\Omega\widetilde{\mathbf{X}}(t)
+\mathcal{F} \Big(\mathbf{S}^{-}F\big(\mathcal{F}^{-1}\widetilde{\mathbf{X}}(t)+ \mathbf{S}^{+}\mathcal{F}^{-1}\widetilde{\mathbf{V}}(t), \mathbf{C}^{+}\mathcal{F}^{-1}\widetilde{\mathbf{V}}(t)\big) \Big),\\
&\frac{d}{dt}\widetilde{\mathbf{V}}(t)=\mathrm{i}\Omega\widetilde{\mathbf{V}}(t)
+ \mathcal{F}\Big(\mathbf{C}^{-}F\big(\mathcal{F}^{-1}\widetilde{\mathbf{X}}(t)+ \mathbf{S}^{+}\mathcal{F}^{-1}\widetilde{\mathbf{V}}(t), \mathbf{C}^{+}\mathcal{F}^{-1}\widetilde{\mathbf{V}}(t)\big) \Big).
 \end{aligned}
\end{equation}
Then  the intermediate algorithm (IA)  is defined by
  \begin{equation*}
\begin{aligned}
&X_{\mathcal{M},j}^{ni}(\tau)=\sum\limits_{k\in \mathcal{M}}
\widetilde{X^{ni}_{k,j}}\mathrm{e}^{\mathrm{i} k \tau },\quad \  \ V_{\mathcal{M},j}^{ni}(\tau)=\sum\limits_{k\in \mathcal{M}}
\widetilde{V^{ni}_{k,j}}\mathrm{e}^{\mathrm{i} k \tau },\ \ \ \ \ i=1,2,\ldots,s,\\
&X_{\mathcal{M},j}^{n+1}(\tau)=\sum\limits_{k\in \mathcal{M}}
\widetilde{X^{n+1}_{k,j}}\mathrm{e}^{\mathrm{i} k \tau },\quad V_{\mathcal{M},j}^{n+1}(\tau)=\sum\limits_{k\in \mathcal{M}}
\widetilde{V^{n+1}_{k,j}}\mathrm{e}^{\mathrm{i} k \tau },\ \ \  n=0,1,\ldots, T/h-1,
\end{aligned}
\end{equation*}
 where the following $s$-stage exponential integrator is applied  to  solving \eqref{2scale Fourier-f-f}:
\begin{equation}
\begin{array}[c]{ll}%
\widetilde{X^{ni}}&=\varphi_0(c_{i}hM)\widetilde{X^{n}}+h\textstyle\sum\limits_{j=1}^{s}\bar{a}_{ij}(hM)\mathcal{F} \Big(\mathbf{S}^{-}F\big(\mathcal{F}^{-1}\widetilde{X^{nj}} + \mathbf{S}^{+}\mathcal{F}^{-1}\widetilde{V^{nj}}, \mathbf{C}^{+}\mathcal{F}^{-1}\widetilde{V^{nj}}\big) \Big),\
i=1,2,\ldots,s,\\
\widetilde{V^{ni}}&=\varphi_0(c_{i}hM)\widetilde{V^{n}}+  h\textstyle\sum\limits_{j=1}^{s}\bar{a}_{ij}(hM)\mathcal{F} \Big(\mathbf{C}^{-}F\big(\mathcal{F}^{-1}\widetilde{X^{nj}} + \mathbf{S}^{+}\mathcal{F}^{-1}\widetilde{V^{nj}}, \mathbf{C}^{+}\mathcal{F}^{-1}\widetilde{V^{nj}}\big) \Big),\
i=1,2,\ldots,s,\\
\widetilde{X^{n+1}}&=\varphi_0(hM)\widetilde{X^{n}}+ h\textstyle\sum\limits_{j=1}^{s}\bar{b}_{j}(hM)\mathcal{F} \Big(\mathbf{S}^{-}F\big(\mathcal{F}^{-1}\widetilde{X^{nj}} + \mathbf{S}^{+}\mathcal{F}^{-1}\widetilde{V^{nj}}, \mathbf{C}^{+}\mathcal{F}^{-1}\widetilde{V^{nj}}\big) \Big),\\
\widetilde{V^{n+1}}&=\varphi_0(hM)\widetilde{V^{n}}+  h\textstyle\sum\limits_{j=1}^{s}\bar{b}_{j}(hM)\mathcal{F} \Big(\mathbf{C}^{-}F\big(\mathcal{F}^{-1}\widetilde{X^{nj}} + \mathbf{S}^{+}\mathcal{F}^{-1}\widetilde{V^{nj}}, \mathbf{C}^{+}\mathcal{F}^{-1}\widetilde{V^{nj}}\big) \Big).
\end{array}
 \label{erk dingyi}%
\end{equation}

To study the accuracy of the fully discrete scheme, consider the standard projection operator  $P_{\mathcal{M}} : L^2([-\pi,\pi]) \rightarrow
Y_{\mathcal{M}}:=\textmd{span}\{e^{ \mathrm{i}   k \tau },\ k\in \mathcal{M},\ \tau \in[-\pi,\pi]\}
$ as
$(P_{\mathcal{M}}v)(\tau)=\sum\limits_{k\in \mathcal{M}}\widetilde{v_k} e^{ \mathrm{i}k \tau }.$
With the notations $\mathbf{X}^{ni}=\sum\limits_{\ell\in \mathcal{M}}(\widehat{X^{ni}})_{\ell}\fe^{\mathrm{i}\ell (t_{n}+c_ih)/\eps}$ and $\mathbf{V}^{ni}=\sum\limits_{\ell\in \mathcal{M}}(\widehat{V^{ni}})_{\ell}\fe^{\mathrm{i}\ell (t_{n}+c_ih)/\eps}$,
define the error functions of our fully discrete scheme by
\[
\begin{aligned}
&e_X^n(\tau):=X(t_n,\tau)-\mathbf{X}^n,\quad \ \ \ E_X^{ni}(\tau):=X(t_n+c_ih,\tau)-\mathbf{X}^{ni},\\
&e_V^n(\tau):=V(t_n,\tau)-\mathbf{V}^n,\quad \ \ \ \   E_V^{ni}(\tau):=V(t_n+c_ih,\tau)-\mathbf{V}^{ni},\end{aligned}
\]
and the projected errors of   the intermediate algorithm as
\[
\begin{aligned}
&e_{\mathcal{M},X}^n(\tau):=P_{\mathcal{M}} X(t_n,\tau)-X_{\mathcal{M}}^n,\quad \ \ \ E_{\mathcal{M},X}^{ni}(\tau):=P_{\mathcal{M}} X(t_n+c_ih,\tau)-X_{\mathcal{M}}^{ni},\\
&e_{\mathcal{M},V}^n(\tau):=P_{\mathcal{M}} V(t_n,\tau)-V_{\mathcal{M}}^n,\quad \ \ \ \ E_{\mathcal{M},V}^{ni}(\tau):=P_{\mathcal{M}} V(t_n+c_ih,\tau)-V_{\mathcal{M}}^{ni}.\\\end{aligned}
\] Based on the estimates on projection error \cite{Shen} and the
triangle inequality, it follows that
\[
\begin{aligned}  \normo{e_X^n}&\leq  \normo{e_{\mathcal{M},X}^n}+\normo{ X_{\mathcal{M}}^n-\mathbf{X}^n}+ \normo{X(t_n,\tau)-P_{\mathcal{M}} X(t_n,\tau)}  \leq \normo{e_{\mathcal{M},X}^n}   +C (2\pi/N_{\tau})^{m_0},\\
\normo{E_X^{ni}}&\leq  \normo{E_{\mathcal{M},X}^{ni}}+ \normo{ X_{\mathcal{M}}^{ni}-\mathbf{X}^{ni}}+
\normo{X(t_n+c_i h,\tau)-P_{\mathcal{M}} X(t_n+c_i h,\tau)} \\&\leq \normo{E_{\mathcal{M},X}^{ni}}+ C(2\pi/N_{\tau})^{m_0}.
\end{aligned}
\]
Similar results can be derived for $e_V^n,\ E_V^{ni}$.
Therefore, the estimations for $e_X^n, e_V^n$ and  $E_X^{ni},\ E_V^{ni}$  can be turned to estimate the estimations for  $e_{\mathcal{M},X}^n, e_{\mathcal{M},V}^n$ and $E_{\mathcal{M},X}^{ni},\ E_{\mathcal{M},V}^{ni}$.

The error system of the intermediate algorithm is given by
  \begin{equation*}
\begin{aligned}
 &e_{\mathcal{M},X}^{n+1}(\tau)=\sum\limits_{k\in \mathcal{M}}
\big(\widetilde{e_{\mathcal{M},X}^{n+1}}\big)_k\mathrm{e}^{\mathrm{i} k \tau },\quad\quad
E_{\mathcal{M},X}^{ni}(\tau)=\sum\limits_{k\in \mathcal{M}}
\big(\widetilde{E_{\mathcal{M},X}^{ni}}\big)_k\mathrm{e}^{\mathrm{i} k \tau },\\
 &e_{\mathcal{M},V}^{n+1}(\tau)=\sum\limits_{k\in \mathcal{M}}
\big(\widetilde{e_{\mathcal{M},V}^{n+1}}\big)_k\mathrm{e}^{\mathrm{i} k \tau },\quad\quad
E_{\mathcal{M},V}^{ni}(\tau)=\sum\limits_{k\in \mathcal{M}}
\big(\widetilde{E_{\mathcal{M},V}^{ni}}\big)_k\mathrm{e}^{\mathrm{i} k \tau },
\end{aligned}
\end{equation*}
where
\begin{equation*}
\begin{aligned}
\widetilde{e_{\mathcal{M},X}^{n+1}}&=\varphi_0(hM)\widetilde{e_{\mathcal{M},X}^{n}}+h \sum\limits_{j=1}^{s}\bar{b}_{j}(hM) \Delta\widetilde{f^{nj}_X}+\widetilde{\delta_X^{n+1}},\
\widetilde{E_{\mathcal{M},X}^{ni}}=\varphi_0(c_ihM)\widetilde{e_{\mathcal{M},X}^{n}}+ h \textstyle\sum\limits_{j=1}^{s}\bar{a}_{ij}(hM) \Delta\widetilde{f^{nj}_X}+\widetilde{\Delta_X^{ni}},\\
\widetilde{e_{\mathcal{M},V}^{n+1}}&=\varphi_0(hM)\widetilde{e_{\mathcal{M},V}^{n}}+h \sum\limits_{j=1}^{s}\bar{b}_{j}(hM) \Delta\widetilde{f^{nj}_V}+\widetilde{\delta_V^{n+1}},\
\widetilde{E_{\mathcal{M},V}^{ni}}=\varphi_0(c_ihM)\widetilde{e_{\mathcal{M},V}^{n}}+ h \textstyle\sum\limits_{j=1}^{s}\bar{a}_{ij}(hM) \Delta\widetilde{f^{nj}_V}+\widetilde{\Delta_V^{ni}},\\
\end{aligned}
\end{equation*}
{and}
\begin{equation*}
\begin{aligned} &\widetilde{G}( {X}, {V})=F\big(\mathcal{F}^{-1}X + \mathbf{S}^{+}\mathcal{F}^{-1}V, \mathbf{C}^{+}\mathcal{F}^{-1}V\big) ,\\
&\Delta\widetilde{f^{nj}_X}=\mathcal{F}  \mathbf{S}^{-} \big(\widetilde{G}(P_{\mathcal{M}} X(t_n+c_jh,\tau),P_{\mathcal{M}} V(t_n+c_jh,\tau)) -\widetilde{G}(X_{\mathcal{M}}^{nj},V_{\mathcal{M}}^{nj})\big),\\
&\Delta\widetilde{f^{nj}_V}=  \mathcal{F}  \mathbf{C}^{-} \big(\widetilde{G}(P_{\mathcal{M}} X(t_n+c_jh,\tau),P_{\mathcal{M}} V(t_n+c_jh,\tau)) -\widetilde{G}(X_{\mathcal{M}}^{nj},V_{\mathcal{M}}^{nj})\big).
\end{aligned}
\end{equation*}
 Here the remainders
$\widetilde{\delta_X^{n+1}},\ \widetilde{\Delta_X^{ni}}$  and
$\widetilde{\delta_V^{n+1}},\ \widetilde{\Delta_V^{ni}}$ are determined by  inserting the exact solution of \eqref{2scale Fourier-f-f}
into the numerical approximation  \eqref{erk dingyi}, {i.e.,}
\begin{equation*}
\begin{array}[c]{ll}%
\widetilde{\mathbf{X}(t_n+c_ih)}&=\varphi_0(c_{i}hM)\widetilde{\mathbf{X}(t_n)}+h\textstyle\sum\limits_{j=1}^{s}\bar{a}_{ij}(hM)
\mathcal{F}  \mathbf{S}^{-} \widetilde{G}(\widetilde{\mathbf{X}(t_n+c_jh)},\widetilde{\mathbf{V}(t_n+c_jh)})+ \widetilde{\Delta_X^{ni}},\\
\widetilde{\mathbf{V}(t_n+c_ih)}&=\varphi_0(c_{i}hM)\widetilde{\mathbf{V}(t_n)}+  h\textstyle\sum\limits_{j=1}^{s}\bar{a}_{ij}(hM)
\mathcal{F}  \mathbf{C}^{-} \widetilde{G}(\widetilde{\mathbf{X}(t_n+c_jh)},\widetilde{\mathbf{V}(t_n+c_jh)})+ \widetilde{\Delta_V^{ni}}, \\
\widetilde{\mathbf{X}(t_n+h)}&=\varphi_0(hM)\widetilde{\mathbf{X}(t_n)}+ h\textstyle\sum\limits_{j=1}^{s}\bar{b}_{j}(hM)
\mathcal{F}  \mathbf{S}^{-} \widetilde{G}(\widetilde{\mathbf{X}(t_n+c_jh)},\widetilde{\mathbf{V}(t_n+c_jh)})+\widetilde{\delta_X^{n+1}},\\
\widetilde{\mathbf{V}(t_n+h)}&=\varphi_0(hM)\widetilde{\mathbf{V}(t_n)}+ h\textstyle\sum\limits_{j=1}^{s}\bar{b}_{j}(hM)
\mathcal{F}  \mathbf{C}^{-} \widetilde{G}(\widetilde{\mathbf{X}(t_n+c_jh)},\widetilde{\mathbf{V}(t_n+c_jh)})+\widetilde{\delta_V^{n+1}}.
\end{array}
\end{equation*}
Following the same arguments of Lemma \ref{LBS}, the bounds of these remainders are derived as
\begin{equation*}
\begin{aligned} & \normm{\widetilde{\Delta_X^{ni}}}\leq C \epsilon ^2 h^{r},\  \normm{\widetilde{\Delta_V^{ni}}}\leq C \epsilon ^2
h^{r}, \  \normm{\widetilde{\delta_X^{n+1}}}\leq C \epsilon ^2 h^{r+1},\  \normm{ \widetilde{\delta_V^{n+1}}}\leq C \epsilon ^2
h^{r+1}.\end{aligned} 
\end{equation*}
Based on the foregoing estimates, we deduce that
\begin{equation*}
\begin{aligned}
&\normo{ e_{\mathcal{M},X}^{n+1}} \leq\normo{ e_{\mathcal{M},X}^{n}}+h  C
\sum\limits_{j=1}^{s}\normo{\Delta\widetilde{f^{nj}_X}} + C\eps^2 h^{r+1},\\
& \normo{ e_{\mathcal{M},V}^{n+1}} \leq\normo{ e_{\mathcal{M},V}^{n}}+h C
\sum\limits_{j=1}^{s}\normo{\Delta\widetilde{f^{nj}_V}} + C\eps^2 h^{r+1}.
\end{aligned}\end{equation*}
The proof is, then, concluded by   using the same arguments as Theorem \ref{UA thm}.
\end{proof}

\section{Practical integrators and numerical tests}\label{sec:4}
\subsection{Some practical integrators}
Before the above discretization is applied in practical computations,  the coefficients $c_i$, $\bar{a}_{ij}(h/\epsilon \partial_\tau)$ and
$\bar{b}_{i}(h/\epsilon \partial_\tau)$ appearing in \eqref{cs ei-ful2} should be determined, which is derived in this present section.

\textbf{Second order integrator}.
  We first consider second order integrators which can be realized by one-stage schemes, i.e., $s=1$. The first method is obtained  by considering
$$c_1=\frac{1}{2}, \quad\bar{b}_{1}(h/\epsilon \partial_\tau)=   \varphi_1(h/\epsilon \partial_\tau),\quad
\bar{a}_{11}(h/\epsilon \partial_\tau)=\varphi_1(c_1h/\epsilon \partial_\tau) ,$$
which yields an implicit integrator. We shall refer to it as   \textbf{IO2}. To get an explicit scheme, we modify the scheme \eqref{cs ei-ful2} of IO2 as
\begin{equation*}
\begin{array}[c]{ll}%
X^{n1}&=\varphi_0(c_{1}h/\epsilon \partial_\tau)X^{n}-h \bar{a}_{11}(h/\epsilon \partial_\tau)\tau  \varphi_1( -\tau J )F\big( X^{n} +\tau  \varphi_1( \tau J ) V^{n},\varphi_0( \tau J )V^{n}\big) ,\\
V^{n1}&=\varphi_0(c_{1}h/\epsilon \partial_\tau)V^{n}+h \bar{a}_{11}(h/\epsilon \partial_\tau)\varphi_0(- \tau J )F\big( X^{n} +\tau  \varphi_1( \tau J ) V^{n},\varphi_0( \tau J )V^{n}\big),\\
X^{n+1}&=\varphi_0(h/\epsilon \partial_\tau)X^{n}- h\bar{b}_{1}(h/\epsilon \partial_\tau)\tau  \varphi_1( -\tau J )F\big( X^{n1} +\tau  \varphi_1( \tau J ) V^{n1},\varphi_0( \tau J )V^{n1}\big),\\
V^{n+1}&=\varphi_0(h/\epsilon \partial_\tau)V^{n}+  h\bar{b}_{1}(h/\epsilon \partial_\tau)\varphi_0(- \tau J )F\big( X^{n1} +\tau  \varphi_1( \tau J ) V^{n1},\varphi_0( \tau J )V^{n1}\big).
\end{array}\end{equation*}
This explicit  second order method is denoted by  \textbf{EO2}. We note that for these two methods,    the second order  initial data, i.e.,  \eqref{inv} with $k=2$, is enough.

\textbf{Fourth order integrator}.
We now turn to the fourth order integrators with the fourth order  initial data which is obtained by \eqref{inv} with $k=4$.
This can be achieved by three-stage implicit integrators. Solving
$\psi_{i}(h/\epsilon \partial_\tau)=0$ and $\psi_{j,i}(h/\epsilon \partial_\tau)=0$ for $i,j=1,2,3,
$ and choosing $c_1=1,\   c_2=1/2,\  c_3=0$ yields
  \begin{equation*}
\begin{aligned} &\bar{a}_{31}(\Upsilon)=\bar{a}_{32}(\Upsilon)=\bar{a}_{33}(\Upsilon)=0,\qquad \qquad  \quad \ \  \ \
\bar{a}_{21}(\Upsilon)=-\frac{1}{4} \varphi_2(c_2\Upsilon) + \frac{1}{2}\varphi_3(c_2\Upsilon) ,\\
&\bar{a}_{22}(\Upsilon)=\varphi_2(c_2\Upsilon) -\varphi_3(c_2\Upsilon),\qquad   \qquad   \qquad  \quad \bar{a}_{23}(\Upsilon)=\frac{1}{2}\varphi_1(c_2\Upsilon)-\frac{3}{4}\varphi_2(c_2\Upsilon) +\frac{1}{2}\varphi_3(\Upsilon),\\
&\bar{a}_{11}(\Upsilon)=\bar{b}_{1}(\Upsilon)= 4\varphi_3(\Upsilon)  - \varphi_2(\Upsilon) ,\qquad  \quad \quad \ \
 \bar{a}_{12}(\Upsilon)=\bar{b}_{2}(\Upsilon)=4\varphi_2(\Upsilon) -8\varphi_3(\Upsilon),\\
 &\bar{a}_{13}(\Upsilon)=\bar{b}_{3}(\Upsilon)=\varphi_1(\Upsilon)-3\varphi_2(\Upsilon) +4\varphi_3(\Upsilon),
\end{aligned}
\end{equation*}
with $\Upsilon=h/\epsilon \partial_\tau.$
 It can be checked that these coefficients satisfy all the fourth stiff order conditions presented in Table \ref{tab1}.
 This implicit integrator of order four is referred as \textbf{IO4}. For explicit  examples, we need to consider $s=5$ and  choose the  coefficients (\cite{Ostermann06})
 \begin{equation*}
\begin{array}
[c]{ll}%
&c_1=0,\ \ \qquad\qquad\quad\ c_2=c_3=c_5=\frac{1}{2}, \qquad\  c_4=1,\\
&a_{2,1}=\frac{1}{2}\varphi_{1,2},\qquad  \ \quad a_{3,1}=\frac{1}{2}\varphi_{1,3}-\varphi_{2,3},
\quad \ a_{3,2}=\varphi_{2,3},\\
&a_{4,1}= \varphi_{1,4}-2\varphi_{2,4},\ \ a_{4,2}= a_{4,3}=\varphi_{2,4},\qquad
a_{5,1}= \frac{1}{2}\varphi_{1,5}-2a_{5,2}-a_{5,4},\\ &a_{5,2}=  \frac{1}{2}\varphi_{2,5}-\varphi_{3,4}+\frac{1}{2}\varphi_{2,4}-\frac{1}{2}\varphi_{3,5},\qquad
\ \ \ \ a_{5,3}=  a_{5,2},\ \ a_{5,4}=  \frac{1}{2}\varphi_{2,5}-\varphi_{5,2},\\
&b_{1}=\varphi_{1}-3\varphi_{2}+4\varphi_{3},\ \ b_{2}=b_{3}=0,\qquad\ \ \ \ \ \ b_{4}=-\varphi_{2}+4\varphi_{3},\ \ b_{5}=4\varphi_{2}-8\varphi_{3},\\
\end{array}
\end{equation*}
where $\varphi_{i,j}=\varphi_{i,j}(\Upsilon)=\varphi_{i}(c_j\Upsilon).$
 This integrator is referred as \textbf{EO4}.

We end this section by noting that, with the definition of symmetric methods \cite{hairer2006}, it can be verified that the above two explicit schemes (EO2 and EO4) are not symmetric but the implicit ones (IO2 and IO4) are symmetric.

\subsection{Numerical experiment}\label{subNT}
\begin{figure}[t!]
$$\begin{array}{cc}
\psfig{figure=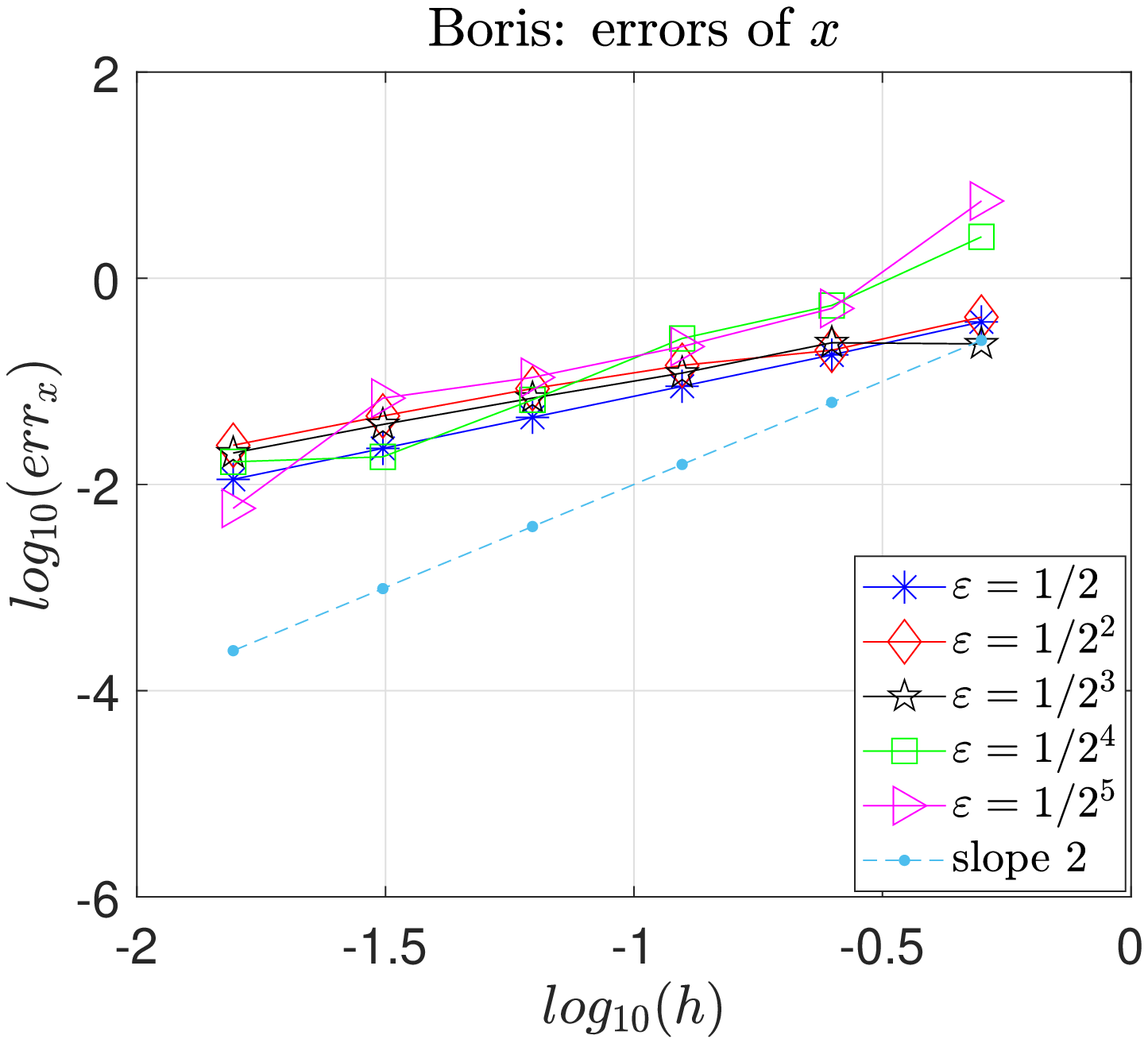,height=4.0cm,width=4.8cm}
\psfig{figure=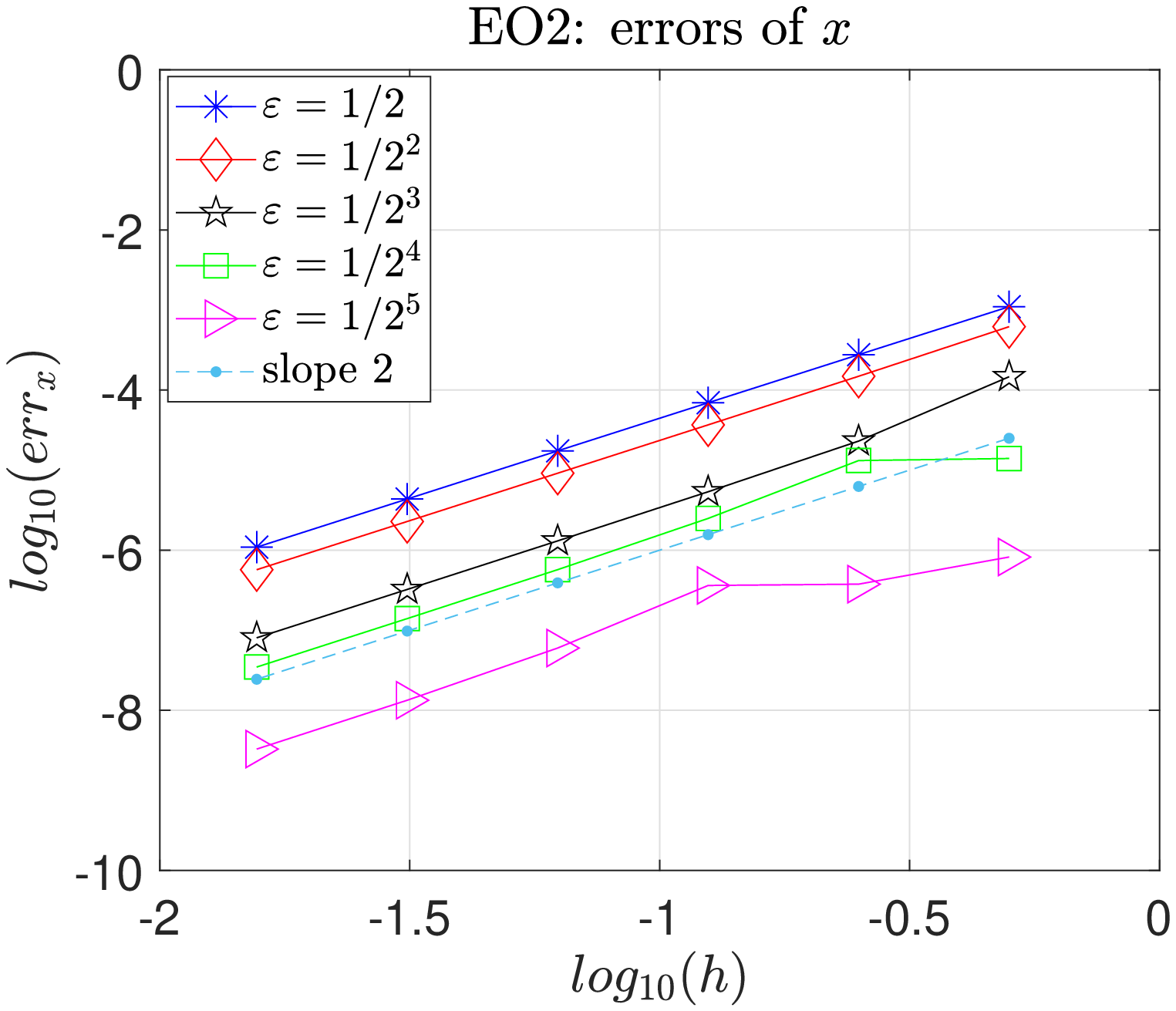,height=4.0cm,width=4.8cm}
\psfig{figure=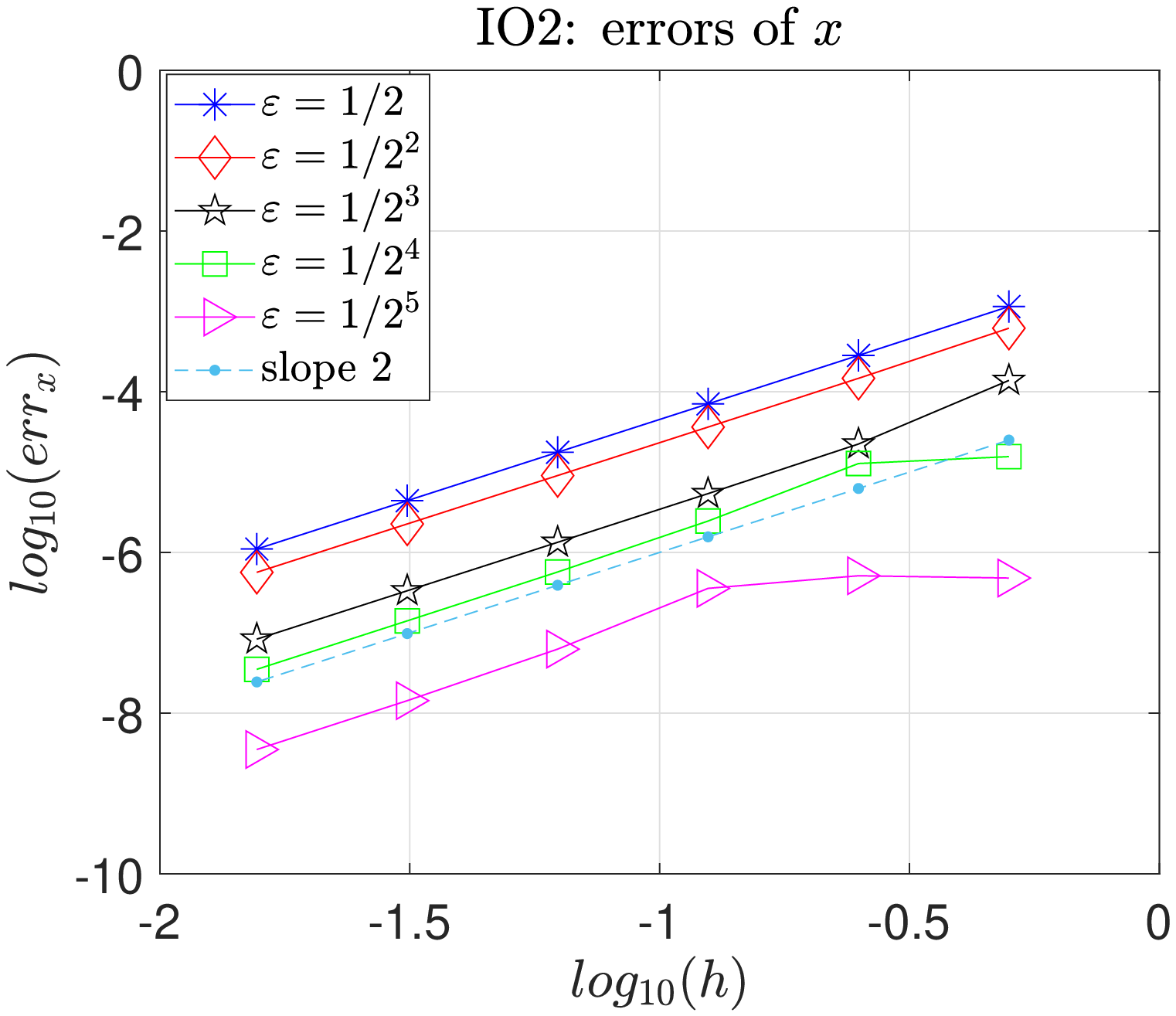,height=4.0cm,width=4.8cm}\\
\psfig{figure=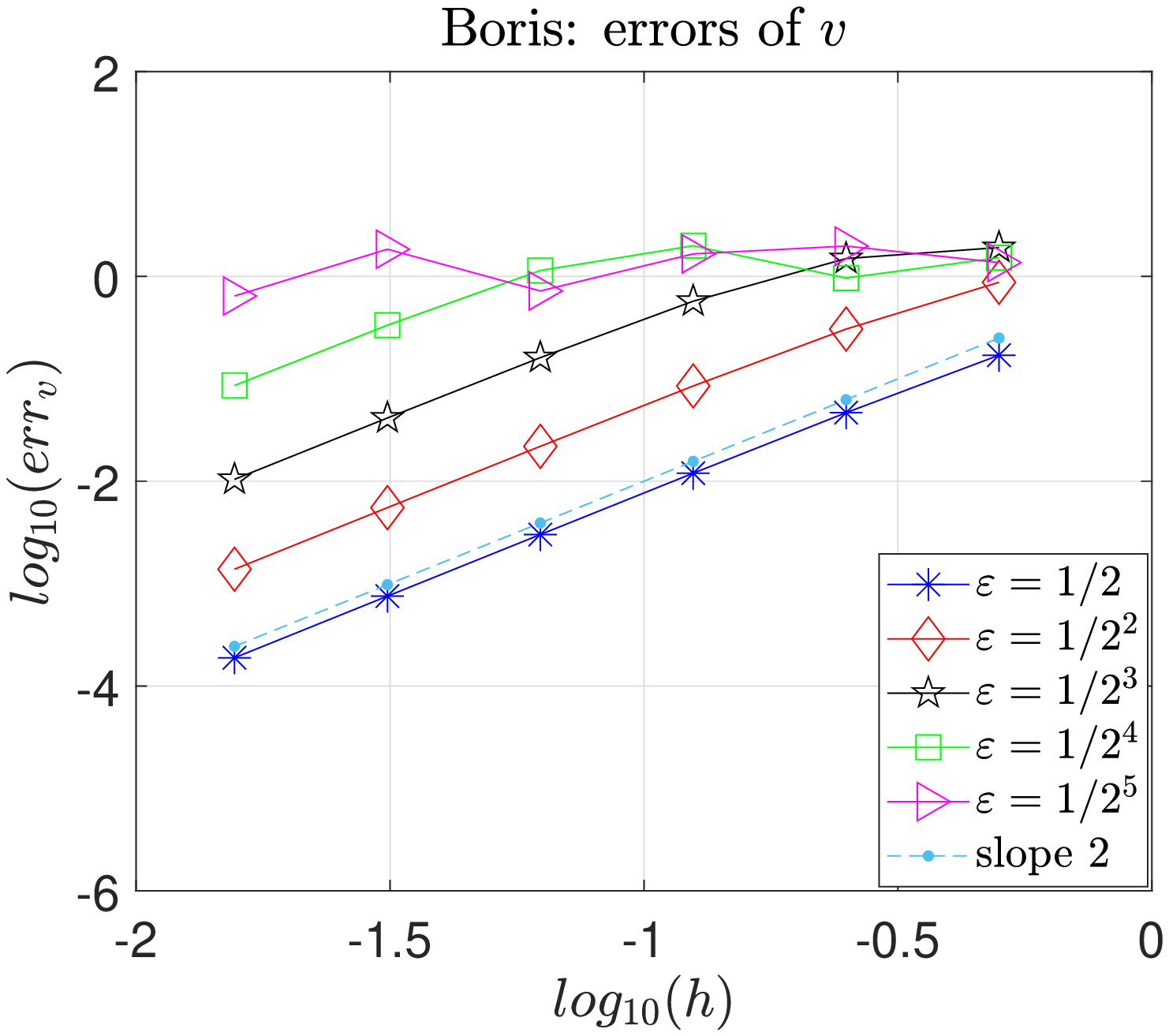,height=4.0cm,width=4.8cm}
\psfig{figure=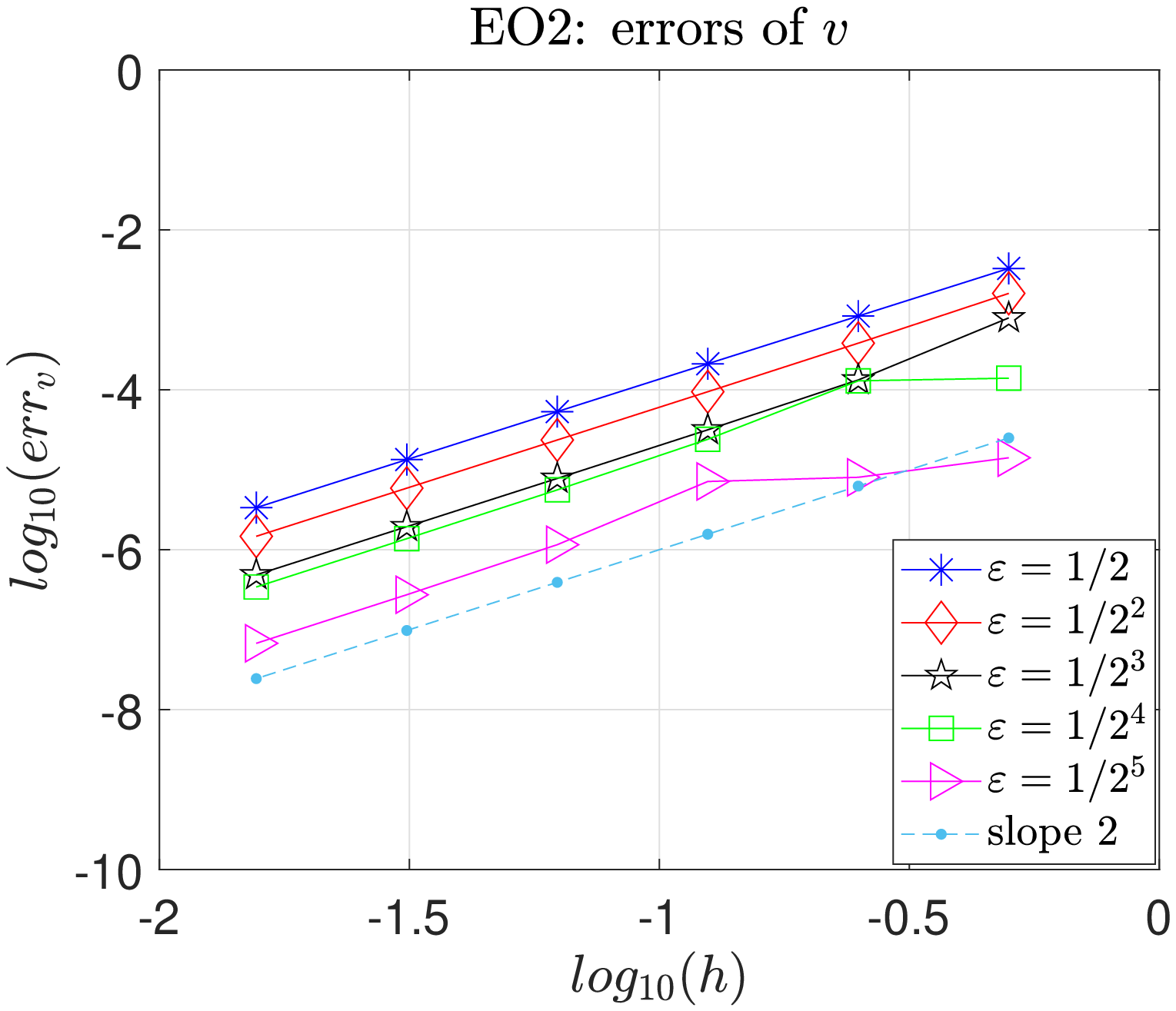,height=4.0cm,width=4.8cm}
\psfig{figure=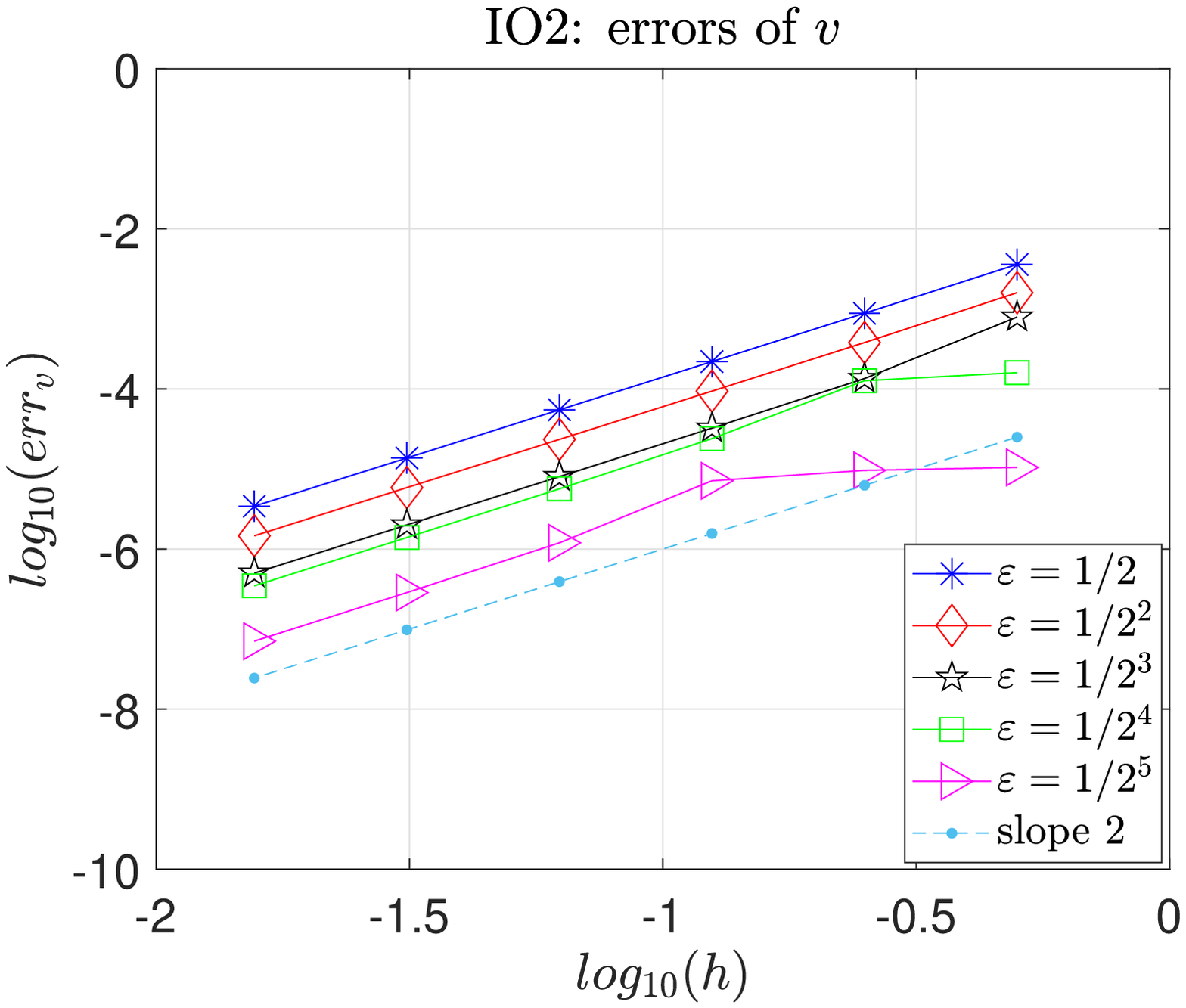,height=4.0cm,width=4.8cm}\\
\psfig{figure=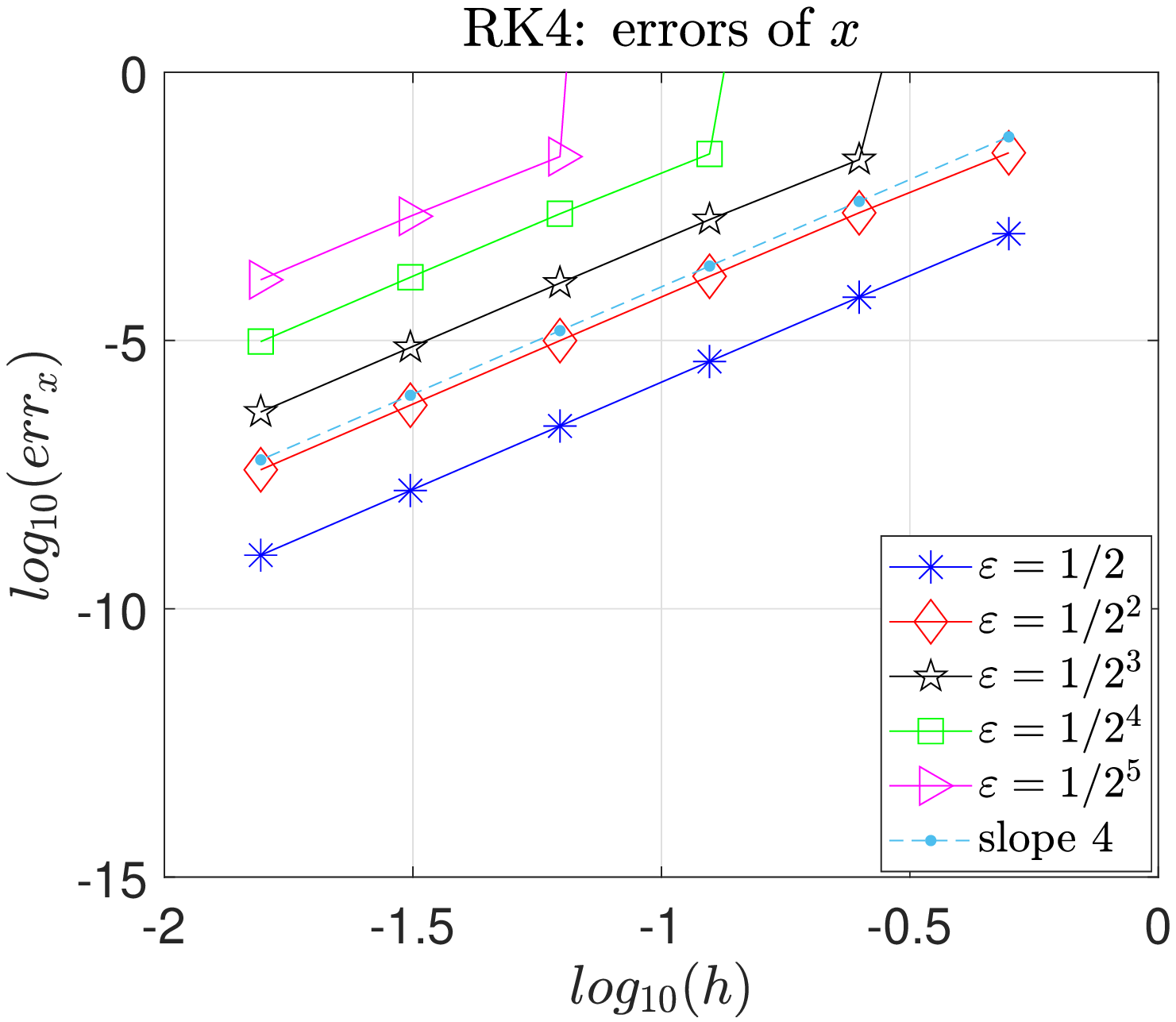,height=4.0cm,width=4.8cm}
\psfig{figure=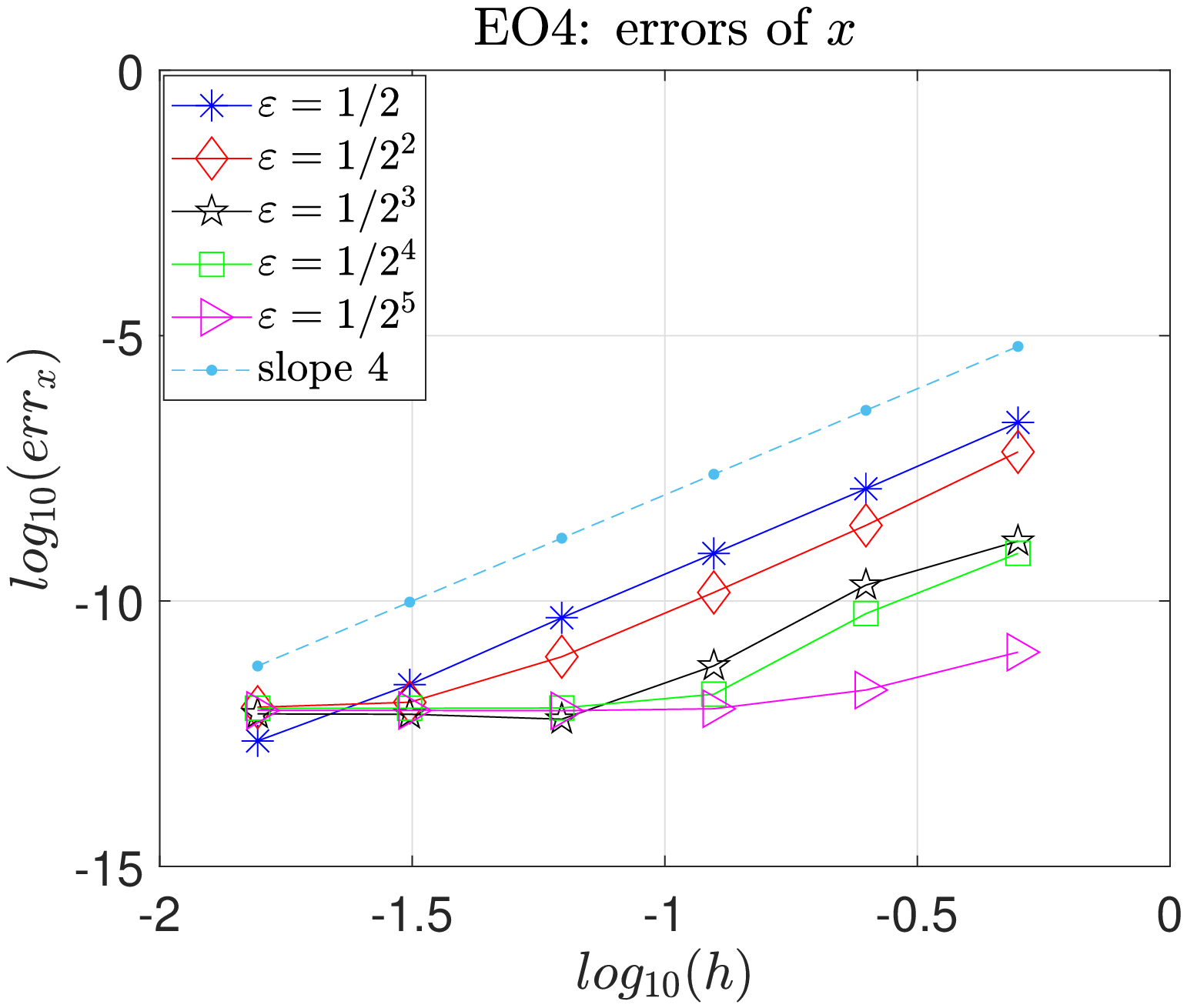,height=4.0cm,width=4.8cm}
\psfig{figure=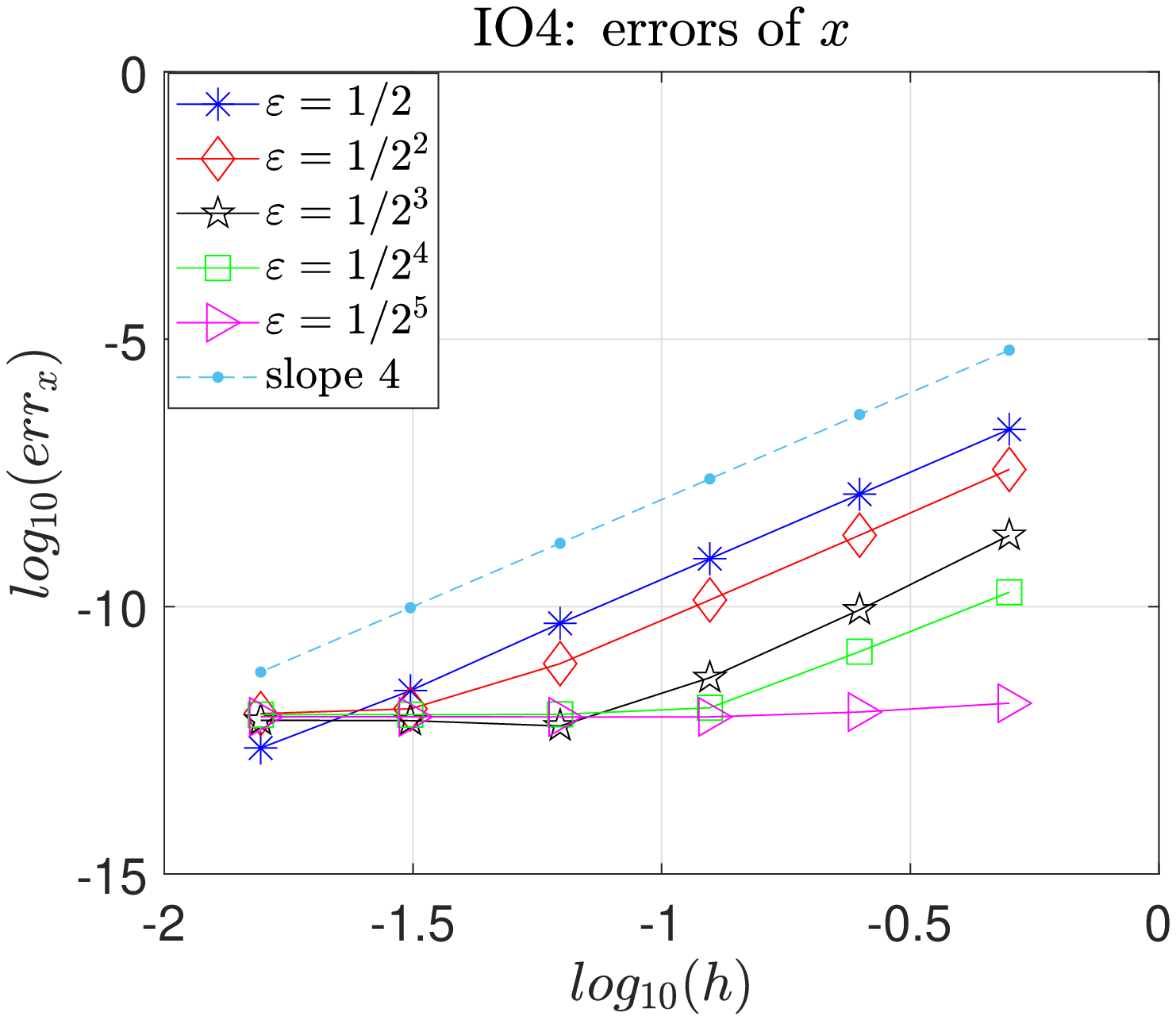,height=4.0cm,width=4.8cm}\\
\psfig{figure=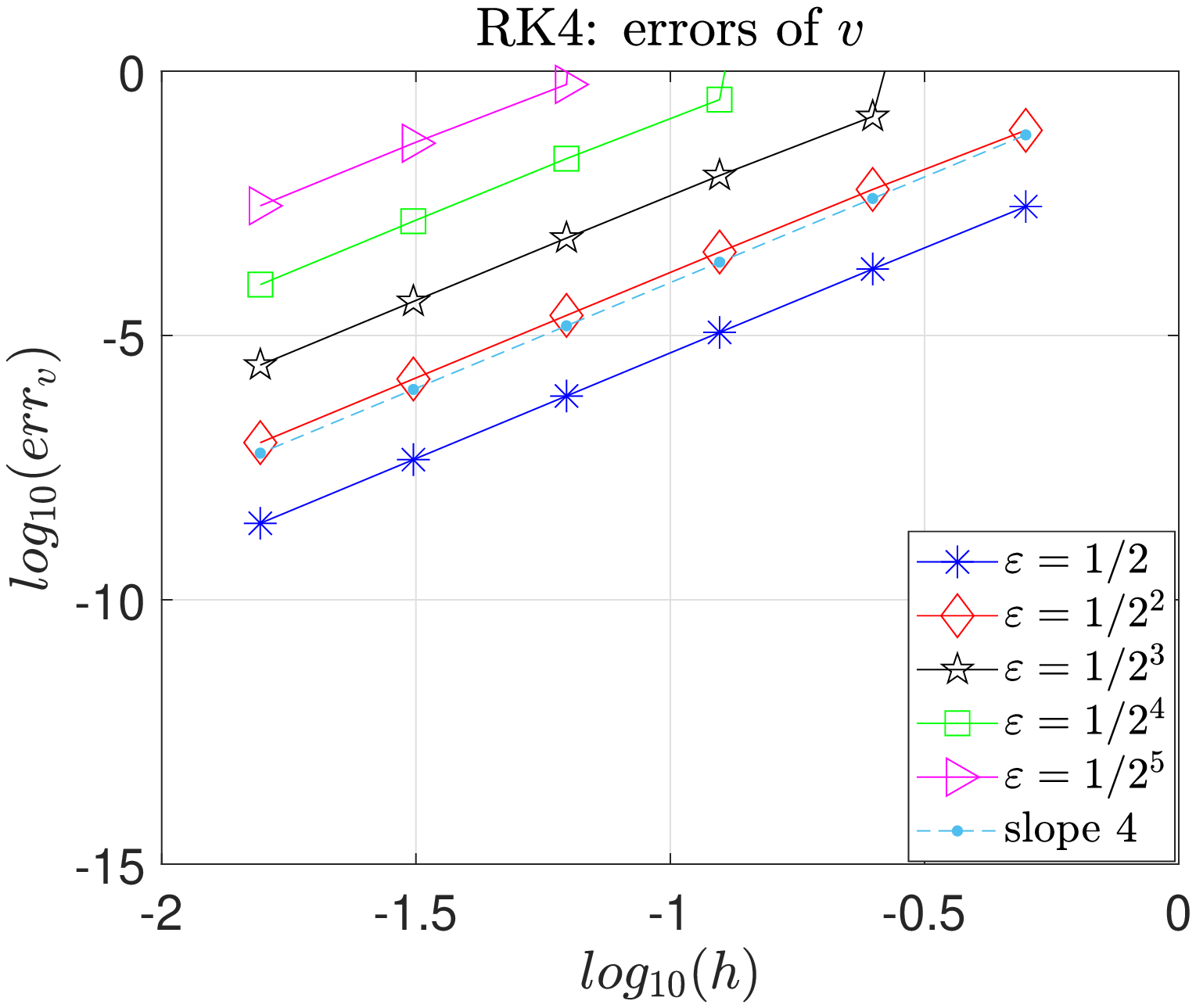,height=4.0cm,width=4.8cm}
\psfig{figure=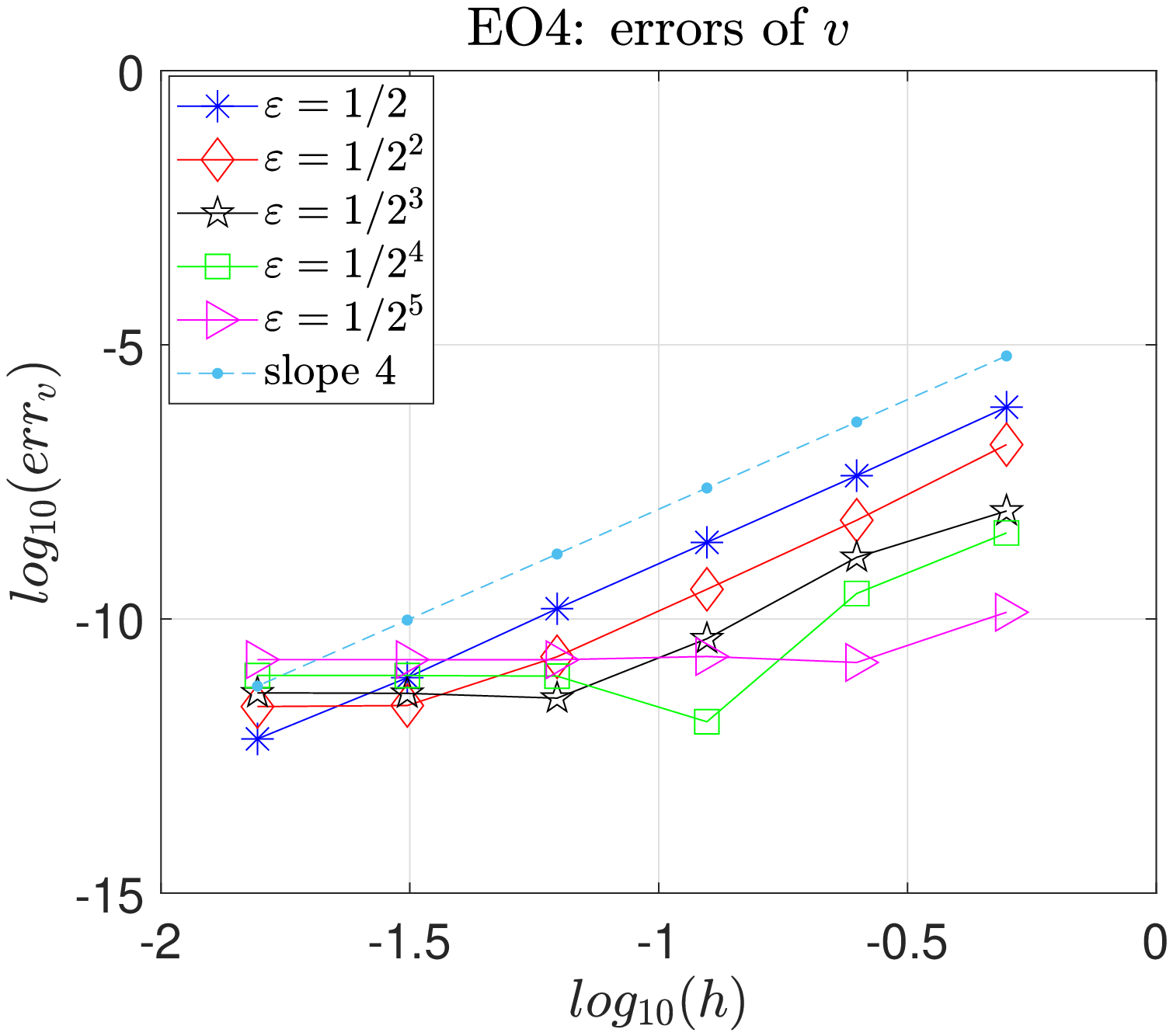,height=4.0cm,width=4.8cm}
\psfig{figure=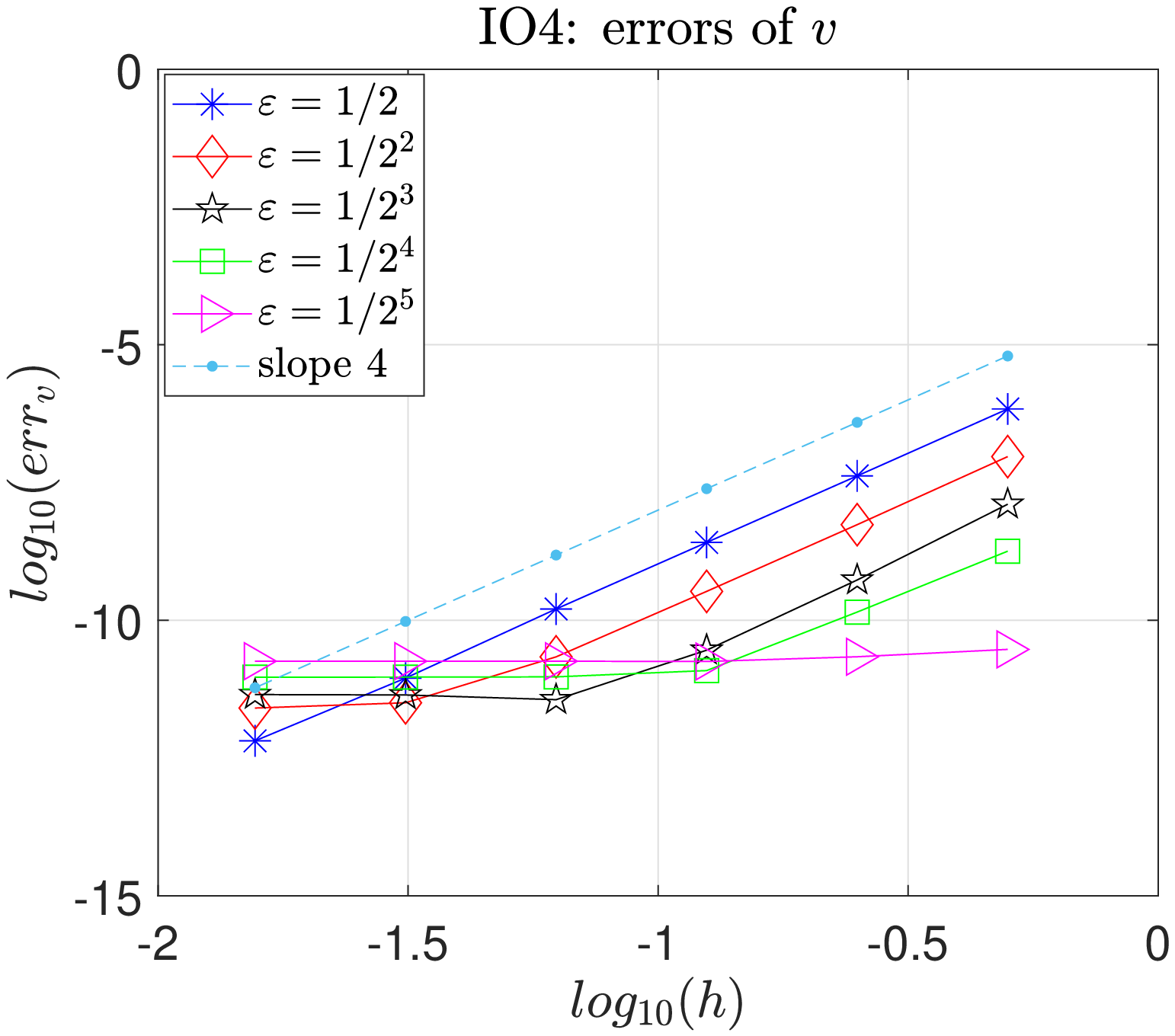,height=4.0cm,width=4.8cm}
\end{array}$$
\caption{The errors (\ref{errx})  at $t=1$ of the second-order schemes (top two rows) and fourth-order schemes (below two rows)  with $h=1/2^k$ for $k=1,2,\ldots,6$ under different $\eps$. }\label{fig21}
\end{figure}

\begin{figure}[h!]
$$\begin{array}{cc}
\psfig{figure=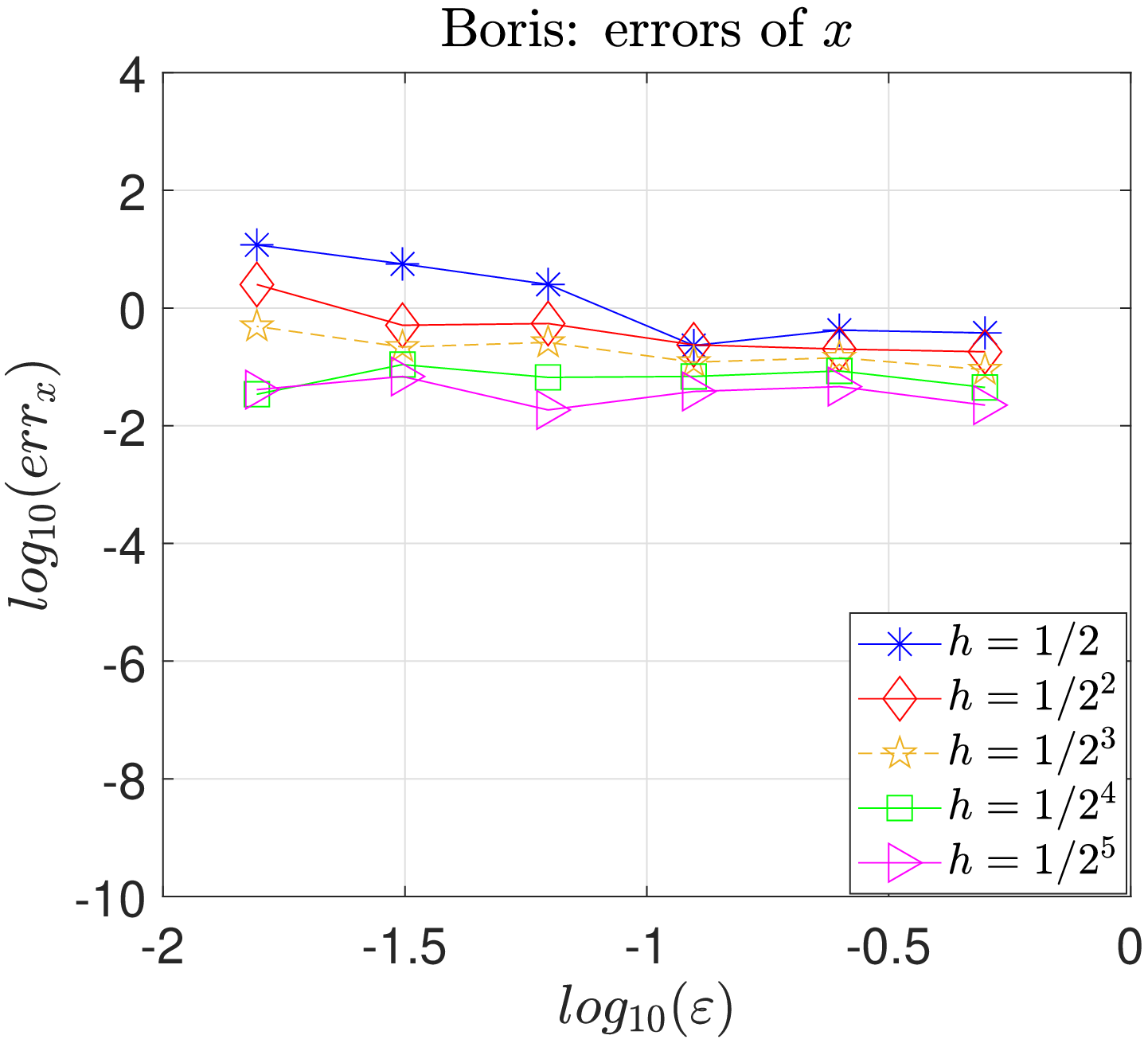,height=4.0cm,width=4.8cm}
\psfig{figure=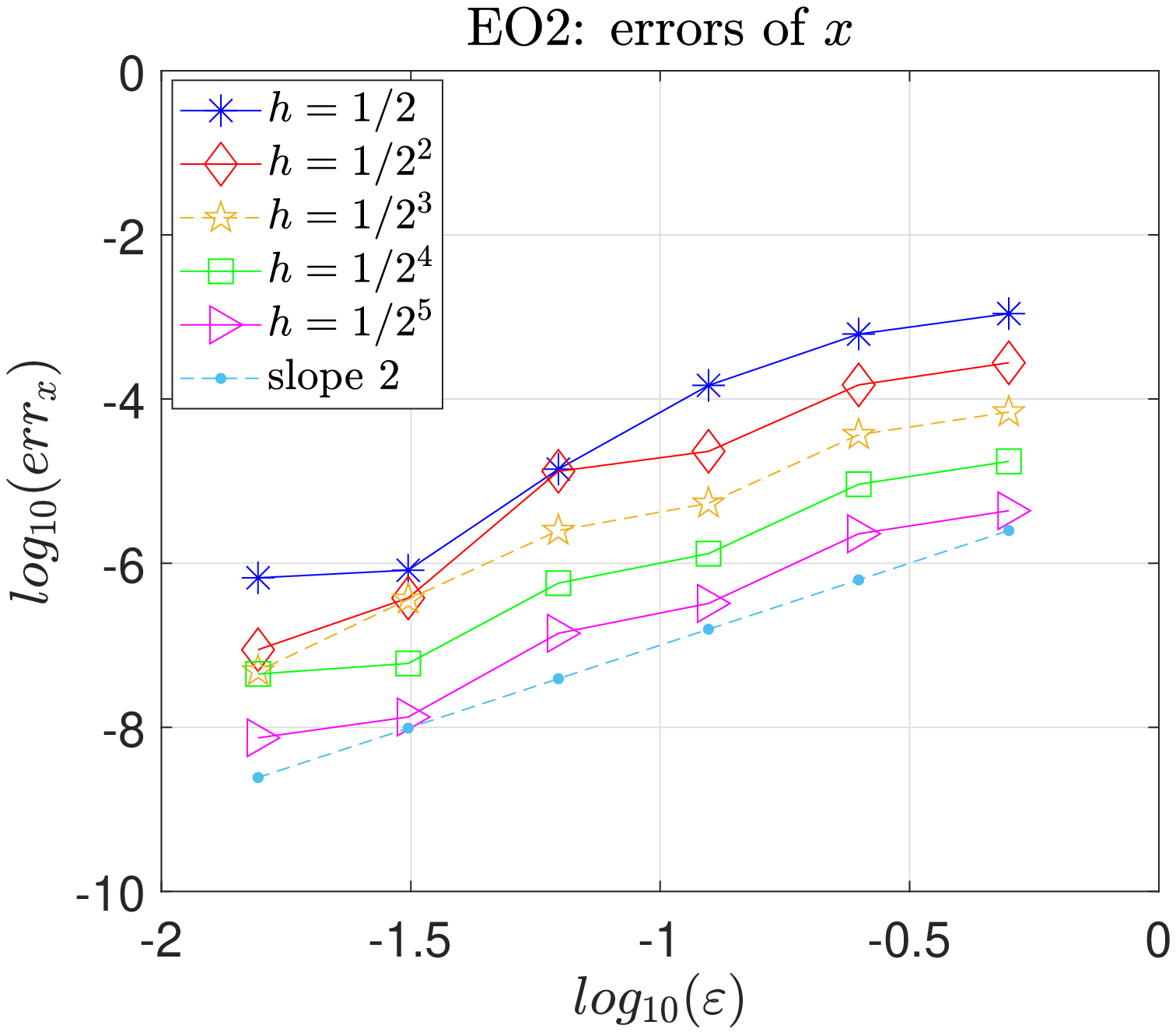,height=4.0cm,width=4.8cm}
\psfig{figure=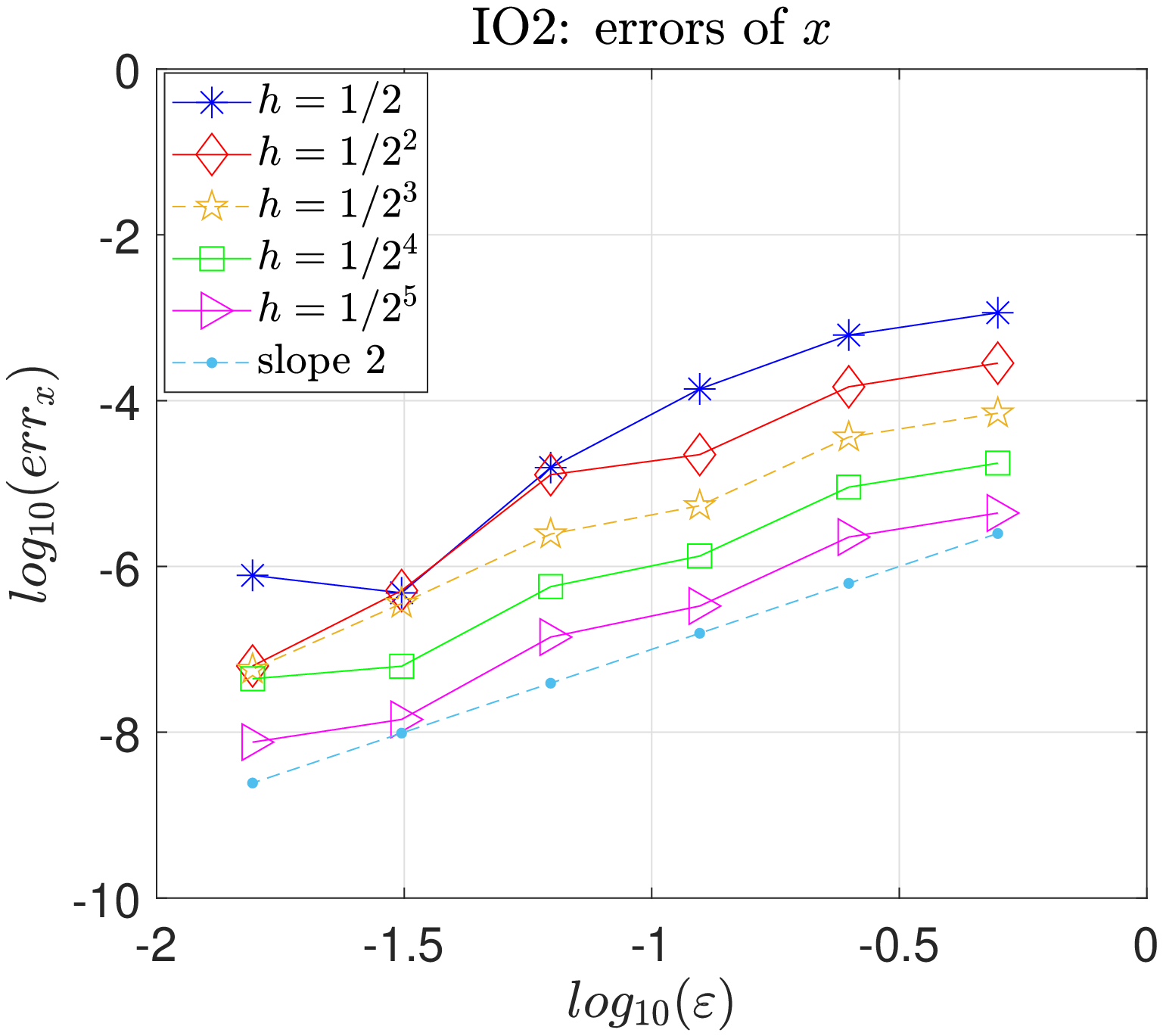,height=4.0cm,width=4.8cm}\\
\psfig{figure=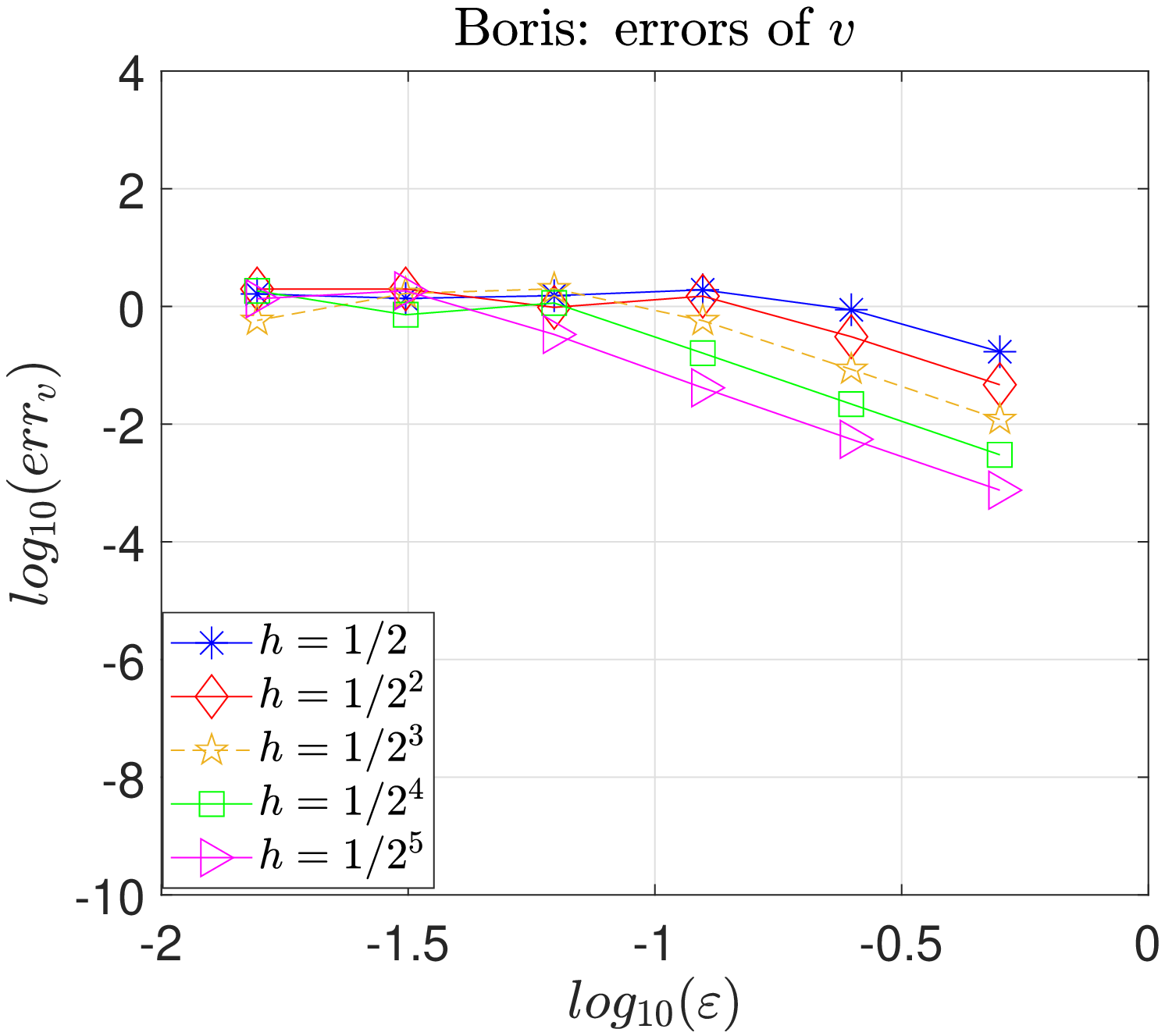,height=4.0cm,width=4.8cm}
\psfig{figure=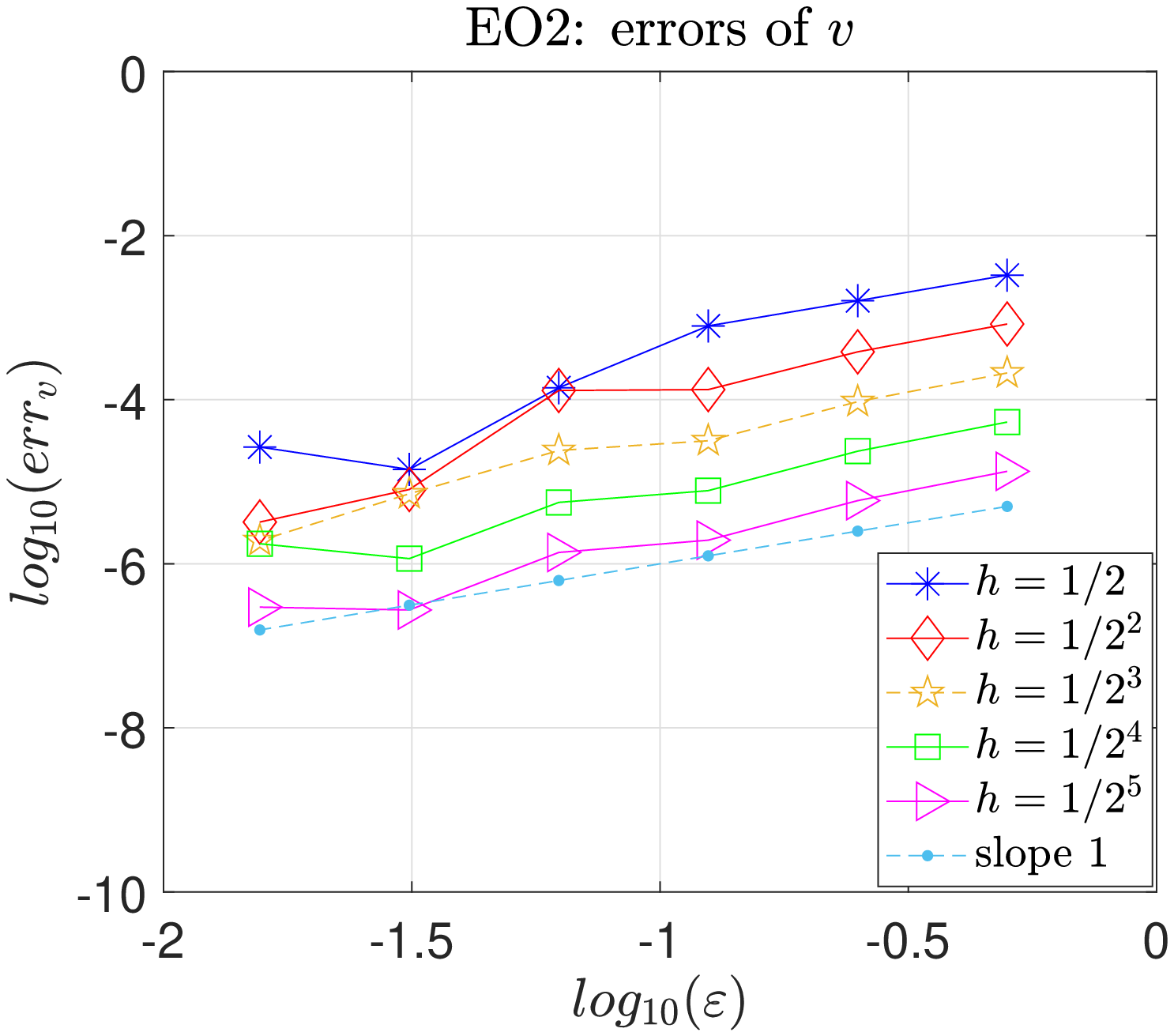,height=4.0cm,width=4.8cm}
\psfig{figure=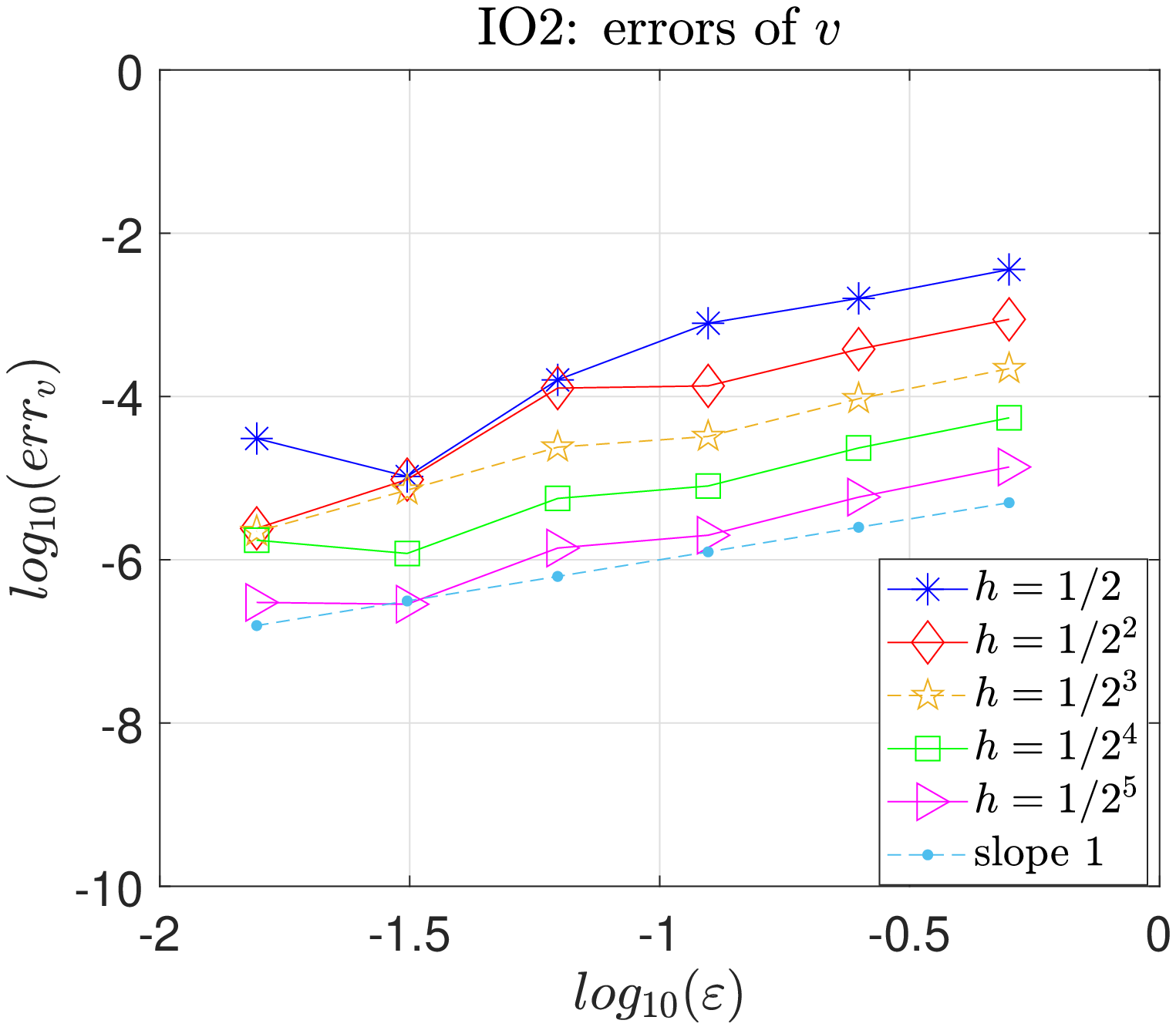,height=4.0cm,width=4.8cm}\\
\psfig{figure=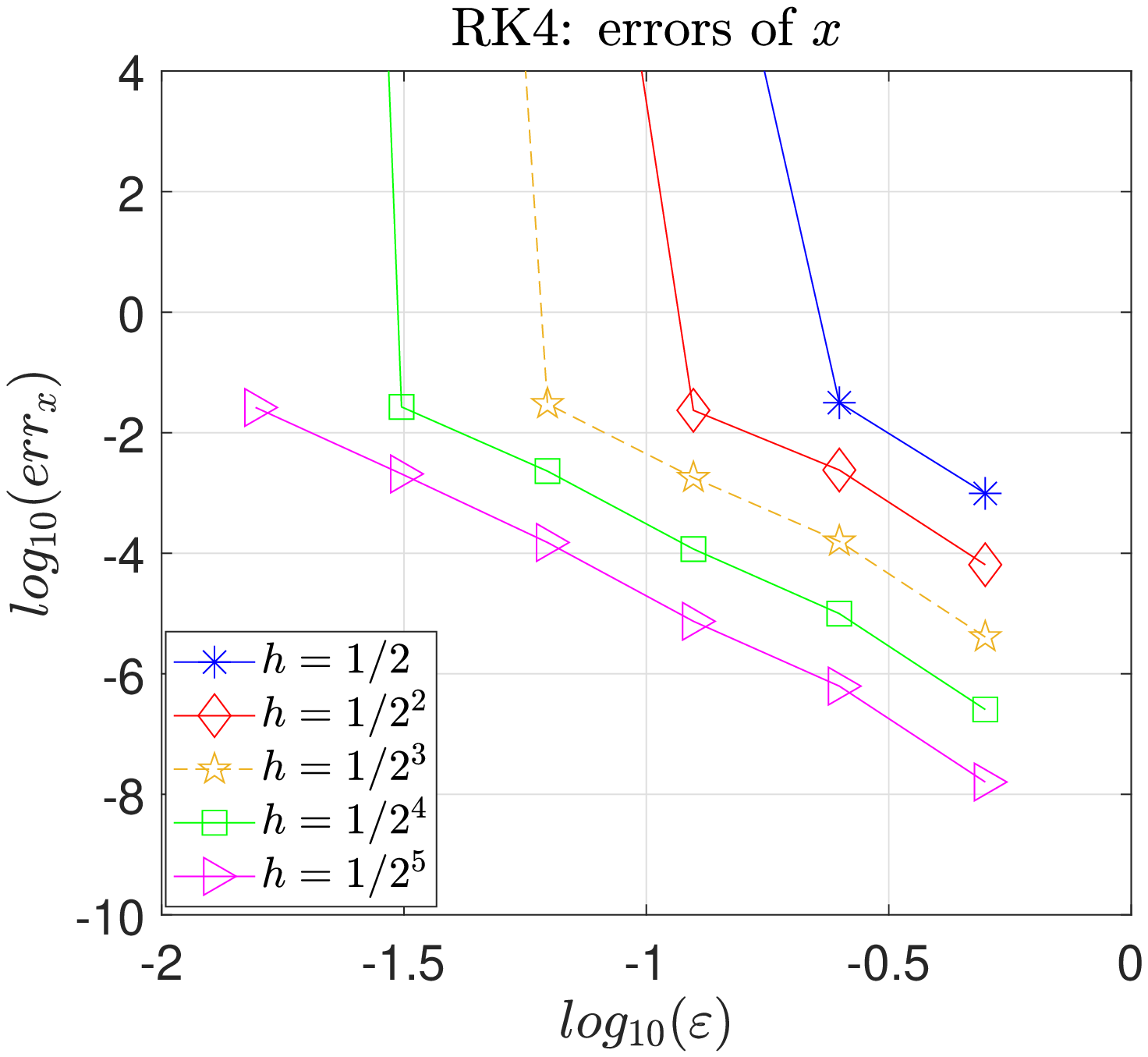,height=4.0cm,width=4.8cm}
\psfig{figure=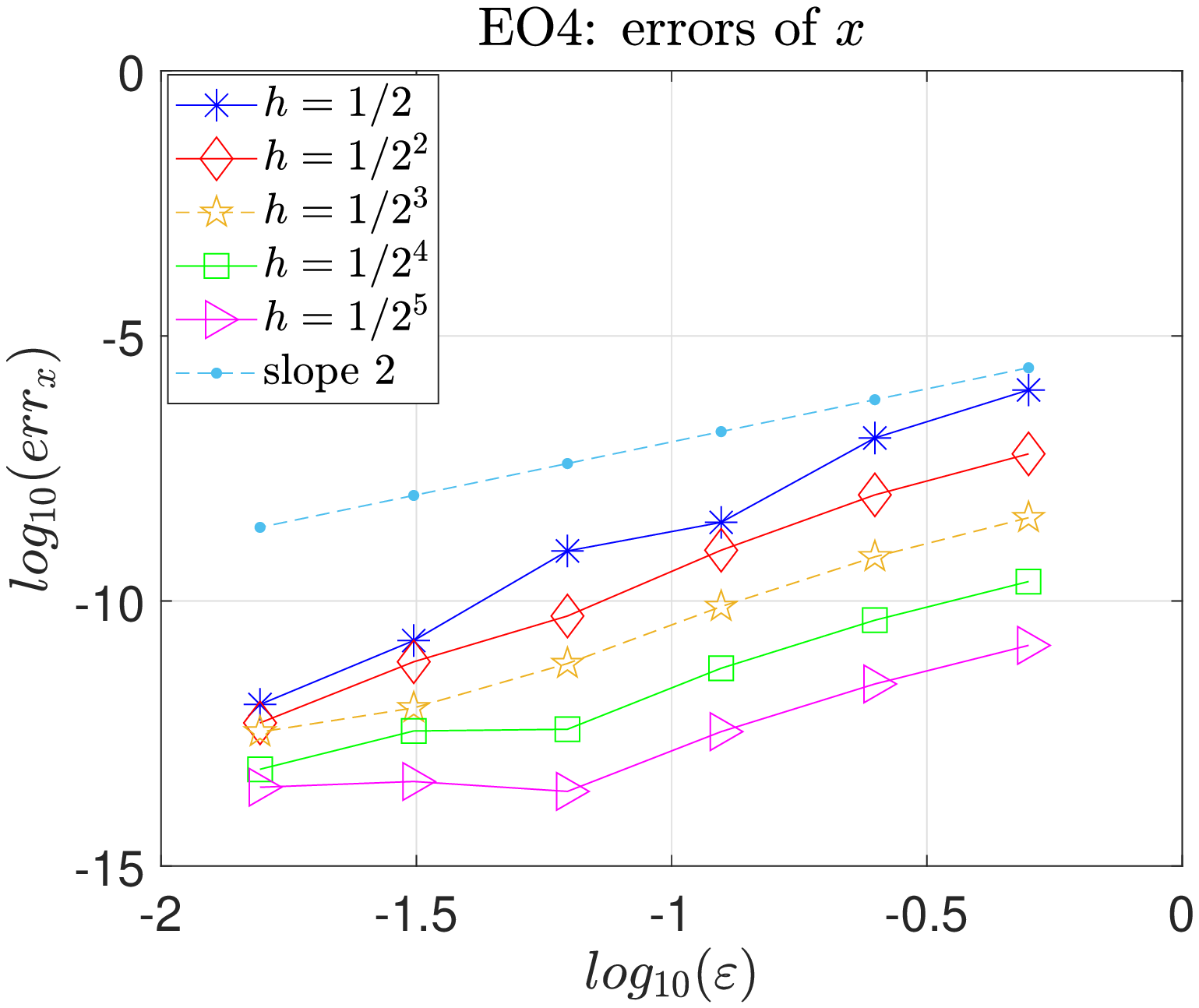,height=4.0cm,width=4.8cm}
\psfig{figure=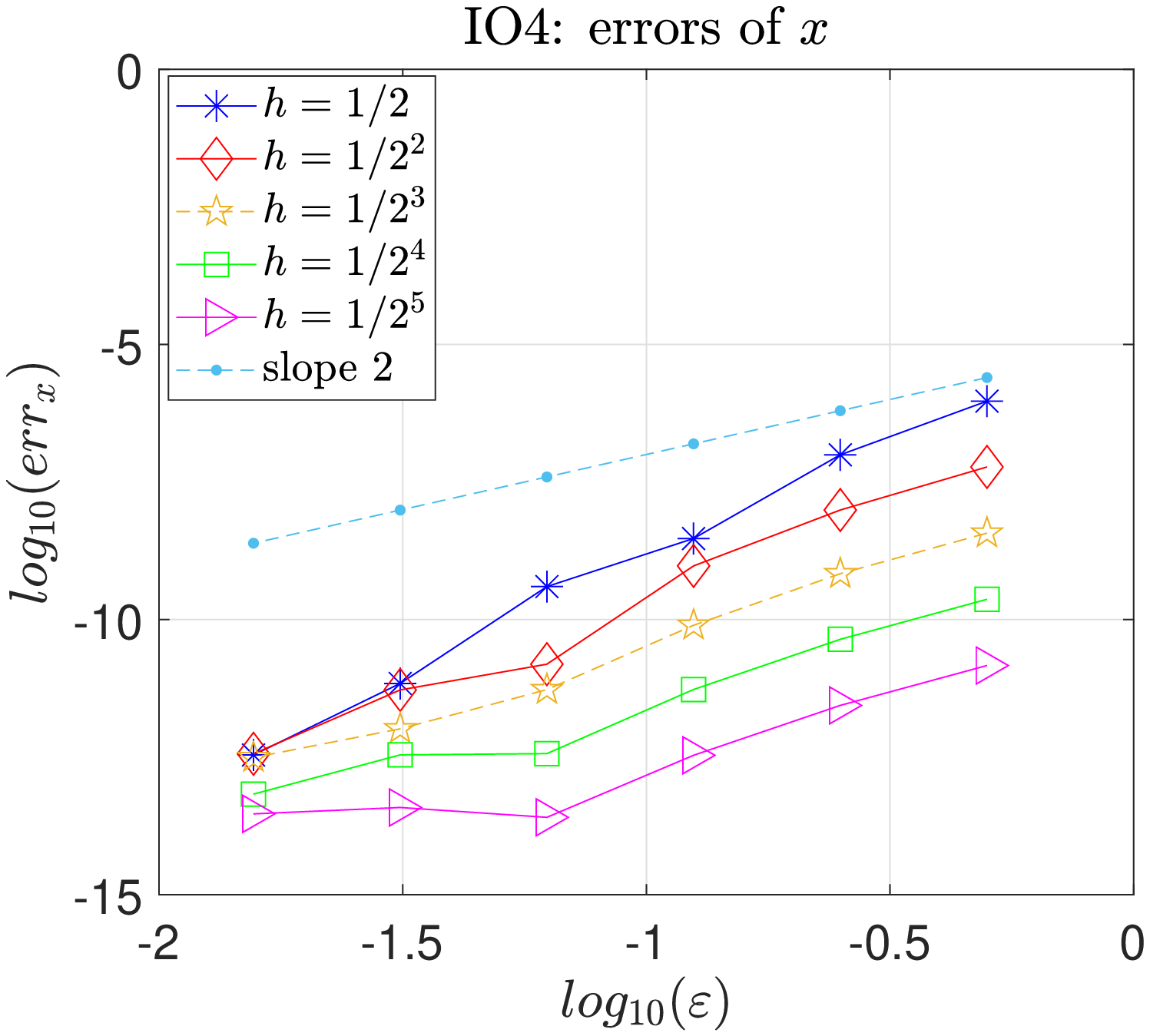,height=4.0cm,width=4.8cm}\\
\psfig{figure=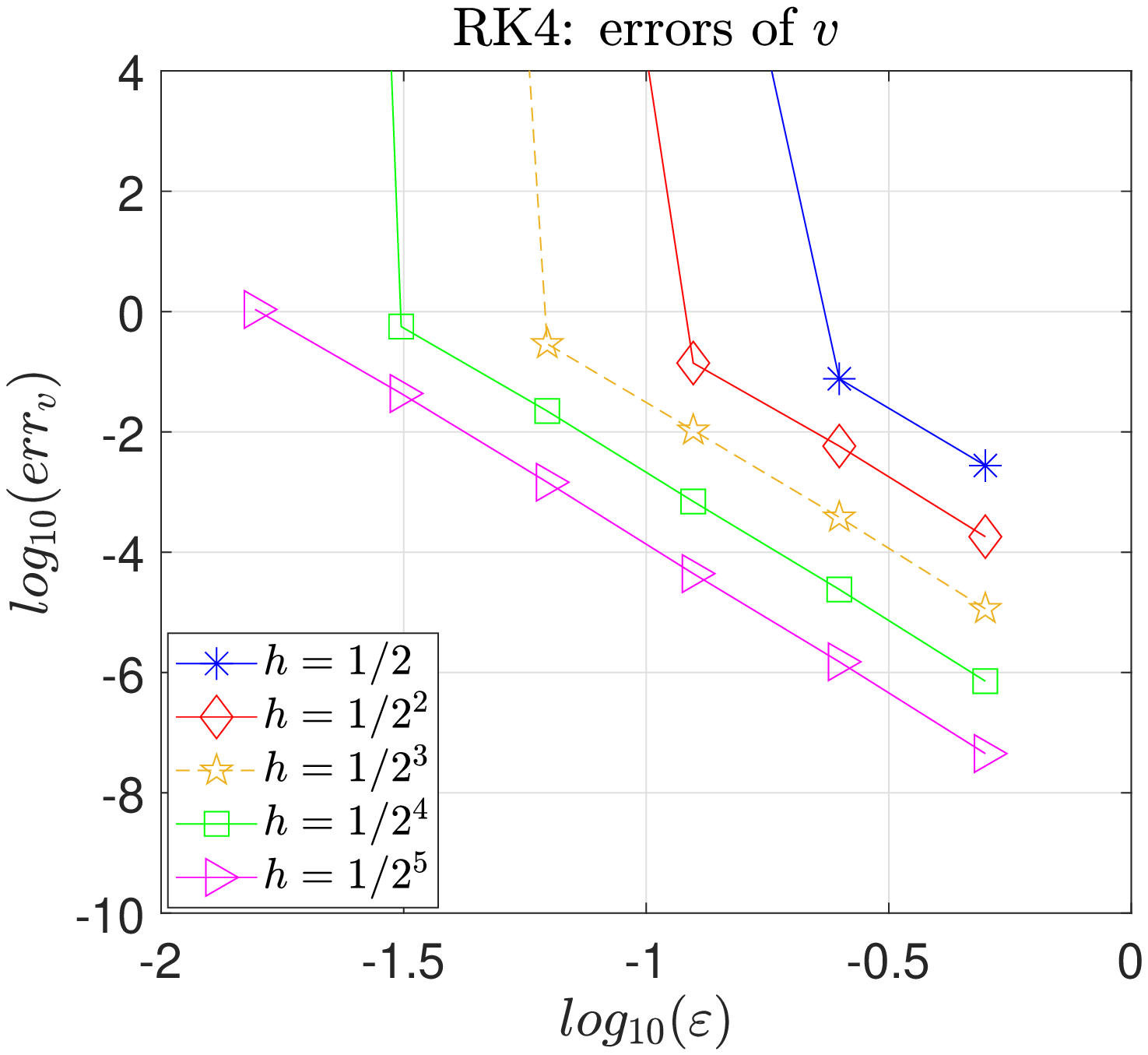,height=4.0cm,width=4.8cm}
\psfig{figure=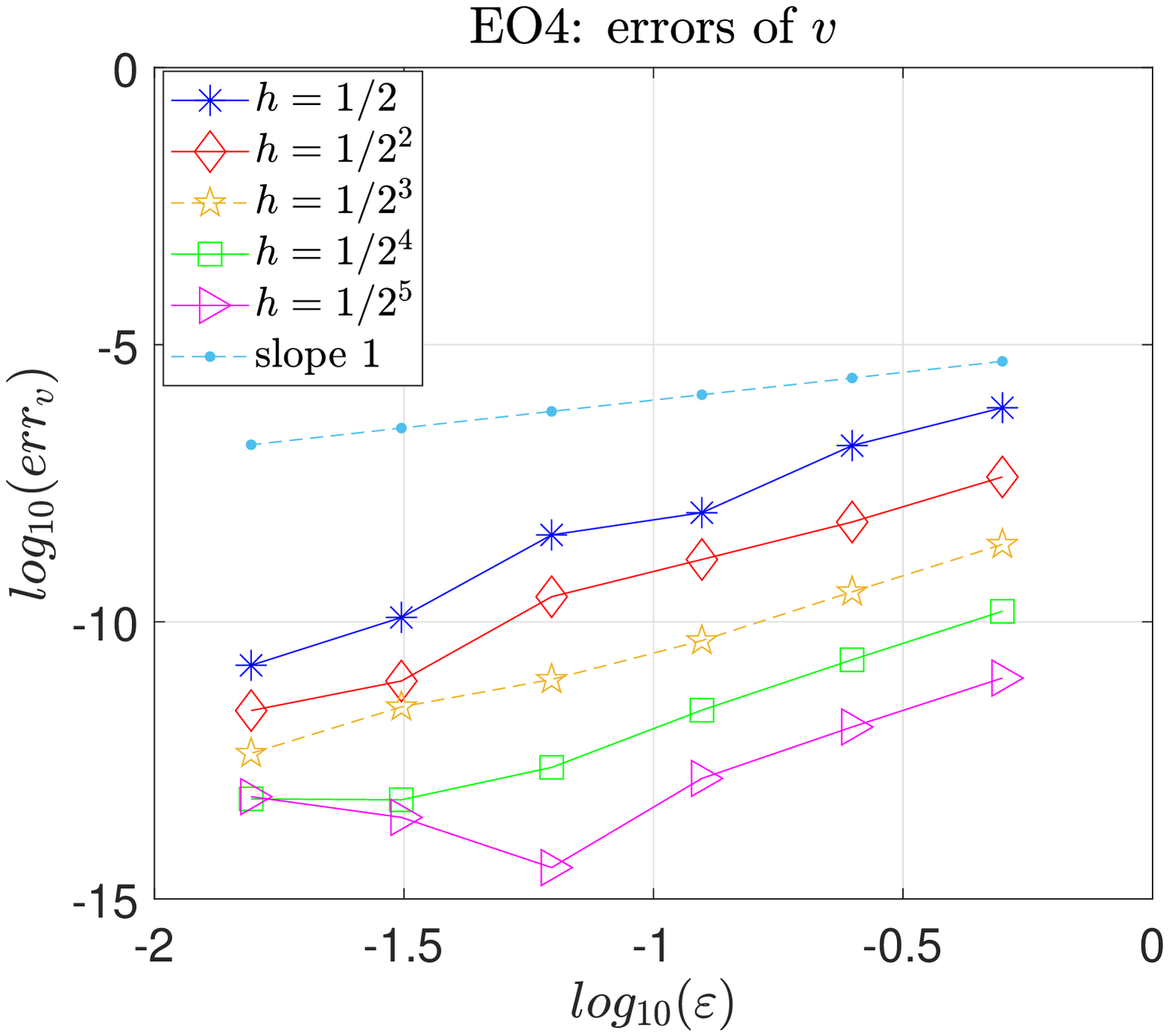,height=4.0cm,width=4.8cm}
\psfig{figure=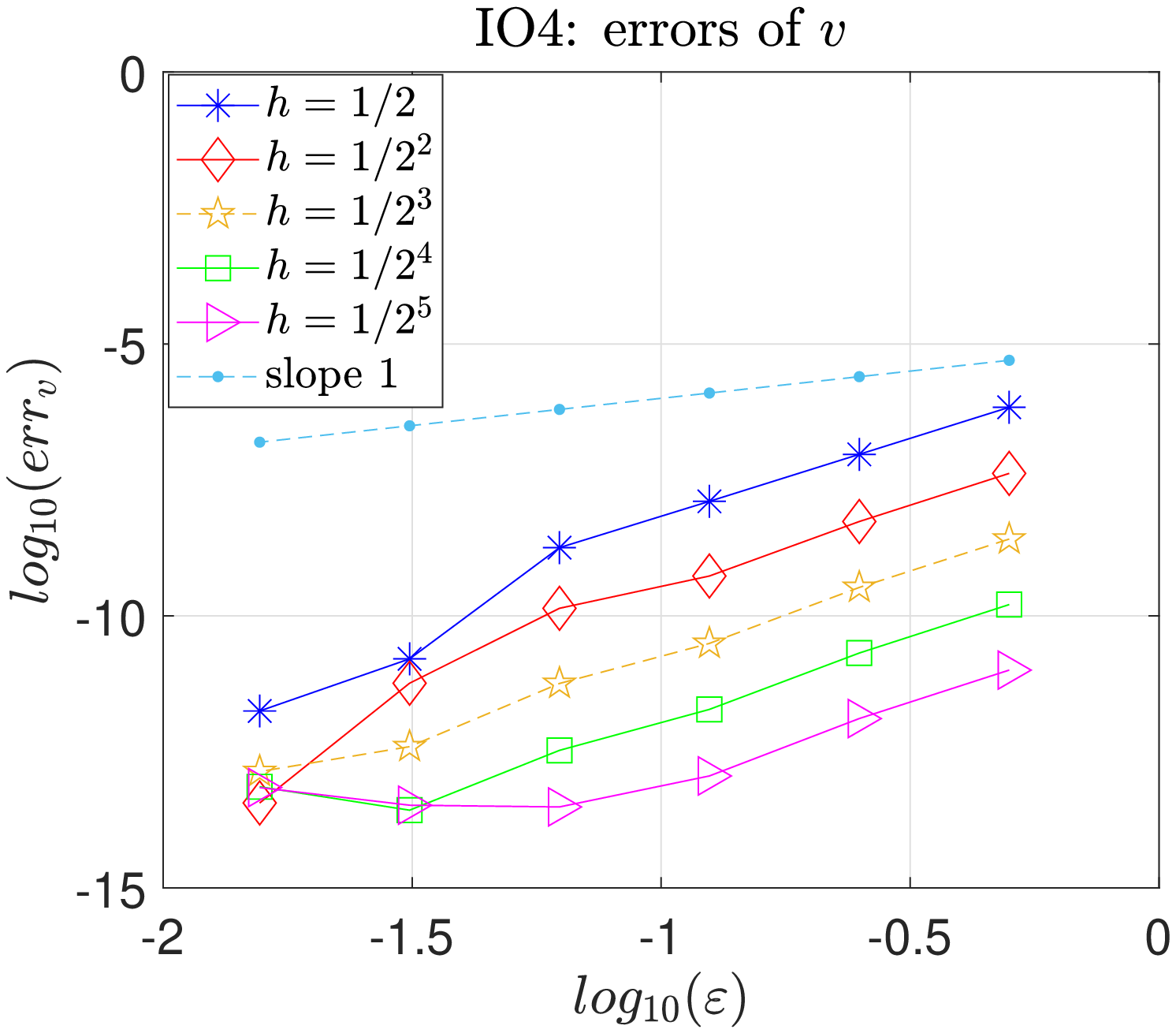,height=4.0cm,width=4.8cm}
\end{array}$$
\caption{The errors (\ref{errx})   at $t=1$ of the second-order schemes (top two rows) and fourth-order schemes (below two rows)  with $\eps=1/2^k$ for $k=1,2,\ldots,6$ under different $h$.}\label{fig22}
\end{figure}
Let us illustrate the performance of our schemes with a single particle in two space dimensions under a strong magnetic field \cite{VP1}:
$$
  \dot{x}(t)=v(t), \ \ \dot{v}(t)=\frac{b(x)}{\eps}Jv(t)+g(x(t)),  \quad t>0,
$$
where $g(x)=(\cos(q(1)/2)\sin(q(2))/2,\sin(q(1)/2)\cos(q(2)))^\intercal$ and $b(x)=1+\sin(q(1))\sin(q(2))$.
This is a reduced model from the three dimensional CPD case \ when the external magnetic field has a fixed direction and is homogenous in space. We choose the initial value $x(0)=(0.1,0.1)^\intercal,\,v(0)=(0.2,0.1)^\intercal$, and fix $N_\tau=2^6$ in the computations. For comparison, we choose   the well known Boris method (with second-order accuracy) denoted by Boris and a Runge-Kutta method (the fourth order Gauss-Legendre method) denoted by RK4.
Figure \ref{fig21} displays the numerical errors    \begin{equation}\label{errx}err_x:=\frac{\norm{x^n-x(t_n)}}{\norm{x(t_n)}},\ \ \ \ err_v:=\frac{\norm{v^n-v(t_n)}}{\norm{v(t_n)}}.\end{equation} against $h$ for different $\eps$. These results show that Boris, EO2, IO2 perform second order and  RK4, EO4, IO4 display fourth order.  In order to show the influence of $\eps$ on the accuracy, we present the errors $err_x$ and $err_v$ against $\eps$ for different $h$ in Figure \ref{fig22}. In the light of these results, we have the following observations.
  The four integrators  formulated in this paper have   improved uniformly high  accuracy  in both  position and velocity,  and when $\eps$  decreases, the accuracy is improved.  However, for the methods Boris and RK4, they do not have such optimal accuracy. The accuracy of these two methods becomes worse as $\eps$  decreases.

\section{Application to the three dimensional CPD}\label{sec:5}
 For the three dimensional CPD \eqref{charged-particle 3d}, it is noted that we cannot take the approach given for the  two dimensional case, since in that way,
$s  \varphi_1(s B(x(0) ))$ is no longer  periodic and thus two-scale exponential  integrator cannot be used. For the uniformly accurate (UA) methods for solving three dimensional CPD, some novel algorithms have been recently proposed in
\cite{Zhao}.
Based on the approach given in this paper for the two dimensional case, we can formulate a kind of UA methods with more simple scheme for the  three dimensional CPD \eqref{charged-particle 3d} in maximal ordering case. We state the application as follows.

For the three dimensional CPD \eqref{charged-particle 3d} in maximal ordering case  \cite{scaling1,lubich19,scaling2}, i.e., $B=B(\eps x):=(b_1(\eps x),b_2(\eps x),b_3(\eps x))^\intercal\in \RR^3$, we first rewrite it as
\begin{equation}\label{charged-particle2 3d}
  \dot{x}(t)=v(t), \ \ \dot{v} (t)= \frac{\widehat B_0}{\eps} v(t)+F(x(t),v(t)),\ \ 0<t\leq T, \ \
     x(0)=x_0\in\RR^3,\quad v(0)=v_0\in\RR^3,
\end{equation}
where $\widehat B_0=\widehat B(\eps x(0))$ with $\widehat B (\eps x)= \begin{pmatrix}
                     0 & b_3(\eps x) & -b_2(\eps x) \\
                     -b_3(\eps x) & 0 & b_1(\eps x) \\
                     b_2(\eps x) & -b_1(\eps x) & 0 \\
                  \end{pmatrix}
$ and $F(x(t),v(t))=\frac{\widehat B(\eps x(t))-\widehat B_0}{\eps} v(t) +E(x(t)).$
With the maximal ordering property, it is worth noticing that $F(x(t),v(t))$ is uniformly bounded w.r.t. $\eps$.
We introduce the filtered variable $w(t)=\varphi_0(-t \widehat B_0/\eps)v(t)$, then (\ref{charged-particle2 3d}) reads:
\begin{equation}\label{H2 model problem}  \dot{x}(t)=\varphi_0(t \widehat B_0/\eps)w(t), \quad
\dot{w}(t)= \varphi_0(-t \widehat B_0/\eps) F\big(x(t),\varphi_0(t \widehat B_0/\eps)w(t)\big),\quad \ x(0)=x_0,\quad
   w(0)=   v_0.
\end{equation}
With the help of $\widehat B_0$,  $\fe^{\pm t \widehat B_0/\eps}$ is periodic in $t/\eps$ on  $[0,2\pi]$.
By isolating the fast time variable $t/\epsilon $ as another variable $\tau$ and denoting
 $X(t,\tau)=x(t),\ W(t,\tau)=w(t),$
the two-scale system of \eqref{H2 model problem} takes the form
\begin{equation}\label{2scalenew}\left\{\begin{split}
  &\partial_tX(t,\tau)+\frac{1}{\epsilon }\partial_\tau X(t,\tau)=\varphi_0( \tau \widehat B_0 )W(t,\tau),\\
    &\partial_tW(t,\tau)+\frac{1}{\epsilon }\partial_\tau W(t,\tau)=\varphi_0( -\tau \widehat B_0 )F\big( X(t,\tau), \varphi_0( \tau \widehat B_0 )W(t,\tau)\big).
    \end{split}\right.
\end{equation}
 The initial data $[X^{0};W^{0}]:=[X(0,\tau);W(0,\tau)]$ for
(\ref{2scalenew}) is obtained by \eqref{inv} with the replacement $f_\tau$ \eqref{ftau} of
$$ f_\tau([X;W])=\left(
                      \begin{array}{c}
                      \varphi_0( \tau \widehat B_0 )W\\
                      \varphi_0( -\tau \widehat B_0 )F( X, \varphi_0( \tau \widehat B_0)W)\\
                      \end{array}
                    \right).$$

We now obtain the semi-discretization of the three dimensional CPD \eqref{charged-particle 3d} in maximal ordering case.

\begin{defi}\label{3Dmethod}
For the  three dimensional CPD  \eqref{charged-particle 3d} in maximal ordering case,   choose  a  time step $h$. Then for solving the  equation  \eqref{2scalenew} with the initial value $[X^{0};W^{0}]$,
consider an $s$-stage two-scale exponential integrator   \begin{equation*}
\begin{array}[c]{ll}%
X^{ni}&=\varphi_0(c_{i}h/\epsilon \partial_\tau)X^{n}+h\textstyle\sum\limits_{j=1}^{s}\bar{a}_{ij}(h/\epsilon \partial_\tau)  \varphi_0(\tau \widehat B_0 )W^{nj} ,\qquad\qquad\qquad\qquad\qquad
i=1,2,\ldots,s,\\
W^{ni}&=\varphi_0(c_{i}h/\epsilon \partial_\tau)W^{n}+h\textstyle\sum\limits_{j=1}^{s}\bar{a}_{ij}(h/\epsilon \partial_\tau)\varphi_0(- \tau \widehat B_0 )F\big( X^{nj},\varphi_0( \tau\widehat B_0 )W^{nj}\big),\quad\ \
i=1,2,\ldots,s,\\
X^{n+1}&=\varphi_0(h/\epsilon \partial_\tau)X^{n}+ h\textstyle\sum\limits_{j=1}^{s}\bar{b}_{j}(h/\epsilon \partial_\tau) \varphi_0(\tau \widehat B_0 )W^{nj},\\
W^{n+1}&=\varphi_0(h/\epsilon \partial_\tau)W^{n}+  h\textstyle\sum\limits_{j=1}^{s}\bar{b}_{j}(h/\epsilon \partial_\tau)\varphi_0(- \tau \widehat B_0 )F\big( X^{nj},\varphi_0( \tau \widehat B_0 )W^{nj}\big),
\end{array}\end{equation*}
with the coefficients $c_i\in[0,1]$, $\bar{a}_{ij}(h/\epsilon \partial_\tau)$ and
$\bar{b}_{j}(h/\epsilon \partial_\tau)$.
 The numerical solution $x^{n+1}\approx x(t_{n+1})$ and $v^{n+1}\approx v(t_{n+1})$ of \eqref{charged-particle 3d} is given by  \begin{equation*} \begin{aligned}
&x^{n+1}=  X^{n+1},\qquad \
 v^{n+1}=  \varphi_0(  t_{n+1} \widehat B_0/\epsilon )  W^{n+1}.
 \end{aligned}
\end{equation*}

\end{defi}
Based on the Fourier pseudospectral method in $\tau$,  the full-discretization  can be formulated by using the same way as that of two dimensional CPD. For simplicity, we do not go further on this point here.
For the semi-discretization, it has a uniform accuracy and we state it as follows.
\begin{theo} Under the conditions of Theorem \ref{UA thm}, for the final numerical solution   $x^n,  v^n$  produced
by the method  given in Definition \ref{3Dmethod}, the global error is \begin{equation*}
\begin{aligned}
\norm{x^n-x(t_n)}+\norm{  v^n-  v(t_n)} \leq C  h^r,\qquad 0\leq n\leq T/h,\\
\end{aligned}
\end{equation*}
where $C$ is independent of $n, h, \epsilon$.
\end{theo}

This result can be proved in a similar way as stated in Section  \ref{sec:2} and we skip it for brevity.

\vskip2mm\noindent \textbf{Numerical test.}
As an illustrative
numerical experiment, we consider  the  three dimensional CPD \eqref{charged-particle 3d}
with a strong magnetic field \cite{lubich19}
$$B(x,t) =\nabla \times \frac 1 \eps \,  \left(
                              \begin{array}{c}
                             0 \\
                             x_1\\
                              0 \\
                              \end{array}
                            \right)+\nabla \times  \,  \left(
                              \begin{array}{c}
                            0  \\
                            x_1 x_3 \\
                              0  \\
                              \end{array}
                            \right)
                            =\frac 1 \eps \,  \left(
                              \begin{array}{c}
                              0\\
                              0 \\
                              1 \\
                              \end{array}
                            \right)
 +  \left(
     \begin{array}{c}
      -  x_1 \\
       0  \\
      x_3 \\
     \end{array}
   \right),
 $$
 and $E(x,t)=-\nabla_x U(x)$ with the potential
$U(x)=\frac{1}{\sqrt{x_1^2+x_2^2}}.$
 The initial values are chosen as $x(0)=(\frac 1 3,\frac 1 4,\frac 1
2)^{\intercal}$ and $v(0)=(\frac 2 5,\frac 2 3,1)^{\intercal}$. We
solve this problem on  $[0,1]$  by the same methods of Section \ref{subNT} combined with the Fourier pseudospectral method ($N_\tau=2^6$).
The   errors of all the methods \begin{equation}\label{err}err:=\frac{\norm{x^n-x(t_n)}}{\norm{x(t_n)}}+\frac{\norm{v^n-v(t_n)}}{\norm{v(t_n)}}
\end{equation}are displayed in Figures \ref{fign}-\ref{fign2}. From these results, it follows that
   EO2 and IO2 show    uniform second order accuracy,    EO4 and IO4 have   uniform fourth order accuracy, but Boris and RK4  do not have such uniform accuracy.

\begin{figure}[t!]
$$\begin{array}{cc}
\psfig{figure=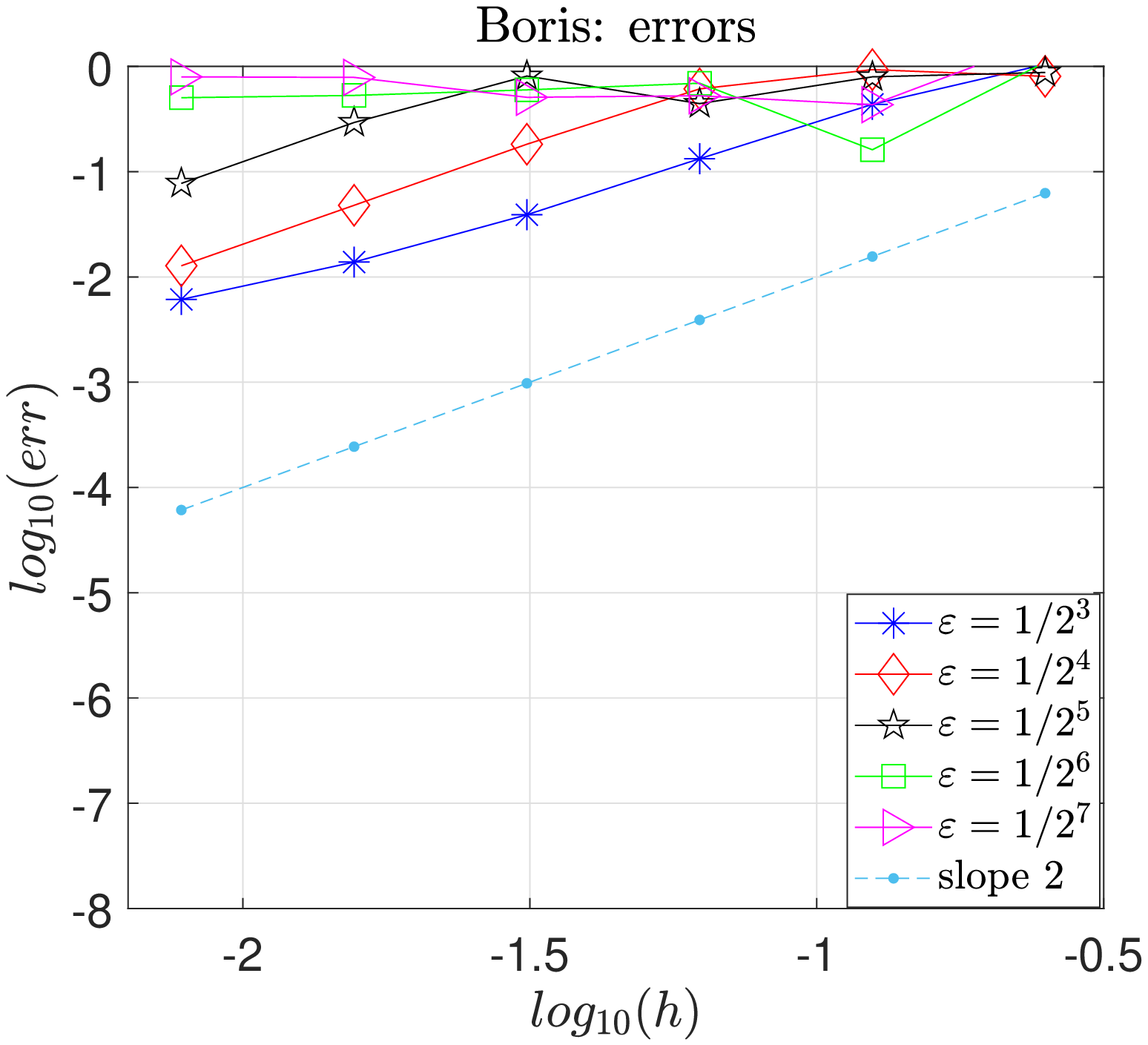,height=4.0cm,width=4.8cm}
\psfig{figure=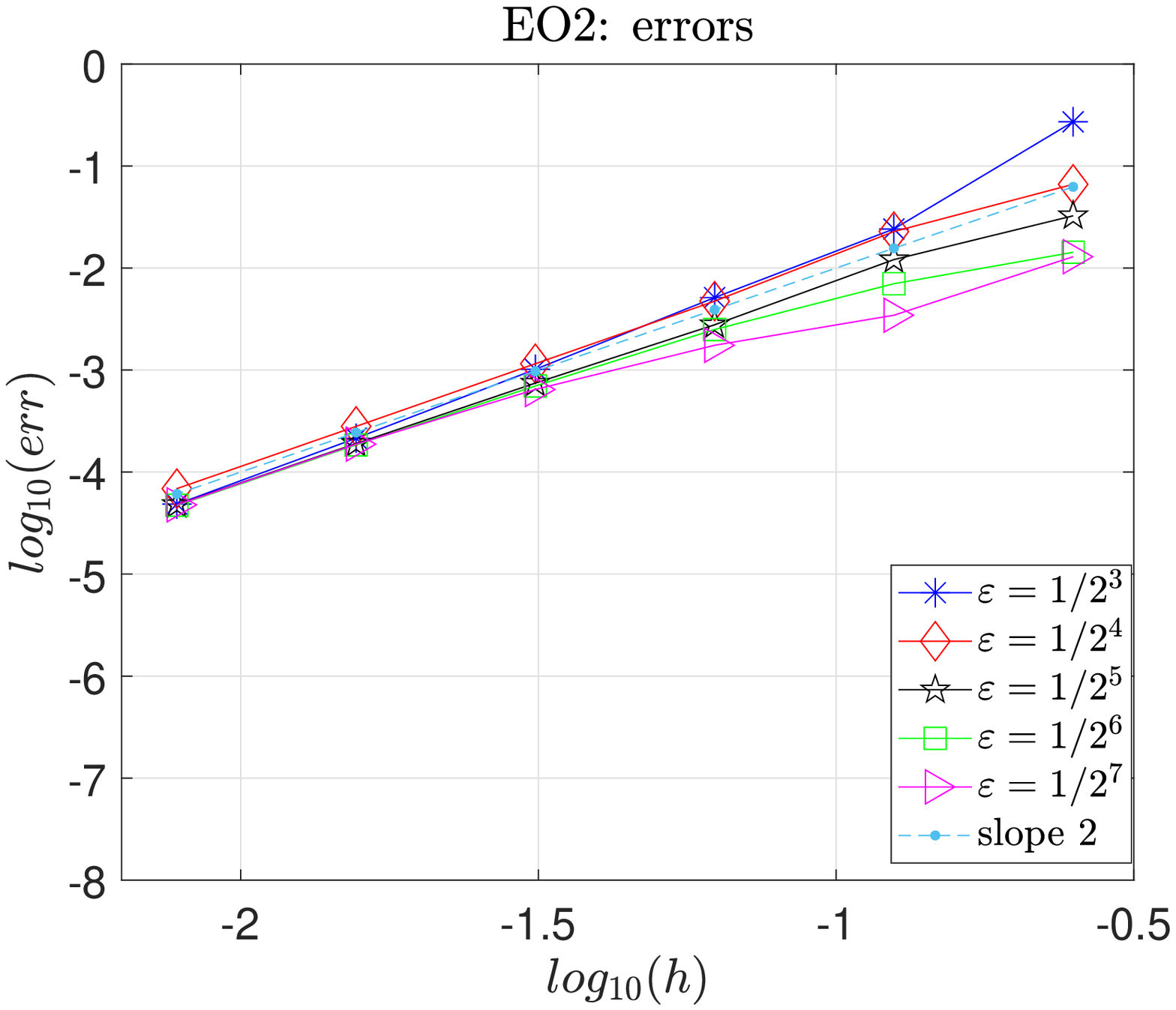,height=4.0cm,width=4.8cm}
\psfig{figure=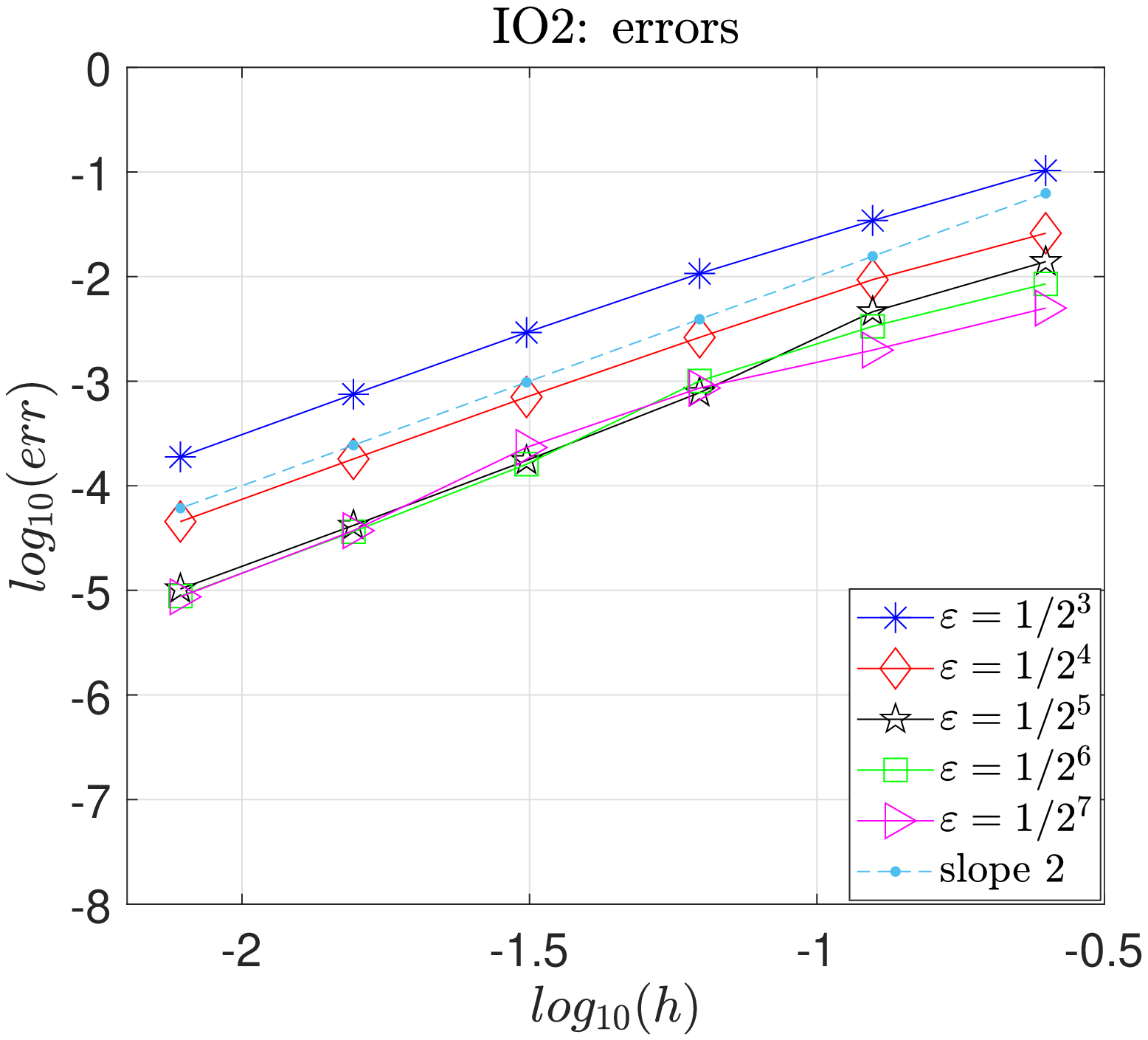,height=4.0cm,width=4.8cm}\\
\psfig{figure=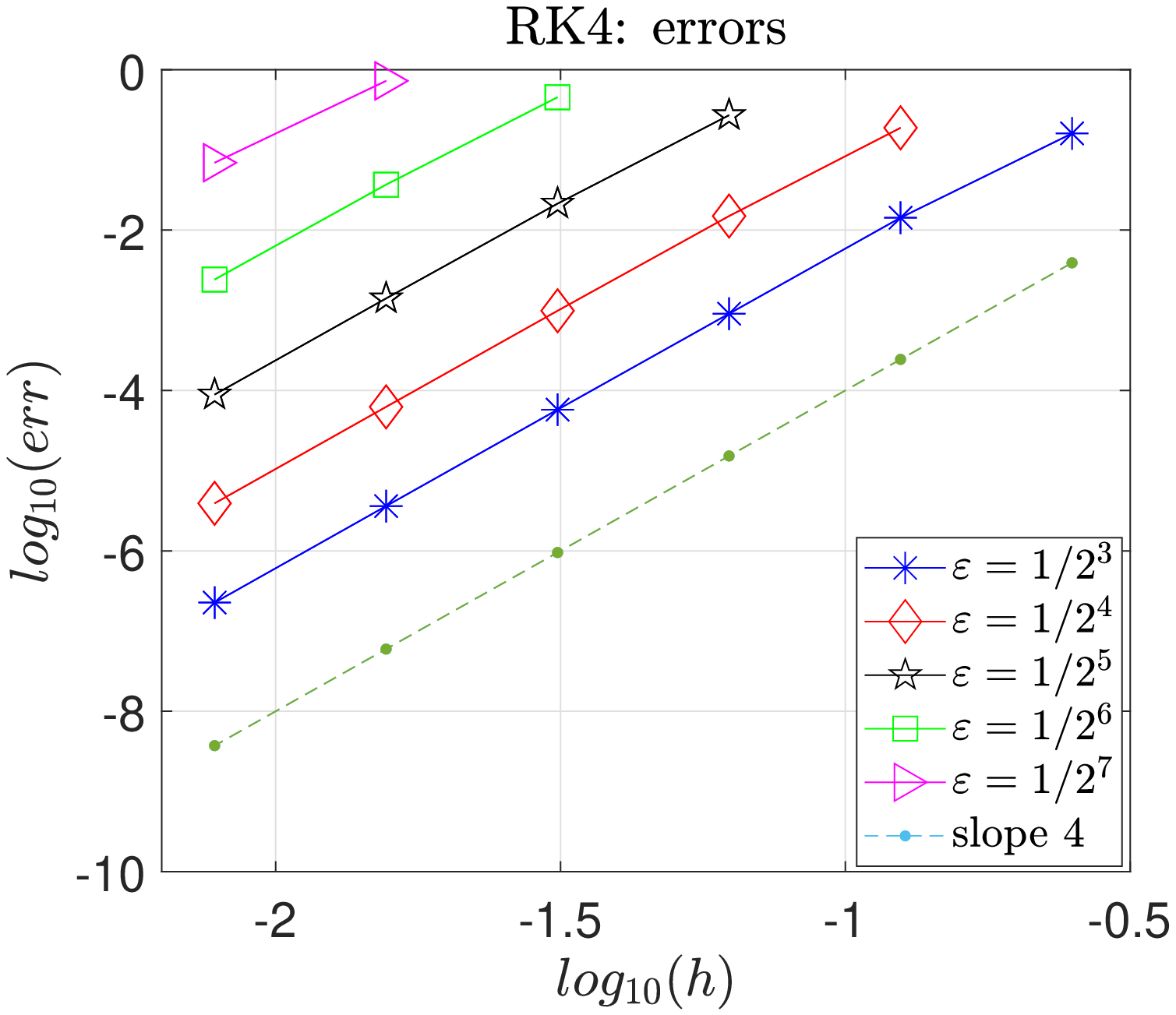,height=4.0cm,width=4.8cm}
\psfig{figure=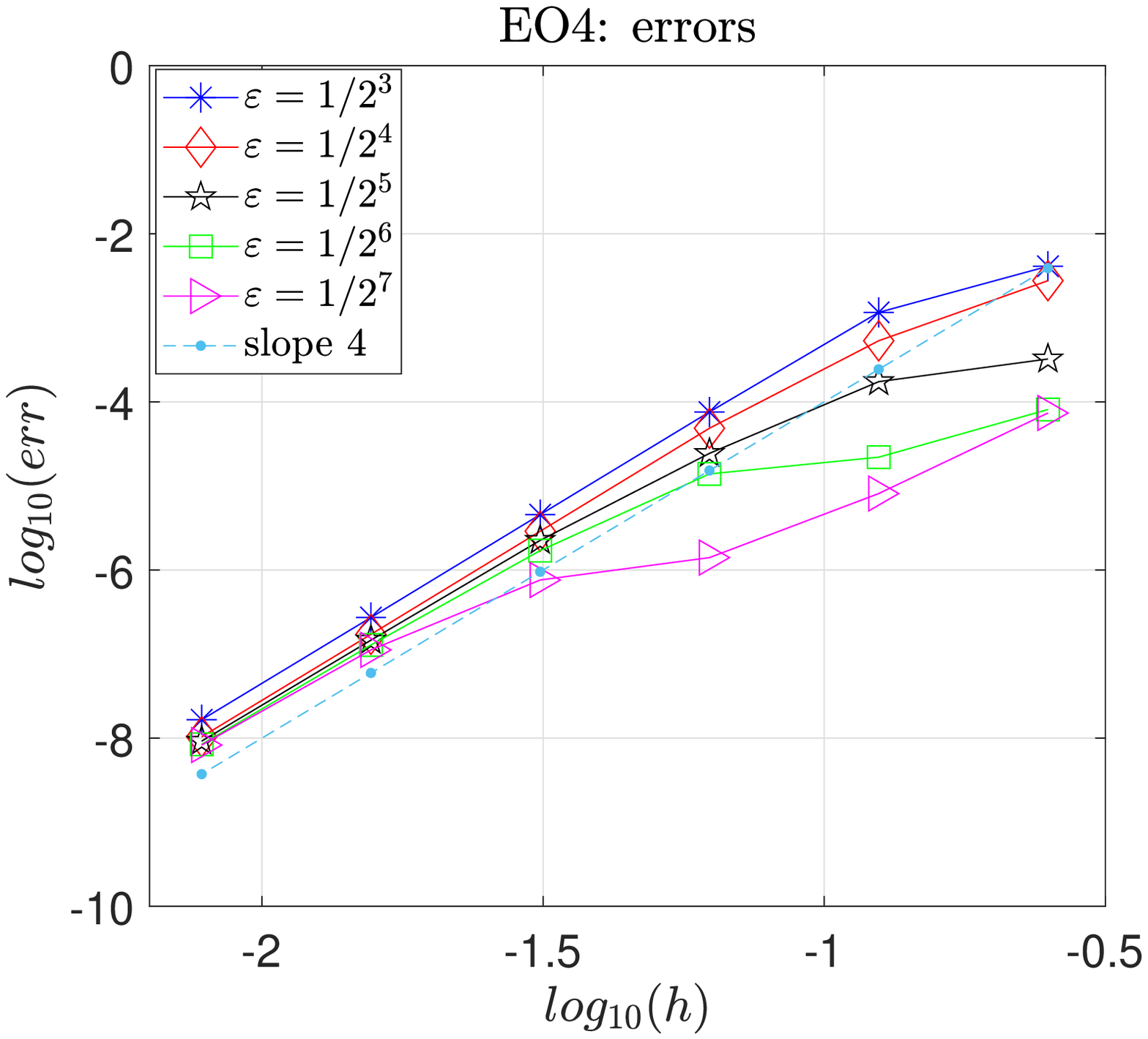,height=4.0cm,width=4.8cm}
\psfig{figure=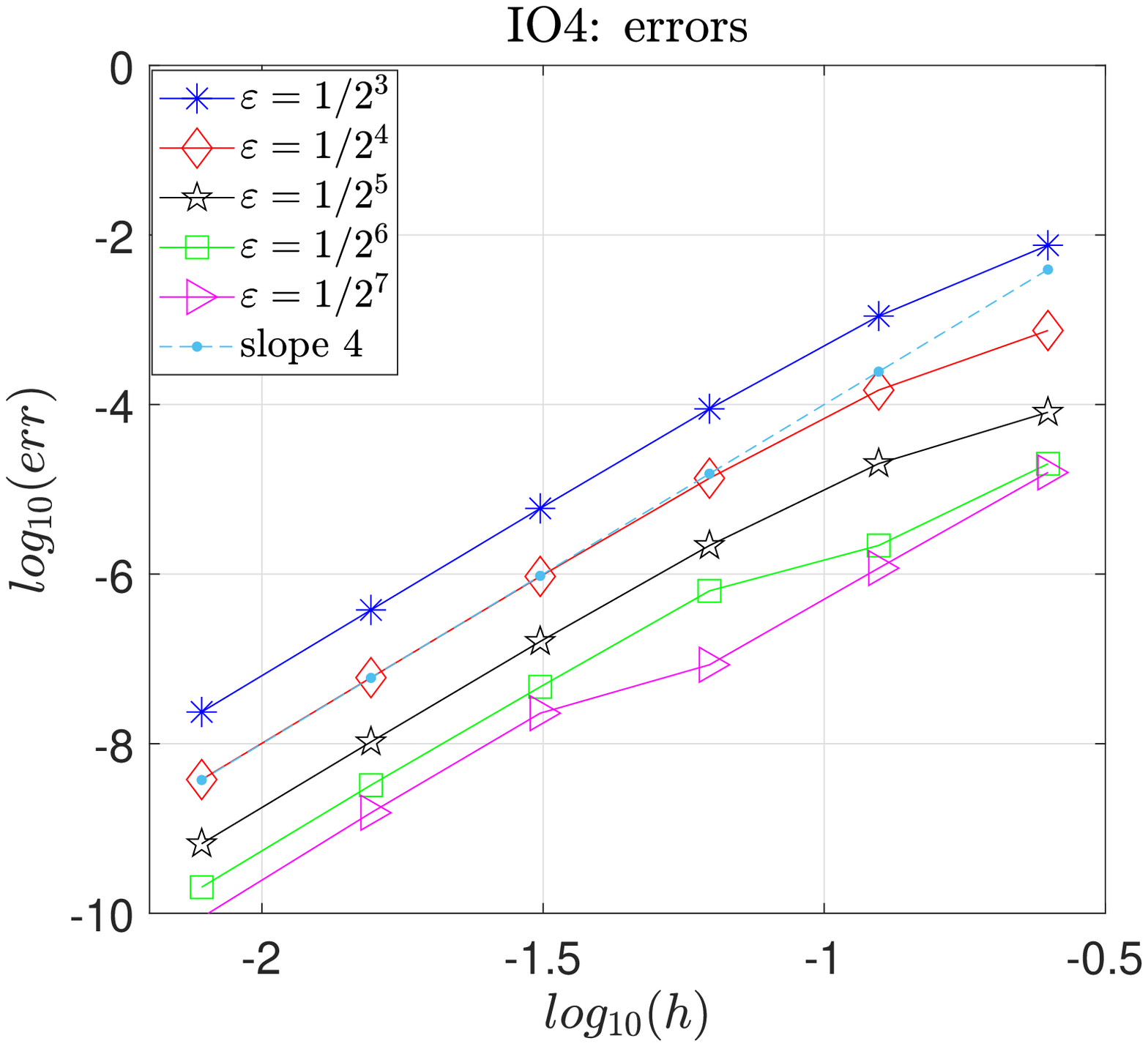,height=4.0cm,width=4.8cm}
\end{array}$$
\caption{The errors (\ref{err})  at $t=1$ of the second-order schemes (top row) and fourth-order schemes (below row)  with $h=1/2^k$ for $k=1,2,\ldots,6$ under different $\eps$.}\label{fign}
\end{figure}

\begin{figure}[t!]
$$\begin{array}{cc}
\psfig{figure=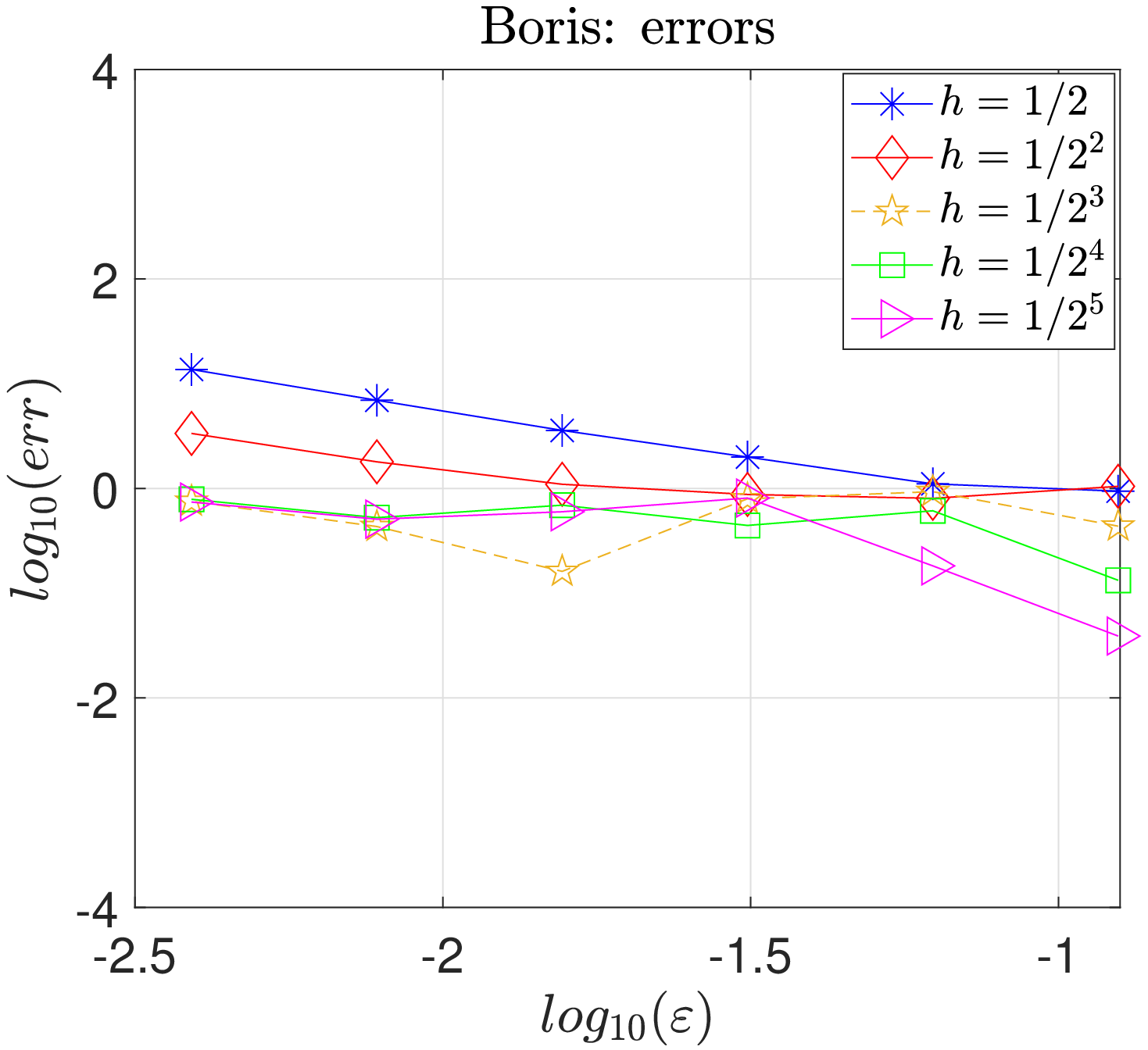,height=4.0cm,width=4.8cm}
\psfig{figure=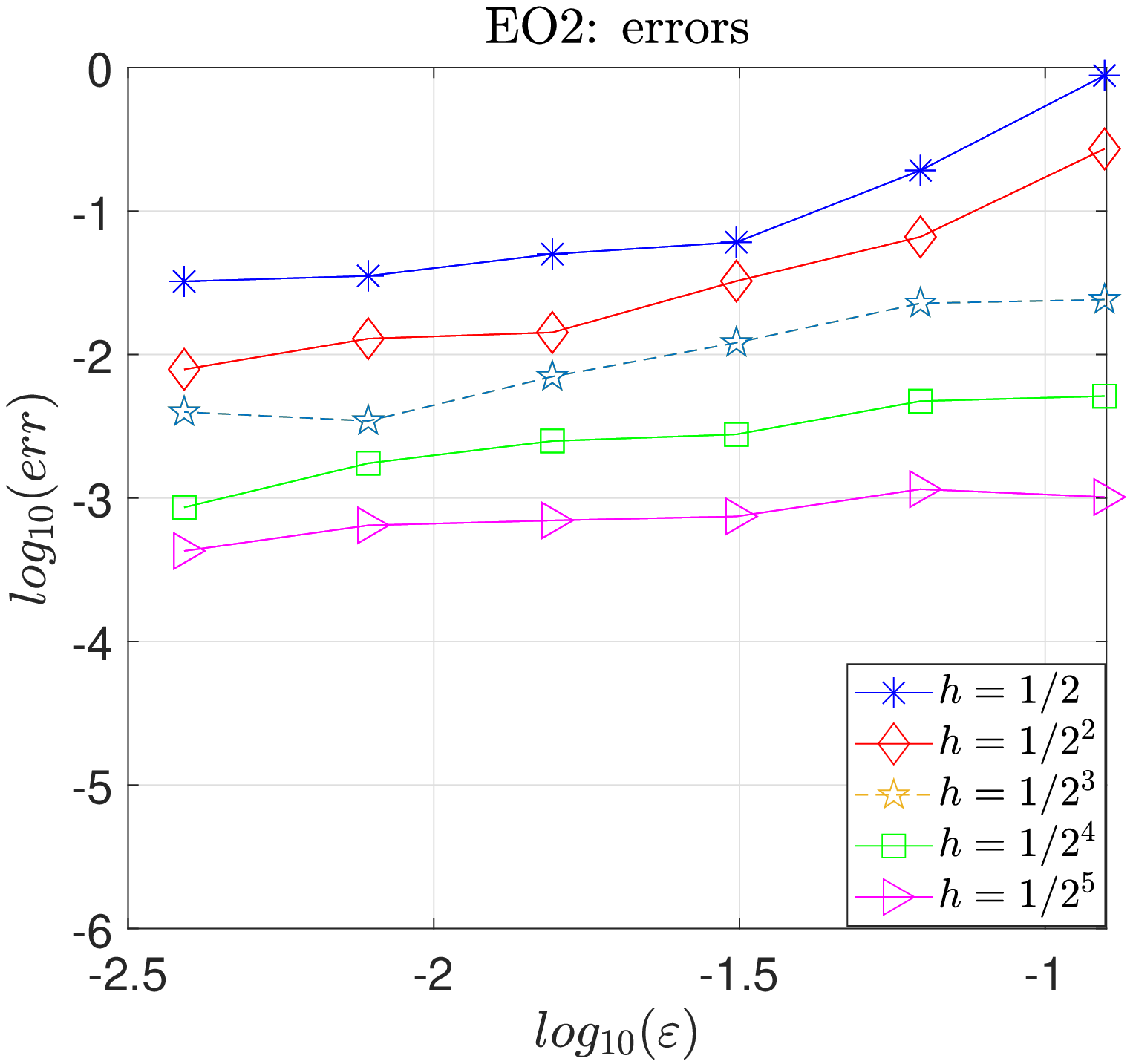,height=4.0cm,width=4.8cm}
\psfig{figure=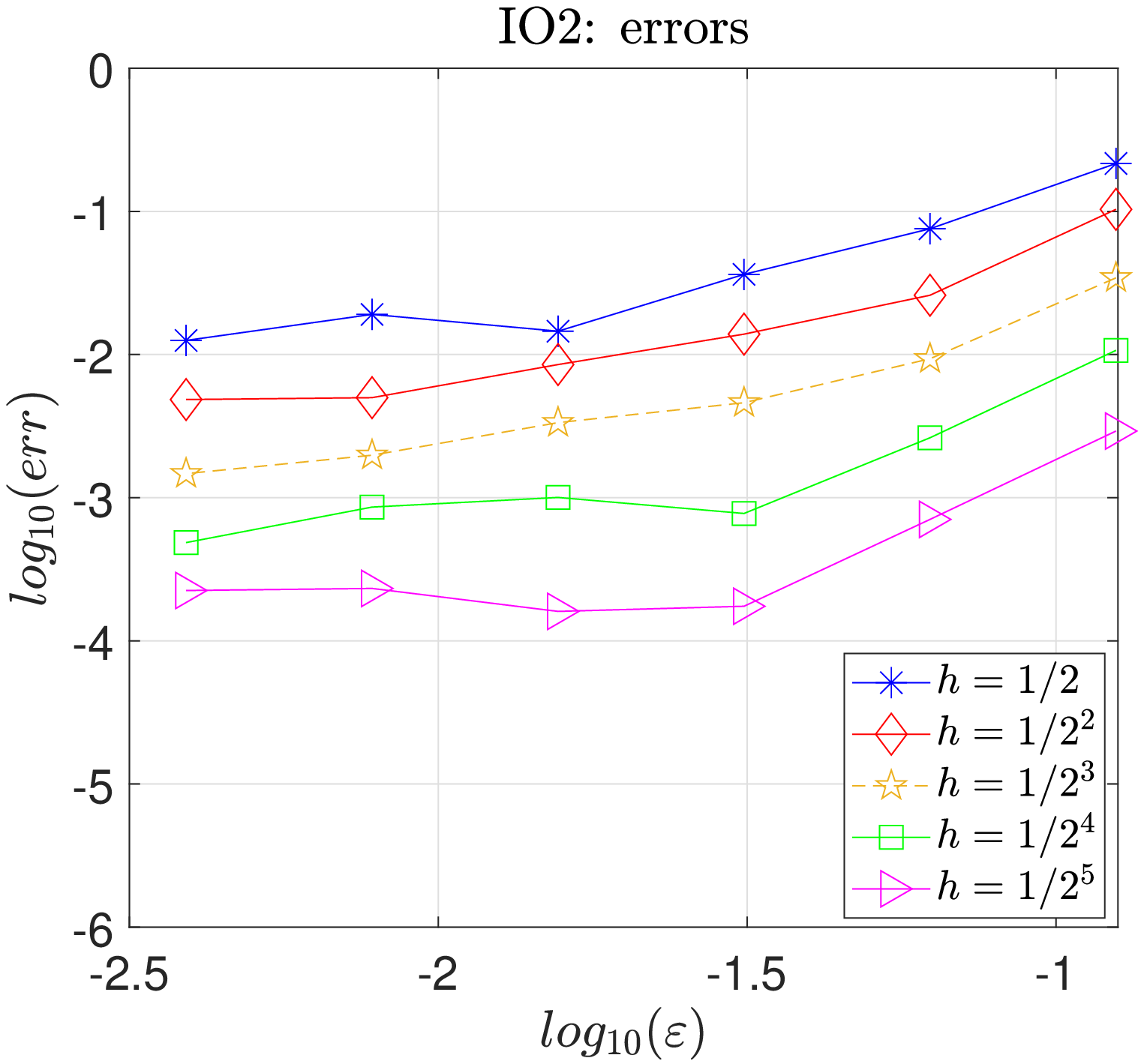,height=4.0cm,width=4.8cm}\\
\psfig{figure=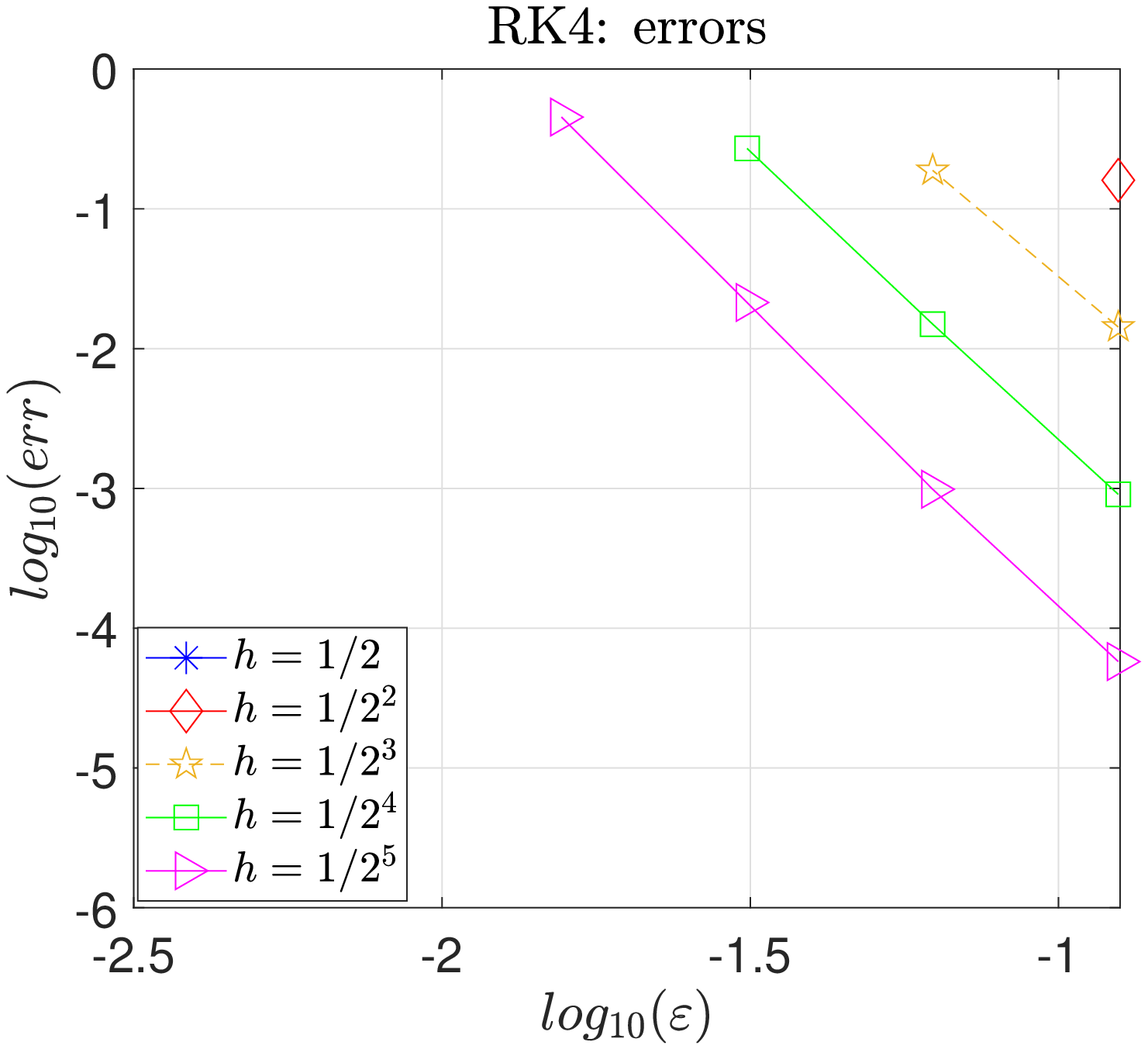,height=4.0cm,width=4.8cm}
\psfig{figure=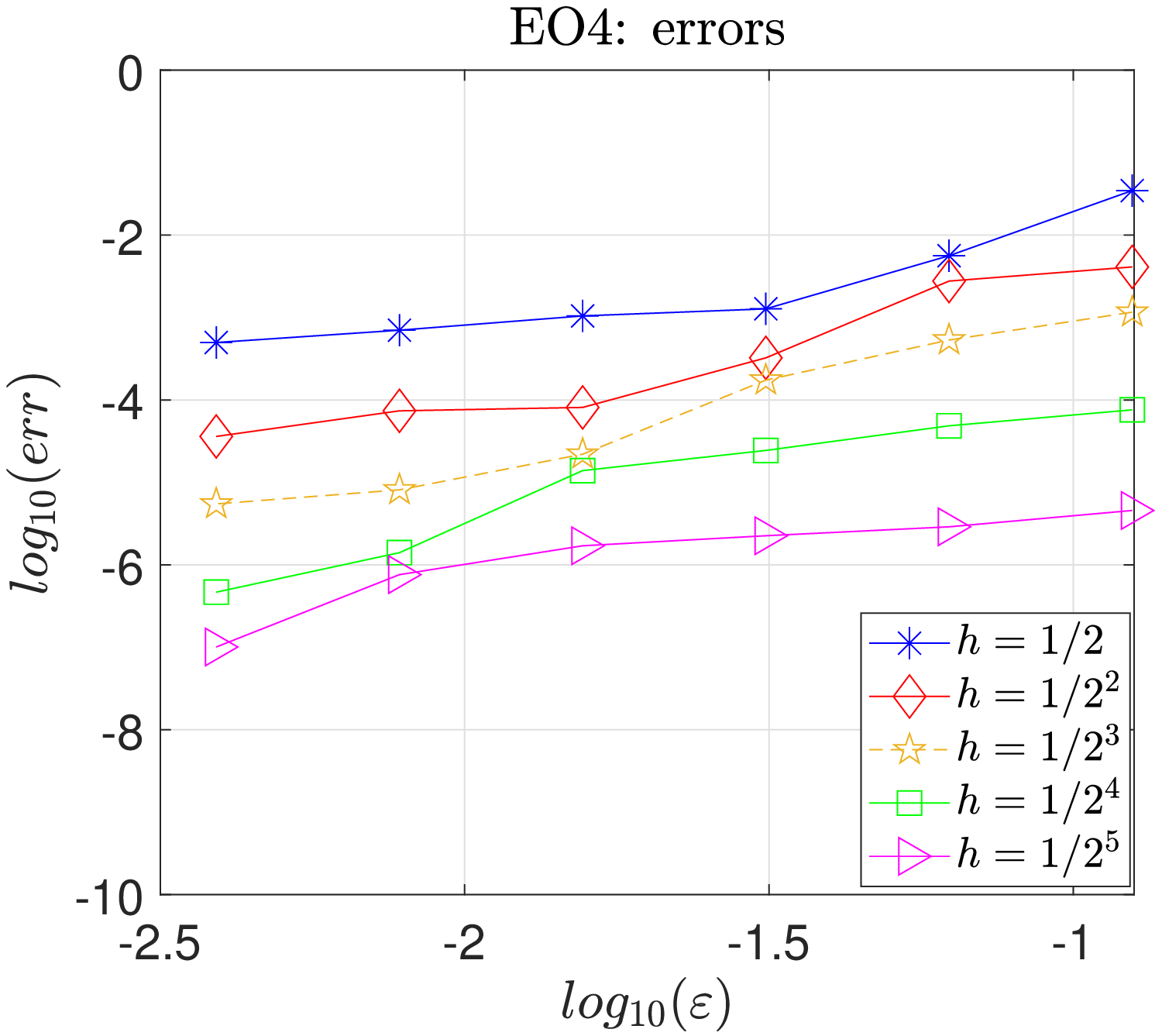,height=4.0cm,width=4.8cm}
\psfig{figure=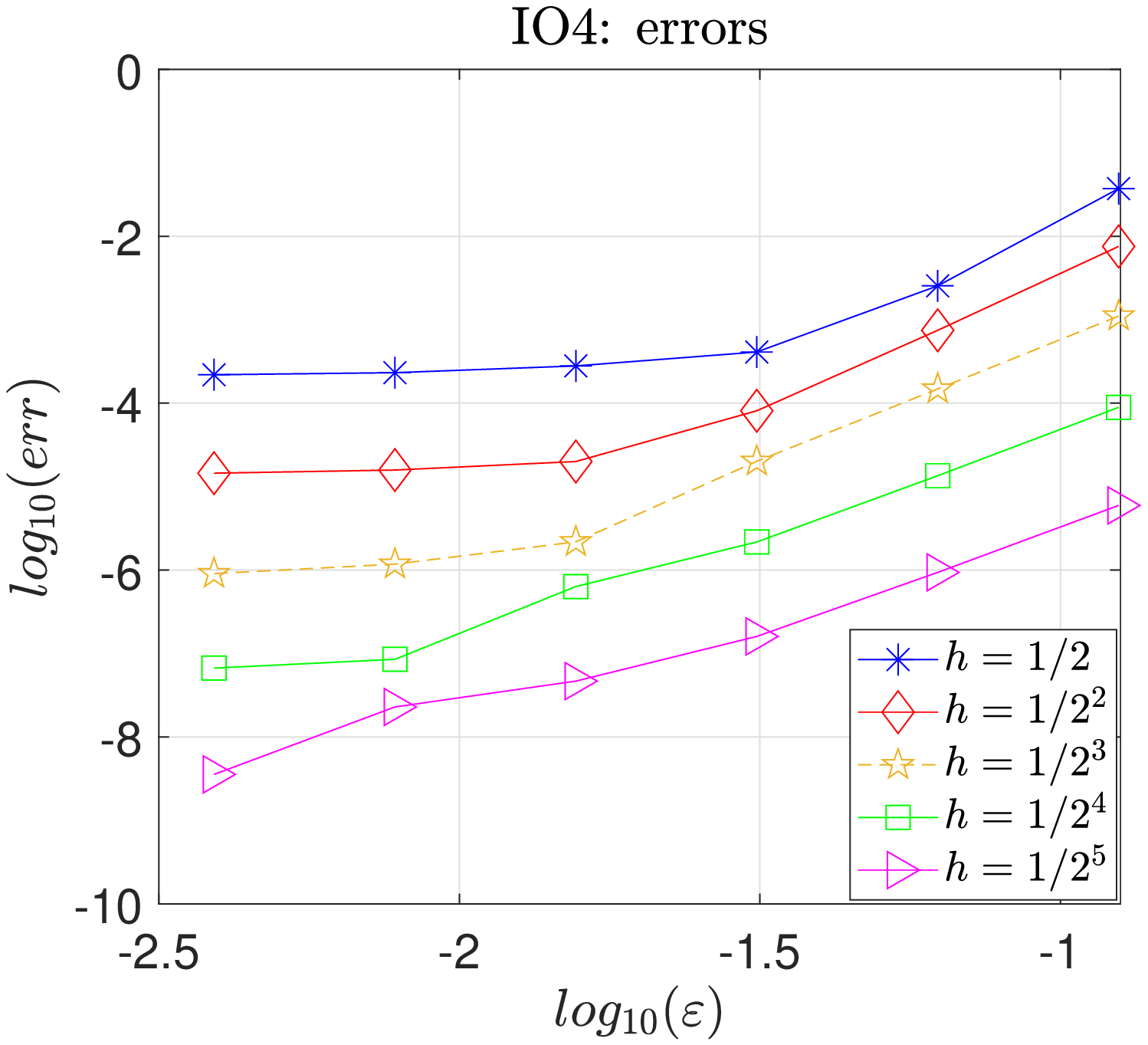,height=4.0cm,width=4.8cm}
\end{array}$$
\caption{The errors (\ref{err})  at $t=1$ of the second-order schemes (top row) and fourth-order schemes (below row)  with $\eps=1/2^k$ for $k=3,4,\ldots,8$ under different $h$.}\label{fign2}
\end{figure}

\section{Conclusion}\label{sec:con}
In this paper, we formulated and studied the numerical solution of the charged-particle dynamics (CPD) in a strong nonuniform magnetic field.
The system involves a small parameter $0<\eps\ll 1 $ inversely proportional to the  strength of the external magnetic field. Firstly, a novel class of semi-discretization  and full-discretization was presented for the  two dimensional CPD  and an optimal accuracy was rigorously derived. It was shown that the accuracy of those discretizations is improved  in the  position and  in the velocity when   $ \eps  $  becomes smaller. Then based on the  approach given  for the two dimensional case,   a kind of uniformly accurate methods with simple scheme was formulated for the three dimensional CPD in maximal ordering case. The optimal accuracy of the obtained discretizations was illustrated by some numerical tests.

Finally, it is remarked that higher-order algorithms with optimal accuracy would be an  issue for future exploration.
Another object of future study could be the complete convergence
analysis of the discretizations introduced in this paper  combining the PIC  approximation for Vlasov equations.

\section*{Acknowledgements}
This work was supported by NSFC 11871393 and by International Science and Technology Cooperation Program of Shaanxi Key Research \& Development Plan 2019KWZ-08.

\end{document}